\documentclass[9pt,leqno,a4paper]{article}
\usepackage{xcolor}
\usepackage{amsfonts}
\usepackage{amsmath, amsfonts, amsthm, amssymb, amscd}
\usepackage[mathcal]{euscript}
\usepackage{mathrsfs}
\usepackage{dsfont}
\usepackage{ulem}

\usepackage{slashed}

\newtheorem{theorem}{Theorem}[section]
\newtheorem{proposition}[theorem]{Proposition}
\newtheorem{lemma}[theorem]{Lemma}

\newtheorem{definition}[theorem]{Definition}
\newtheorem{remark}[theorem]{Remark}

\numberwithin{equation}{section}

\usepackage{a4wide}
\usepackage{multirow}
\usepackage{tabularx}
\usepackage{lscape}

\newcommand \Kcal {\mathcal K}
\newcommand \Hcal {\mathcal H}

\newcommand \Ecal {\mathcal E}
\newcommand \Econ {\mathcal{E}_{\text{con}}}
\newcommand \Fcon {\mathcal{F}_{\text{con}}}

\newcommand \xb {\bar {x}}
\newcommand \delb {\bar {\del}}
\newcommand \gb {\bar g}
\newcommand \hb{\bar h}

\newcommand \Tb {\overline {T}}
\newcommand \Phib{\overline{\Phi}}
\newcommand \Psib{\overline{\Psi}}

\newcommand \mb{\overline m}

\newcommand \Boxt {\widetilde {\Box}}
\newcommand \del \partial
\newcommand \delu {\uline{\del}}
\newcommand \Bu {\uline{B}}
\newcommand \Tu {\uline{T}}
\newcommand \Hu {\uline{H}}
\newcommand \hu {\uline{h}}

\newcommand \minu{\uline{m}}
\newcommand \Au{\uline{A}}
\newcommand \Pu{\uline{P}}
\newcommand \Psiu{\uline{\Psi}}
\newcommand \Phiu{\uline{\Phi}}

\newcommand \Ec {E_{\text{con}}}
\newcommand \Fc {{F_\text{con}}}
\newcommand \ec {e_{\text{con}}}
\newcommand \ebc {\bar{e}_{\text{con}}}

\newcommand \RR{\mathbb{R}}

\newcommand {\vep}{\varepsilon}

\newcommand {\seq}{\overset{*}{=}}

\newcommand {\dels}{\slashed{\del}}

\newcommand {\verified}{\marginpar{Verified}}

\makeatletter
\def\hlinew#1{%
  \noalign{\ifnum0=`}\fi\hrule \@height #1 \futurelet
   \reserved@a\@xhline}
\makeatother

\title{Global solutions of nonlinear wave-Klein-Gordon system in two spatial dimensions: weak coupling case \footnote{The present work belongs to a research project ``Global stability of quasilinear wave-Klein-Gordon system in $2 + 1$ space-time dimension'' (11601414), supported by NSFC.}}

\author{Yue MA \footnote{School of Mathematics and Statistics, Xi'an Jiaotong University, Xi'an, Shaanxi 710049, P.R. China.\ E-mail: yuemath@xjtu.edu.cn}}

\begin{document}

\maketitle



\section{Introduction}
\subsection{Objective}
In the present work and its successor \cite{M5} we will give a systematic investigation on the quadratic nonlinearities  coupled  in diagonalized wave-Klein-Gordon system in two spatial dimensions. More precisely, we will regard the following system:
\begin{equation}\label{eq 1 main}
\left\{
\aligned
&\Box u = F_1(u,\del u,\del\del u, v,\del v),
\\
&\Box v + c^2v = F_2(u,\del u,v,\del v, \del\del v),
\\
&u|_{t=2} = u_0,\quad \del_t u|_{t=2} = u_1,\quad v|_{t=2} = v_0,\quad \del_t v|_{t=2} = v_1.
\endaligned
\right.
\end{equation}
Here $F_i$ are linear with respect to $\del\del u$ or $\del\del v$ respectively and quadratic with respect to the rest arguments.  The fact that $F_1$ and $F_2$ do not contain $\del\del v$ and $\del\del u$ respectively is due to the quasilinearity and the fact that the system is diagonalized. When the initial data are sufficiently regular and being small, this Cauchy problem has unique local solution in classical sense, i.e., all derivatives appear in the equations are continuous.

The main objective of this work is to understand when the initial data is sufficiently regular and small in Sobolev norm, i.e.,
\begin{equation}\label{eq 2 main}
\|u_0\|_{H^{N+1}} + \|v_0\|_{H^{N+1}} + \|u_1\|_{H^{N}} + \|v\|_{H^{N}}\leq \vep,\quad N\in\mathbb{N}\text{ sufficiently large,}
\end{equation}
will the local solution extends to time infinity? And when this is true, what is its asymptotic behavior? 

The interest of regarding such problem is two-fold. 

First, we are encouraged by \cite{Hu2018} where the Einstein vacuum equation in $3+1$ space-time with a translation space-like
Killing field is reduced to a $2+1$ dimensional quasilinear wave system. Then it is natural to consider what will happen if the $3+1$ Einstein equation is coupled with a self-gravitating massive scalar field. Similar formulation leads to a $2+1$ dimensional wave-Klein-Gordon system (to be written as W-KG system in the follows), which contains the essential quasi-null structure of Einstein equation. However, since the decay of both wave and Klein-Gordon equations in $2+1$ dimension is weaker than in $3+1$ case, the analysis on this system, compared with our previous work \cite{LM2}, \cite{LM3} (see also \cite{Q.Wang2016}, \cite{Ionescu-Pausader}) in $3+1$ case, will be much more delicate. This article and its successor can be considered as technical preparations, in which we will regard \eqref{eq 1 main} as a model and concentrate firstly on the nonlinear terms which do not concern the quasi-null structure and/or (generalized-)wave gauge conditions enjoyed by Einstein-scalar system. Compared with our previous work \cite{M2}, in \eqref{eq 1 main} we will show how to treat the inevitable semi-linear terms on metric components and Klein-Gordon scalar (modeled as $u$ and $v$ respectively) which are (counterintuitively) much more difficult than the quasilinear terms treated in \cite{M2} (for alternative approach to these nonlinear terms, see \cite{Stingo-2018}).

A second interest comes form \eqref{eq 1 main} it-self. The research on global behavior of quasilinear wave equation/system has attracted a lot of attention of the mathematical community. Since the dimension is higher, the decay rates of both linear wave and linear KG equation are stronger, the problem of global existence of small regular solution  becomes trivial when dimension is sufficiently large. 

In dimension $3+1$, \cite{Kl1} established the global existence for wave equation with null quadratic nonlinearities (see also \cite{Christodoulou-1986}), \cite{Kl2} established the global existence for Klein-Gordon equation with arbitrary quadratic nonlinearities (see also \cite{Shatah85}). 

For W-KG system of the form \eqref{eq 1 main}, we have established its global existence in \cite{LM1} for quadratic nonlinearities satisfying the so-called ``minimal null condition'', i.e. we only demand null conditions on quadratic terms of wave components coupled in wave equation. This demand is ``minimal'' in the sens that, in the wave equation of \eqref{eq 1 main} if we take $v\equiv 0$,  it reduces to a quasilinear wave equation treated in \cite{Kl1}. The ``minimal null condition'' is the minimal demand such that the wave equation of \eqref{eq 1 main} reduces to the case of \cite{Kl1}.

In dimension $2+1$ the situation becomes more complicated. For wave equation, \cite{A1} and \cite{A2} gave a complete description on quasilinear quadratic terms. The semi-linear terms, being counterintuitively more difficult, is treated in \cite{Hoshiga-2006} several years latter (the techniques in \cite{Godin-1993} works only in the case of single equation). For Klein-Gordon equation, \cite{Dfx} combined the normal form transform developed in \cite{Shatah85} and the vector field method from \cite{Kl2} and established the global existence for arbitrary quadratic nonlinearities in the case of single equation and ``non-mass-resonance'' system.  Then \cite{KS-2011} regarded the case with mass-resonance.

For W-KG system, we naturally demand whether it is possible to obtain analogue result as in dimension three as we have done in \cite{LM1}, but this is far from trivial due to the lack of decay. In the present work and its successor, we will develop techniques aimed at the following question: in \eqref{eq 1 main}, which are the nonlinearities permitted in order to maintain the global existence?

\subsection{Structure of the system and main results}
In general, the existence and asymptotic behavior of the global solution depends on two factors, the structure of nonlinear terms and the profile of initial data. In this work we are mainly interested in the former one, therefor the initial data are supposed to be compactly supported in unique disc (in the following discussion, this property is often called {\it localized}). Furthermore, as we are discussing small amplitude solution, the first step is to consider quadratic nonlinearities. (However, in contrast to the $\RR^{3+1}$ case, where all cubic terms lead to global existence, there are cubic terms leads to finite time blow-up, see \cite{A2} for pure wave case.)  So $F_i$ is taken to be quadratic with constant coefficients. Now let us write the general form of $F_i$:
\begin{subequations}
\begin{equation}\label{eq 1 F1}
F_1 = \mathcal{P}_w^{\alpha\beta}(\del u, u,\del v,v)\del_{\alpha}\del_{\beta}u + \mathcal{A}_w^{\alpha}(\del u,\boxed{u},\del v,v) \del_{\alpha}u  
+\boxed{\mathcal{D}_w(u,\del v,v)u}
+ \uline{\mathcal{B}_w^{\alpha}(\del v,v)\del_{\beta} v} + \uline{K_1v^2}
\end{equation}
where
\begin{equation}\label{eq 2 F1}
\aligned
&\mathcal{P}_w^{\alpha\beta}(\del u, u, \del v, v) =  P_1^{\alpha\beta\gamma}\del_{\gamma}u + P_2^{\alpha\beta}u  + P_3^{\alpha\beta\gamma}\del_{\gamma}v+  P_4^{\alpha\beta}v,
\\
&\mathcal{A}_w^{\alpha}(\del u,\boxed{u},\del v,v) = A_1^{\alpha\beta}\del_{\beta}u  +\boxed{ A_2^{\alpha}u} +  A_3^{\alpha\beta}\del_{\beta}v + A_4^{\alpha}v,
\\
&\boxed{\mathcal{D}_w(u,\del v,v) = D_1u + D_2^{\alpha}\del_{\alpha}v + D_3v},
\\
&\uline{\mathcal{B}_w^{\alpha}(\del v,v)} = \uline{B_1^{\alpha\beta}\del_{\beta}v} + \uline{B_2^{\alpha}v} 
\endaligned
\end{equation}
and
\begin{equation}\label{eq 1 F2}
F_2 = \mathcal{P}_{kg}^{\alpha\beta}(\del u,u,\del v,v)\del_{\alpha}\del_{\beta}v + \mathcal{A}_{kg}^{\alpha}(\del u, u,\del v,v)\del_{\alpha}u 
+\boxed{\mathcal{D}_{kg}(u,\del v,v)u}
+ \mathcal{B}_{kg}^{\alpha}(\del v,v)\del_{\alpha} v + K_2v^2
\end{equation}
where
\begin{equation}\label{eq 2 F2}
\aligned
&\mathcal{P}^{\alpha\beta}_{kg}(\del u, \uwave{u}, \del v,v) = P_5^{\alpha\beta\gamma}\del_{\gamma}u +\uwave{ P_6^{\alpha\beta}u} + P_7^{\alpha\beta\gamma}\del_{\gamma}v + P_8^{\alpha\beta}v,
\\
&\mathcal{A}_{kg}^{\alpha}(\del u,u,\del v,v) = A_5^{\alpha\beta}\del_{\beta}u + A_6^{\alpha}u +  A_7^{\alpha\beta}\del_{\beta}v + A_8^{\alpha}v,
\\
&\boxed{\mathcal{D}_{kg}(u,\del v,v) = D_5u + D_6^{\alpha}\del_{\alpha}v + D_7v},
\\
&\mathcal{B}_{kg}^{\alpha}(\del v,v) = B_3^{\alpha\beta}\del_{\beta}v + B_4^{\alpha}v
\endaligned
\end{equation}
where all coefficients are supposed to be constants except $A_6^{\alpha}$. 
\end{subequations}

For further application in Einstein-Scalar system, $A_6^{\alpha}\del_{\alpha}u$ is supposed to be a linear combination of the following derivatives with  homogeneous coefficients of degree zero (Definition \ref{func-homo}) :
$$
(s/t)^2\del_t u,\quad (x^a/t)\del_t u + \del_a u,\quad  \del_t u + (x^a/t)\del_a u.
$$ 

The \boxed{boxed} terms will not be considered (suppose to be zero). The reason is that these terms do not appear in Einstein-scalar system.

The \uwave{wavy underlined} terms $P_6$ do appear in Einstein-scalar system, however, their treatment demands a detailed construction and investigation on the gauge conditions and quasi-null structure of Einstein equation which is not the purpose of this article. Moreover, in Einstein-scalar system, the terms $A_3$ and $A_4$ vanishes. This will give not a little convenience when we estimate $P_6$. So in this work this term is supposed to be zero.

The \uline{underlined terms} coupled in the wave equation is called {\sl strong coupling terms}. They change significantly the asymptotic behavior of the global solution. More precisely, when these terms disappear,
\begin{equation}\label{decay weak-coupling}
u\sim (1+|t-r|)^{-1/2+\delta/2}t^{-1/2+\delta/2},\quad \del_{\alpha} u \sim (1+|t-r|)^{-3/2+\delta/2}t^{-1/2+\delta/2}
\end{equation}
while when they appear, we can only obtain
\begin{equation}\label{decay strong-coupling}
u\sim (1+|t-r|)^{1/2+\delta/2}t^{-1/2+\delta/2},\quad\del_{\alpha}u \sim  (1+|t-r|)^{-1/2+\delta/2}t^{-1/2+\delta/2}.
\end{equation}

The system is said to be in strong coupling case, if it contains these strong coupling terms. Otherwise the system is said to be in weak coupling case. 

The reason why we distinguish between weak and strong coupling case is also two-fold. Firstly, it is important to understand the dependence of asymptotic behavior on nonlinear structure, and secondly and most importantly, in Einstein-massive-scalar system, some components of the metric is weakly coupled to the scalar field while the rests are in strong coupling. It is necessary to make a hierarchy between them.

Then we recall the standard null condition. A multi-linear form defined on $\RR^{2+1}$ is said to be {\bf null}, if it vanishes on the light-cone 
$$
\{(\xi_0,\xi_1,\xi_2)|\xi_0^2 = \xi_1^2 + \xi_2^2\}.
$$ 
For example, let $A^{\alpha\beta}$ be a quadratic form and $Q^{\alpha\beta\gamma}$ be a cubic form, then $A$ and $Q$ are said to be null, if 
$$
A^{\alpha\beta}\xi_{\alpha}\xi_{\beta} = Q^{\alpha\beta\gamma}\xi_{\alpha}\xi_{\beta}\xi_{\gamma} = 0,\quad \forall\ \xi_0^2 = \xi_1^2 + \xi_2^2.
$$

After these assumptions, we fist state the main result to be established in this article for the weak coupling case:
\begin{theorem}\label{thm 1 weak-coupling}
	Suppose that in  \eqref{eq 1 F1}
	\begin{equation}\label{eq1 thm main} 
	B_1^{\alpha\beta} = B_2^{\alpha} = 0,\quad\text {(weak coupling condition)}
	\end{equation}
	and
\begin{equation}\label{list null terms}
\aligned
\\
&P_1^{\alpha\beta\gamma}, P_2^{\alpha\beta},P_3^{\alpha\beta\gamma},P_5^{\alpha\beta\gamma},
\\
&A_1^{\alpha\beta}, A_3^{\alpha\beta}, A_5^{\alpha\beta},A_7^{\alpha\beta}
\endaligned
\end{equation}
	being null multi-linear forms. Suppose that the initial data are supported in the unit disc $\{|r|<1\}$.  Then there exists a positive constant $\vep_0$ determined by \eqref{eq 1 main}, such that when \eqref{eq 2 main} is satisfied with $0<\vep\leq \vep_0$ and $N\geq 15$, the associated local solution extends to time infinity. Furthermore,  \eqref{decay weak-coupling} holds.
\end{theorem}
\begin{remark}
	This result can be generalized without any essential improvement to the system where $u$ and $v$ are vectors. 
\end{remark}

In \cite{M5} the {\bf Strong coupling case} will be discussed, we will show that when
\begin{equation}
A_5^{\alpha\beta} = A_6^{\alpha} = 0
\end{equation}
and the terms in \eqref{list null terms} are null. the solution associated to small localized regular initial data extends to time infinity.

\subsection{Structure of this article}
This article is composed by two parts. 

In the first part (from Section \ref{sec conformal-energy} to Section \ref{sec tec-last} and the Appendix), we recall the conformal energy identity on hyperboloids (Section \ref{sec conformal-energy}) and the normal form transform on Klein-Gordon equation (Section \ref{sec normal-form transform}), then the basic notion of hyperboloidal foliation method are recalled in Section \ref{sec basic-hyper} and Appendix. Sections \ref{sec bounds-linear terms} to Section \ref{sec tec-last} are devoted to divers estimates based on the previous sections.

In the second part which only contains Section \ref{sec bootstrap}, we apply the bootstrap argument combined with the techniques developed in previous sections in order to prove the main result. 

\vskip .3cm
\centerline{\bf Acknowledgment}
The author is grateful to Pr. J-M. Delort for his proposal of this research topic. The author would also like to thank Dr. A. Stingo for useful discussions and comments. 

\section{Conformal energy estimate on hyperboloids}\label{sec conformal-energy}
In this section we recall the conformal energy estimates on hyperboloid which is introduced in \cite{MH-2017} for $3+1$ dimensional case (see also \cite{Wong-2017} where it is named as ``K-energy''). In this section we only show the this estimate within flat back-ground metic (i.e.,Minkowski). The estimate in curved back-ground metric is postponed in subsection \ref{subsec conforml-bounds} once we have recalled necessary notation in subsection \ref{subsec conforml-bounds} and \ref{subsec linear-bounds}. 

\subsection{Basic notation}
We are working in $\RR^{2+1}$ equipped with the Minkowski metric. We denote by $(t,x)=(x^0,x)$ with $x\in\RR^2$ a point in $\RR^{2+1}$ with $x= (x^1,x^2)$. We also use $r = |x| = \sqrt{|x^1|^2+|x^2|^2}$ for the Euclidean norm of $x$. We denote by
$$
s = \sqrt{t^2-r^2}
$$
the Minkowski distance from a point $(t,x)$ to the origin. We denote by
$$
\aligned
\Kcal :=& \{t>r+1\},\quad \text{The translated light-cone where we work}.
\\
\Hcal_s := & \left\{t=(s^2+r^2)^{1/2}\right\},\quad \text{The upper-nap of the hyperboloid with hyperbolic radius }s.
\\
\Hcal^*_s :=& \Hcal_s\cap \Kcal,\quad\text{The part of }\Hcal_s\text{ contained in }\Kcal,\quad \Hcal^*_s = \{(t,x)\in\Hcal_s: |x|\leq (s^2-1)/2\}.
\\
\Kcal_{[s_0,s_1]}:=& \left\{(t,x)\in\Kcal:(s_0^2+r^2)^{1/2} \leq t\leq (s_1^2+r^2)^{1/2}\right\},\quad\text{The part of }\Kcal\text{ limited by two hyperboloids}.
\endaligned
$$

Let $u$ be a function defined in $\Kcal_{[s_0,s_1]}$, vanishes near the conical boundary $\del\Kcal_{[s_0,s_1]}$. For $s_0\leq s\leq s_1$, we define its restriction on $\Hcal_s$ as
$$
u_s(x) := u\left((t^2+r^2)^{1/2},x\right),\quad \|u\|_{L^2(\Hcal_s)} = \|u_s\|_{L^2(\RR^2)}.
$$
Then we recall the following energies defined on hyperboloids. Firstly, the standard hyperbolic energy (or alternative energy in \cite{Ho1}):
$$
E_c(s,w) = \int_{\Hcal^*_s}(s/t)^2|\del_t w|^2 + \sum_a|\delu_a w|^2 + c^2w^2\ dx.
$$
Then the conformal energy 
$$
\Ec(s,u) = \int_{\Hcal^*_s}\Big((Ku+u)^2 + \sum_a|s\delb_au|^2\Big)\ dx.
$$
We also introduce the following ``high-order'' energies:
\begin{equation}\label{eq 3 01-01-2019}
\Ecal_c^N(s,w) := \sum_{|I|+|J|\leq N}E_c(s,w),\quad \Ecal^N(s,w):= \sum_{|I|+|J|\leq N}E_0(s,w),
\end{equation}
\begin{equation}\label{eq 4 01-01-2019}
\Econ^N(s,u): = \sum_{|I|+|J|\leq N} \Ec(s,u).
\end{equation}

\subsection{Frames adapted to hyperboloidal foliation}
In the  future cone $\Kcal$, we introduce the change of variables
\begin{equation}\label{Hyper vairables}
\aligned
&\xb^0 = s: = \sqrt{t^2 - r^2},
\qquad
\xb^a = x^a,
\endaligned
\end{equation}
together with the corresponding natural frame
\begin{equation}\label{Hyper frame}
\aligned
&\delb_0 := \del_s = \frac{s}{t}\del_t  = \frac{\sqrt{t^2-r^2}}{t}\del_t,
\\
&\delb_a := \del_{\xb^a} = \frac{\xb^a}{t}\del_t + \del_a = \frac{x^a}{t}\del_t + \del_a,
\endaligned
\end{equation}
which we refer to as the {\sl hyperbolic frame.} The transition matrices between the hyperbolic frame and the Cartesian frame are
$$
\big(\Phib^{\beta}_\alpha\big)
=
\big({\Phib^{\beta}}_\alpha\big)
= \left(
\begin{array}{ccc}
s/t &0 &0
\\
x^1/t &1 &0
\\
x^2/t &0 &1
\end{array}
\right),
\quad 
\big(\Phib^{\beta}_\alpha\big)^{-1}
= \big(\Psib^{\beta}_\alpha\big) = \big({\Psib^{\beta}}_\alpha\big)
= \left(
\begin{array}{ccc}
t/s &0 &0
\\
-x^1/s &1 &0
\\
-x^2/s &0 &1
\end{array}
\right),
$$
so that
$
\delb_\alpha = \Phib^{\beta}_\alpha\del_\beta
$
and
$
\del_\alpha = \Psib^{\beta}_\alpha\delb_\beta.
$

The dual hyperbolic frame then reads
$
d\xb^0 := ds = \frac{t}{s}dt - \frac{x^a}{s}dx^a$ and
$d\xb^a := dx^a$.
The Minkowski metric in the hyperbolic frame reads\footnote{Our sign convention is opposite to the one in our monograph \cite{LM1}, since the metric here has signature $(-, +, +, +)$.}
$$
\mb^{\alpha\beta} = \left(
\begin{array}{ccc}
1 &x^1/s &x^2/s
\\
x^1/s &-1 &0
\\
x^2/s &0 &-1
\end{array}
\right).
$$

For a two tensor $T^{\alpha\beta}\del_{\alpha}\otimes\del_{\beta}$, we write $\Tb^{\alpha\beta}$ for its components within hyperbolic frame: 
$$
T^{\alpha\beta}\del_{\alpha}\otimes\del_{\beta} = \Tb^{\alpha\beta}\delb_{\alpha}\otimes \delb_{\beta}.
$$
The transition relations are written as:
$$
\Tb^{\alpha\beta} = T^{\alpha'\beta'}\Psib_{\alpha'}^{\alpha}\Psib_{\beta'}^{\beta}.
$$ 

We also recall the semi-hyperboloidal frame which is introduced in \cite{LM1}. In $\Kcal$,
$$
\delu_0:=\del_t,\quad \delu_a := \delb_a = (x^a/t)\del_t + \del_a.
$$
The transition matrices between this frame and the natural frame $\{\del_{\alpha}\}$ is:
\begin{equation}\label{eq semi-frame}
\Phiu_{\alpha}^{\beta} := \left(
\begin{array}{ccc}
1 &0 &0
\\
x^1/t &1 &0
\\
x^2/t &0 &1 
\end{array}
\right),
\quad
\Psiu_{\alpha}^{\beta} := \left(
\begin{array}{ccc}
1 &0 &0
\\
-x^1/t &1 &0
\\
-x^2/t &0 &1 
\end{array}
\right)
\end{equation}
with
$$
\delu_{\alpha} = \Phiu_{\alpha}^{\beta}\del_{\beta},\quad \del_{\alpha} = \Psiu_{\alpha}^{\beta}\delu_{\beta}.
$$

Let $T = T^{\alpha\beta}\del_{\alpha}\otimes\del_{\beta}$ be a two tensor defined in $\Kcal$ or its subset. Then $T$ can be written with $\{\delu_{\alpha}\}$:
$$
T = \Tu^{\alpha\beta} \delu_{\alpha}\otimes\delu_{\beta}
\quad\text{with}\quad 
\Tu^{\alpha\beta} = T^{\alpha'\beta'}\Psiu_{\alpha'}^{\alpha}\Psiu_{\beta'}^{\beta}.
$$

The dual frame of $\{\delu_{\alpha}\}$ is
$$
\theta^0 = dt - \sum_a(x^a/t)dx^a, \quad \theta^a = dx^a.
$$
We calculate the Minkowski metric in this frame:
$$
\minu^{\alpha\beta} = 
\left(
\begin{array}{ccc}
(s/t)^2 &x^1/t &x^2/t
\\
x^1/t &-1 &0
\\
x^2/t &0 &-1 
\end{array}
\right),
\quad
\minu_{\alpha\beta} = 
\left(
\begin{array}{ccc}
1 &x^1/t &x^2/t
\\
x^1/t &(x^1/t)^2-1 &x^1x^2/t^2
\\
x^2/t &x^2x^1/t^2 &(x^2/t)^2- 1
\end{array}
\right).
$$

For a quadratic form $T$ acting on $(\del u, \del v)$ as $T(\del u,\del v) = T^{\alpha\beta}\del_{\alpha}u\del_{\beta}v$,
we denote by
$$
\Tu(\del u, \del v) := \sum_{(\alpha,\beta)\neq (0,0)}\Tu^{\alpha\beta}\delu_{\alpha}u\delu_{\beta}u.
$$

For a trilinear form acting on $(\del u,\del \del v)$:
$$
\aligned
H^{\alpha\beta\gamma}\del_{\gamma}u\del_{\alpha}\del_{\beta}v =& \Hu^{\alpha\beta\gamma}\delu_{\gamma}u\delu_{\alpha}\delu_{\beta}v 
+ H^{\alpha\beta\gamma}\del_{\gamma}u \del_{\alpha}\big(\Psiu_{\beta}^{\beta'}\big)\delu_{\beta'}v
\\
=&\Hu^{000}\delu_tu\delu_t\delu_tv + \Hu(\del u,\del \del v)
\endaligned
$$
where
\begin{equation}\label{eq trilinear-good}
\Hu^{\alpha\beta\gamma}(\del u,\del\del v) := \sum_{(\alpha,\beta,\gamma)\neq(0,0,0)}\Hu^{\alpha\beta\gamma}\delu_{\gamma}u\delu_{\alpha}\delu_{\beta}v + H^{\alpha\beta\gamma}\del_{\gamma}u \del_{\alpha}\big(\Psiu_{\beta}^{\beta'}\big)\delu_{\beta'}v.
\end{equation}
\begin{remark}
The main advantage of $\{\delb_\alpha\}$ is that $[\delb_\alpha,\delb_\beta] = 0$. However, it has the disadvantage that the transition matrices are singular on the cone $\{t=r\}$. The semi-hyperboloidal frame has the advantage that the transition matrices are homogeneous of degree zero.
\end{remark}

\subsection{Differential identity}
Let $g^{\alpha\beta}$ be a metric defined in $\Kcal_{[s_0,s_1]}$, sufficiently regular. Let $g^{\alpha\beta} = m^{\alpha\beta} + h^{\alpha\beta}$ with $m^{\alpha\beta}$ the standard Minkowski metric. The following differential identities is deduced from the decomposition of $g^{\alpha\beta}\del_{\alpha}\del_{\beta}$ within the hyperbolic frame (for details of calculation, see \cite{MH-2017}). 
\begin{equation}
\aligned
g^{\alpha\beta}\del_{\alpha}\del_{\beta}u = &s^{-1}\delb_s\left(s\mathscr{K}_gu\right) + \gb^{ab}\delb_a\delb_bu
\\
&-\delb_s\gb^{00}\delb_su - 2s^{-1}\left(\gb^{a0} + s\delb_s\gb^{a0}\right)\delb_au
+ \left( g^{\alpha\beta}\del_{\alpha}\left(\Psib_\beta^0\right) - s^{-1}\gb^{00}\right)\delb_su
\endaligned
\end{equation}
with
$$
\mathscr{K}_g = s\left(\gb^{00}\delb_s + 2\gb^{a0}\delb_a\right)
= \big(s\delb_s + 2x^a\delb_a\big)  + s\left(\hb^{00}\delb_s + 2\hb^{a0}\delb_a\right).
$$
This leads to 
\begin{equation}
\label{eq 1 21-05-2017}
\aligned
s(\mathscr{K}_gu + N_g u)\cdot g^{\alpha\beta}\del_{\alpha}\del_{\beta}u =&\frac{1}{2}\delb_s\left(|\mathscr{K}_gu + N_gu|^2  - s^2\gb^{00}\gb^{ab}\delb_au\delb_bu\right)
+ \delb_a(w_g^a)
\\
& + s^2R_g^{ab}\delb_au\delb_bu + (\mathscr{K}_g+N_g)u\cdot S_g[u]  + s\delb_bu\cdot T_g^b[u]
\endaligned
\end{equation}
with
\begin{equation}\label{eq 1-1 21-05-2017}
\aligned
N_g =& sg^{\alpha\beta}\del_{\alpha}\left(\Psib_\beta^0\right) - \delb_s(s\gb^{00}) = g^{00} - \sum_ag^{aa} - 2\gb^{00} - s\delb_s\gb^{00}
\\
=&h^{00} - \sum_ah^{aa} - 2\hb^{00} - s\delb_s\hb^{00} + 1,
\endaligned
\end{equation}
$$
w_g^a = s\mathscr{K}_gu\cdot \gb^{ab}\delb_b u - s^2\gb^{a0}\gb^{cb}\delb_cu\delb_bu + N_gsu\cdot \gb^{ab}\delb_bu
$$
and
\begin{equation}\label{eq 1-2 21-05-2017}
\aligned
s^2R_g^{ab}\delb_au\delb_bu :=&s\big(L_g^{ab}-N_g\gb^{ab}\big)\delb_au\delb_bu 
+ \frac{s^2}{2}\delb_s\left(\hb^{00}\gb^{ab} + \hb^{ab}\right)\delb_au\delb_bu,
\endaligned
\end{equation}
where
$$
\aligned
L_g^{ab} : =& \gb^{00}\gb^{ab} + s\delb_c\left(\gb^{0c}\gb^{ab}\right) - 2s\delb_c\gb^{0a}\cdot\gb^{cb}
\endaligned
$$
and
\begin{equation}\label{eq 1-3 21-05-2017}
\aligned
\big(\mathscr{K}_g +N_g\big) u\cdot S_g[u]
:=& - (\mathscr{K}_g + N_g)u\cdot\left(2\delb_s(s\hb^{a0})\delb_au + s\delb_a\hb^{ab}\delb_b u + u\delb_sN_g\right),
\\
s\delb_bu\cdot T_g^b[u] :=&   - s\delb_b u\left( u\cdot \gb^{ab}\delb_aN_g + s\gb^{ab}\delb_a\hb^{00}\delb_su\right).
\endaligned
\end{equation}

Furthermore, we remark that
$$
L_g^{ab} = (\hb^{00}\gb^{ab} + \mb^{00}\hb^{ab}) + s\delb_c(\hb^{0c}\gb^{ab} + \mb^{0c}\hb^{ab}) - 2s\delb_c(\hb^{a0}\gb^{cb} + \mb^{0a}\hb^{cb}) + \mb^{ab}
$$
and
$$
N_g -1 = h^{00} - \sum_ah^{aa} - 2\hb^{00} - s\delb_s\hb^{00},
$$
So 
\begin{equation}\label{eq 1-4 21-05-2017}
\aligned
R_g^{ab} =& s^{-1}\big((L_g^{ab}-\mb^{ab}) + \mb^{ab}(1-N_g)\gb^{ab} - N_g\hb^{ab}\big)
+ \frac{1}{2}\delb_s\left(\hb^{00}\gb^{ab} + \hb^{ab}\right).
\endaligned
\end{equation}
Remark that when $g^{\alpha\beta} = m^{\alpha\beta}$, $h^{\alpha\beta} = 0$ and
$$
\aligned
&K:=\mathscr{K}_m = s\delb_s + 2x^a\delb_a ,\quad N_m = 1,
\\
&L_m^{ab} = -\delta^{ab},\quad L_m^{ab} - N_m\mb^{ab} = 0.\quad R_g^{ab} = 0.
\endaligned
$$
This leads to
\begin{equation}\label{eq 1 25-06-2019}
R_m^{ab} = S_m[u] = T_g^b[u] = 0.
\end{equation}
Then \eqref{eq 1 21-05-2017} becomes
\begin{equation}\label{eq 1-flat 21-05-2017}
\aligned
s(Ku + u)\cdot\Box u =&\frac{1}{2}\delb_s\left(|Ku + u|^2  + s^2\sum_a|\delu_a u|^2\right)
+ \delb_a(w_m^a).
\endaligned
\end{equation}

\subsection{Conformal energy estimate within flat back-ground metric}
We first analyse the case where $g^{\alpha\beta} = m^{\alpha\beta}$ (i.e. the flat case). For the convenience of discussion, we recall 
\begin{equation}\label{eq conformal-energy}
\Ec(s,u) := \int_{\Hcal_s}\big(|Ku + u|^2 + \sum_a|s\delb_au|^2\big)\ dx.
\end{equation}
\begin{lemma}\label{lem conformal}
	Let $u$ be a function defined in $\Kcal_{[s_0,s_1]}$, sufficiently regular and vanishes near the conical boundary $\del\Kcal_{[s_0,s_1]}$. Then the following bound holds:
	\begin{equation}\label{eq1 lem conformal}
	\Ec(s_1,u)^{1/2} \leq \Ec(s_0,u)^{1/2} + \int_{s_0}^{s_1}s\|\Box u\|_{L^2(\Hcal_s)}ds.
	\end{equation}
\end{lemma}
\begin{proof}
	This is by integrating \eqref{eq 1-flat 21-05-2017} in $\Kcal_{[s_0,s_1]}$ and the Stokes formula:
	$$
	\int_{\Kcal_{[s_0,s_1]}}s(Ku + u)\cdot\Box u\ dxds
	= \frac{1}{2}\int_{s_1}\big(|Ku + u|^2 + \sum_a|s\delu_au|^2\big)dx - \frac{1}{2}\int_{s_0}\big(|Ku + u|^2 + \sum_a|s\delu_au|^2\big)dx.
	$$
	Differentiate with respect to $s$, we obtain
	$$
	\frac{d}{2ds}\int_{s_1}\big(|Ku + u|^2 + \sum_a|s\delu_au|^2\big)dx = \int_{\Hcal_s}s(Ku + u)\cdot\Box u\ dx
	$$
	which leads to
	$$
	\Ec(s,u)^{1/2}\frac{d}{ds}\Ec(s,u)^{1/2}\leq \|Ku + u\|_{L^2(\Hcal_s)}\|s\Box u\|_{L^2(\Hcal_s)}\leq \Ec(s,u)^{1/2}\|s\Box u\|_{L^2(\Hcal_s)}.
	$$
	Thus 
	$$
	\frac{d}{ds}\Ec(s,u)^{1/2}\leq\|s\Box u\|_{L^2(\Hcal_s)}.
	$$
	Integrate the above inequality on the interval $[s_0,s_1]$, the desired result is obtained.
\end{proof}

However, if we regard directly the energy $\Ec(s,u)^{1/2}$, it is not such satisfactory: it can neither control directly the gradient of $u$ nor the $L^2$ norm of $u$ itself. In fact, in $3D$ case we can prove that the $L^2$ norm of  $s(s/t)^2\del_tu$ and $(s/t)u$ can be controlled by the flat conformal energy  as we have done in \cite{MH-2017}, where the Hardy's inequality on hyperboloids is applied, which is valid only for dimension larger than or equal to three (see also \cite{Wong-2017}  where a weighted Hardy inequality is applied in 3D and 2D). Here in two dimensional case, we need other techniques.
\begin{lemma}\label{proposition 1 01-01-2019}
	Let $u$ be a $C^1$ function defined in $\Kcal_{[s_0,s_1]}$ and vanishes near $\del\Kcal$. Then
	\begin{equation}\label{eq 1 14-12-2018}
	\|(s/t)u\|_{L^2(\Hcal_{s_1})}\leq \|(s/t)u\|_{L^2(\Hcal_{s_0})} + C\int_{s_0}^{s_1}s^{-1}\Ec(s,u)^{1/2}ds.
	\end{equation}
\end{lemma}
\begin{proof}
	This relies on the following differential identity:
	\begin{equation}\label{eq 1 01-01-2019}
	\aligned
	(s/t)u\cdot(s/t)\left(Ku + u\right)
	=& \frac{1}{2}s\delb_s\big((s/t)^2u^2\big) + (s/t)u\cdot (x^a/t)s\delb_au + \frac{1}{2}\delb_a\big(x^a(s/t)^2u^2\big).
	\endaligned
	\end{equation}
	Integrate this on $\Hcal_s$ (remark that the restriction of $u$ on $\Hcal_s$ is supported in $\Hcal_s^*$), we obtain:
	$$
	\frac{s}{2}\frac{d}{ds}\int_{\Hcal_s}(s/t)^2u^2\ dx + \int_{\Hcal_s} (s/t)u\cdot (x^a/t)s\delb_au\ dx = \int_{\Hcal_s}(s/t)u\cdot(s/t)\left(Ku + u\right)\ dx
	$$
	This leads to
	$$
	\aligned
	\frac{d}{2ds}\|(s/t)u\|_{L^2(\Hcal_s)}^2
	\leq& Cs^{-1}\|(s/t)u\|_{L^2(\Hcal_s)}\cdot\big(\|Ku+u\|_{L^2(\Hcal_s)} + \sum_a\|s\delb_au\|_{L^2}\big)
	\\
	\leq& Cs^{-1}\|(s/t)u\|_{L^2(\Hcal_s)}\Ec(s,u)^{1/2}.
	\endaligned
	$$
	Thus
	$$
	\frac{d}{ds}\|(s/t)u\|_{L^2(\Hcal_s)}\leq Cs^{-1}\Ec(s,u)^{1/2}.
	$$
	Then integrate on time interval $[s_0,s_1]$, the desired result is established.
\end{proof}

For the convenience of discussion, we introduce the following notation:
$$
\Fc(s_0;s,u)^{1/2} := \|(s/t)u\|_{L^2(\Hcal_{s_0})} + \Ec(s,u)^{1/2} + \int_{s_0}^{s}s^{-1}\Ec(s',u)^{1/2}ds'.
$$
Then the following bound holds:
\begin{proposition}\label{proposition 2 01-01-2019}
	Let $u$ be a $C^1$ function defined in $\Kcal_{[s_0,s_1]}$ and vanishes near $\del\Kcal$. Then the following quantities :
	\begin{equation}\label{eq 2 01-01-2019}
	\|(s/t)u\|_{L^2(\Hcal_s^*)}, \quad \|s(s/t)^2\del_\alpha u\|_{L^2(\Hcal_s^*)}
	\end{equation}
	are bounded by $\Fc(s_0;s,u)^{1/2}$.
\end{proposition}
\section{Normal form transform: differential identities}\label{sec normal-form transform}
In this section we will begin to present a version of normal form transform adapted to our context. Roughly speaking, normal from transform is, instead of considering the original Klein-Gordon component $v$ which satisfies a nonlinear Klein-Gordon equation, regarding a carefully constructed nonlinear perturbation of $v$, who satisfies a much better equation (with more friendly nonlinear terms). 


The techniques that we will introduced in the follows is somehow ``overqualified'', i.e., in order to obtain the main result, the normal form transform introduced in \cite{M1} is sufficient. However, these techniques will be necessary in the analysis on Einstein-Scalar system.

Our construction of normal form transform is divided into tow steps. In this section we will only give the ``algebraic'' part, which contain only the differential identities. The construction of estimates will be postponed to section \ref{sec normal-form transform-bounds} after we introduce necessary notation and results in section \ref{sec basic-hyper}.

\subsection{Differential identities}\label{subsec normal-form-identity}
Suppose that 
\begin{equation}\label{eq 0 normal-form}
\Box v + c^2 v = f
\end{equation}
and
$$
w:=v + av\del_tv + bv^2
$$
with $f,a,b$ regular functions defined in $\Kcal_{[s_0,s_1]}$. Then direct calculation leads to
\begin{equation}\label{eq 1 normal-form}
\aligned
\Box w + c^2 v = 2a(s/t)^2\del_t v \del_t\del_t w + 2b(s/t)^2\del_tv\del_tv - 2c^2av\del_tv - 2c^2bv^2  + \mathscr{R}_1 + f.
\endaligned
\end{equation}
with
\begin{equation}\label{eq R_1 normal-form}
\aligned
\mathscr{R}_1 :=& 2a\minu(\del v,\del \del_t v) + 2b\minu(\del v,\del v)
\\
&+v\del_tv\ \Box a + 2\del_t vm^{\alpha\beta}\del_{\alpha}a\del_{\beta} v +  2vm^{\alpha\beta}\del_{\alpha}a\del_{\beta}\del_tv 
+ v^2\Box b + 4vm^{\alpha\beta}\del_{\alpha}b\del_{\beta}v 
\\
&+ af\del_tv + av\del_tf + 2bvf 
\\
& -2a(s/t)^2\del_tv\del_t\del_t(av\del_tv + bv^2)
\endaligned
\end{equation}
where we recall $\minu(\del v,\del v) = \sum_{(\alpha,\beta)\neq(0,0)}\minu^{\alpha\beta}\del_{\alpha}v\del_{\beta}v$.

We consider the following quasilinear Klein-Gordon equation:
\begin{equation}\label{eq starting normal-form}
\Box v + ( h_0^{\alpha\beta}v + h_1^{\alpha\beta\gamma}\del_{\gamma}v)\del_{\alpha}\del_{\beta}v + c^2v
= A^{\alpha\beta}\del_{\alpha}v\del_{\beta}v + B^{\alpha}v\del_{\alpha}v + Rv^2 + R_0,
\end{equation}
where $h_0,h_1, A,B,R$ are supposed to be constant-coefficient multi-linear forms. $R_0$ is a sufficiently regular function.

In \eqref{eq 1 normal-form}, taking  
\begin{equation}\label{eq f starting normal-form}
f = -  (h_0^{\alpha\beta}v + h_1^{\alpha\beta\gamma}\del_{\gamma}v)\del_{\alpha}\del_{\beta}v + A^{\alpha\beta}\del_{\alpha}v\del_{\beta}v + B^{\alpha}v\del_{\alpha}v + Rv^2 + R_0
\end{equation}
We write:
\begin{equation}\label{eq R_2 normal-form}
\aligned
f =& - \big( \hu_0^{00}v + \hu_1^{000}\del_tv\big)\del_t\del_tw + \Au^{00}\del_tv\del_tv + \Bu^0v\del_tv  + Rv^2 + R_0
\\
&\left.
\aligned
&-\sum_{(\alpha,\beta)\neq(0,0)}v\hu_0^{\alpha\beta}\delu_{\alpha}\delu_{\beta}w 
- vh^{\alpha\beta}\del_{\alpha}\big(\Psiu_{\beta}^{\beta'}\big)\delu_{\beta'}w
\\
&-\sum_{(\alpha,\beta,\gamma)\neq (0,0,0)}\hu_1^{\alpha\beta\gamma}\delu_{\gamma}v\delu_{\alpha}\delu_{\beta}w
- h^{\alpha\beta\gamma}\del_{\gamma}v\del_{\alpha}\big(\Psiu_{\beta}^{\beta'}\big)\delu_{\beta'}w
\\
& + \big( h_0^{\alpha\beta}v + h_1^{\alpha\beta\gamma}\del_{\gamma}v\big)\del_{\alpha}\del_{\beta}(av\del_tv + bv^2)
+ \Au(\del v,\del v) + v\Bu(\del v)
\endaligned
\right \} \mathscr{R}_2
\endaligned
\end{equation}
Then we obtain:
\begin{equation}\label{eq 2 normal-form}
\aligned
&\Box w  + \big(\hu_0^{00}v +  \hu_1^{000}\del_tv - 2a(s/t)^2\del_t v \big)\del_t\del_tw +  c^2v
\\
=&  \big(2b(s/t)^2 + \Au^{00}\big)\del_tv\del_tv  + (\Bu^0- 2c^2a)v\del_tv + (R- 2c^2b)v^2   +R_0 
+ \mathscr{R}_2 + \mathscr{R}_1.
\endaligned
\end{equation}
Write the D'Alembert operator within semi-hyperboloidal frame:
$$
\Box w = (s/t)^2\del_t\del_t w + \frac{2x^a}{t}\del_t\delu_a w -\sum_a\delu_a\delu_a w + t^{-1}\left(2+(r/t)^2\right)\del_tw.
$$
This leads to
\begin{equation}\label{eq 3 normal-form}
\aligned
&\left(1 + h[a,v]\right)(s/t)^2\del_t\del_tw + \left(\frac{2x^a}{t}\del_t\delu_a w -\sum_a\delu_a\delu_a w + t^{-1}\left(2+(r/t)^2\right)\del_tw\right)
 + c^2 v
\\
=& \big(2(s/t)^2b + \Au^{00}\big) \del_tv\del_tv  + (\Bu^0 - 2c^2a)v\del_tv + (R - 2c^2b)v^2 + R_0 +  \mathscr{R}_2 + \mathscr{R}_1
\endaligned
\end{equation}
with
\begin{equation}\label{eq h normal-form}
\aligned
h[a,v] :=& (t/s)^2\big(\hu_0^{00}v + \hu_1^{000}\del_tv\big) - 2a\del_tv.
\endaligned
\end{equation}
Suppose that 
\begin{equation}\label{eq 2.5 normal-form}
|h[a,v]|\leq 1/2,
\end{equation}
we divide \eqref{eq 3 normal-form} by $(1 +h[a,v])$ and obtain
\begin{equation}
\aligned
&(s/t)^2\del_tw\del_tw + \frac{2x^a}{t}\del_t\delu_a w -\sum_a\delu_a\delu_a w + t^{-1}\left(2+(r/t)^2\right)\del_tw 
+ c^2w
\\
=&\big(2(s/t)^2b + \Au^{00}\big) \del_tw\del_tw  + (\Bu^0 - c^2a)v\del_tv + (R - c^2b)v^2 + c^2h[a,v]v
\\
&+\mathscr{R}_3 +(1+h[a,v])^{-1}\left(R_0  + \mathscr{R}_2 + \mathscr{R}_1\right)
\endaligned
\end{equation}
with
\begin{equation}\label{eq R_3 normal-form}
\aligned
\mathscr{R}_3 =& \big(2(s/t)^2b + \Au^{00}\big)\big(|\del_t(av\del_tv+bv^2)|^2-2\del_tv\del_t(av\del_tv+bv^2)\big)
\\
&+\big(1-(1+h[a,v])^{-1} - h[a,v]\big)c^2v
\\
&+\big(1-(1+h[a,v])^{-1}\big)\left(\frac{2x^a}{t}\del_t\delu_a w -\sum_a\delu_a\delu_a w + t^{-1}\left(2+(r/t)^2\right)\del_tw \right)
\\
&+\big((1+h[a,v])^{-1}-1\big)\Big(\big(2b\minu^{00} + \Au^{00}\big) \del_tw\del_tw  + (\Bu^0 - 2c^2a)v\del_tv + (R - 2c^2b)v^2 \Big).
\endaligned
\end{equation}
So we obtain
\begin{equation}
\aligned
\Box w + c^2w
=&  \big(2(s/t)^2b + \Au^{00}\big) \del_tw\del_tw  + (R - c^2b +c^2(t/s)^2\hu_0^{00})v^2 
\\
&+ (\Bu^0 + c^2(t/s)^2\hu_1^{000}  - 3c^2a)v\del_tv 
+ \mathscr{R}
\endaligned
\end{equation}
where
\begin{equation}\label{eq R normal-form}
\aligned
\mathscr{R}=\mathscr{R}_3 +(1+h[a,v])^{-1}\left(R_0  + \mathscr{R}_2 + \mathscr{R}_1\right).
\endaligned
\end{equation}

Taking 
\begin{equation}\label{eq1 lem main normal-form}
a = \frac{1}{3c^2}\left(\Bu^0 + c^2(t/s)^2\hu_1^{000}\right),\quad 
b = \frac{1}{c^2}\left(R + c^2(t/s)^2\hu_0^{00}\right),
\end{equation}
we obtain
\begin{equation}\label{eq ending normal-form}
\Box w + c^2w = \left(2(s/t)^2c^{-2}R + 2\hu_0^{00} + \Au^{00}\right)\del_tw\del_tw + \mathscr{R}.
\end{equation}
Now we have eliminated all quadratic terms except $\del_tw\del_tw$.
\subsection{Modified energy identity}
The semi-linear term $\del_t v\del_t v$ is more difficult to handle. We need to modify the energy estimate.

Suppose that $v,w$ are sufficiently regular and satisfying 
$$
\Box v + c^2v = f,\quad \Box w + c^2 w = g
$$
in $\Kcal_{[s_0,s_1]}$. Direct calculation shows that 
\begin{equation}\label{eq1 normal-form-energy}
\aligned
\Box (A vw) +c^2A vw =& 2(s/t)^2A \del_tv\del_t w - c^2A vw + A\big(\minu(\del v,\del w) +  vg + wf\big)
\\
& + vw\Box A + 2m^{\alpha\beta}\del_{\alpha}A\del_{\beta}(vw).
\endaligned
\end{equation}
where $A$ is a regular function.

Next, let $\omega$ be a function defined in $\Kcal_{[s_0,s_1]}$, sufficiently regular. Then
\begin{equation}\label{eq2 normal-form-energy}
\aligned
&\omega \del_tw(\Box v+c^2v) + \omega\del_t v(\Box w + c^2w)
\\
=& \del_t\left(\omega\big(\del_tw\del_tv + \sum_a\del_aw\del_av + c^2vw\big)\right) - \del_a\left(\omega\big(\del_tw\del_av + \del_tv\del_a w\big)\right)
\\
&- (s/t)^2\del_t\omega \del_tw\del_tv - c^2vw\del_t\omega
\\
&-\del_t\omega \sum_a\delu_aw\delu_av + (x^a/t)\del_t\omega(\delu_aw\del_tv + \delu_av\del_tw) 
\\
&+\sum_a \del_a\omega(\del_tv\delu_aw + \del_tw\delu_av)
-2(x^a/t)\delu_a\omega\del_tv\del_tw
\endaligned
\end{equation}

We consider the system
\begin{equation}\label{eq-main normal-form-energy}
\Box v_i + c^2 v_i = F_i,\quad F_i =  Q_i^{jk}\del_tv_j\del_tv_k + R_i, \quad i=1,2,\cdots N,\quad Q_i^{jk} = Q_i^{kj}
\end{equation}
with $Q_i^{jk}$ and $R_i$ defined in $\Kcal_{[s_0,s_1]}$, sufficiently regular.

Taking $w_i := v_i + P_i^{jk}(t/s)^2v_jv_k$ with $P_i^{jk}$ regular function defined in $\Kcal_{[s_0,s_1]}$, $P_i^{jk} = P_i^{kj}$. Then thanks to \eqref{eq1 normal-form-energy}, 
\begin{equation}
\aligned
\Box w_i + c^2w_i =&\big(Q^{jk}_i + 2P_i^{jk}\big)\del_tv_j\del_tv_k - c^2P_i^{jk}(t/s)^2v_jv_k
\\
&+ P_i^{jk}(t/s)^2\left(2\minu(\del v_j,\del v_k) + v_jF_k + v_kF_j \right) + R_i
\\
&+ v_jv_k\Box\big(P_i^{jk}(t/s)^2\big) + 2m^{\alpha\beta}\del_{\alpha}\big(P_i^{jk}(t/s)^2\big)\del_{\beta}(v_jv_k) .
\endaligned
\end{equation}

We can do energy estimate on the above system: on one hand, 
$$
\del_tw_i\big(\Box w_i + c^2w_i\big) = \frac{1}{2}\del_t\left(\sum_{\alpha}|\del_\alpha w_i|^2 + c^2|w_i|^2\right) - \del_a(\del_tw_i\del_aw_i).
$$
On the other hand,
$$
\aligned
\del_tw_i\big(\Box w_i + c^2w_i\big) =& (Q_i^{jk}+2P_i^{jk})\del_tv_j\del_tv_k\del_tv_i - c^2P_i^{jk}(t/s)^2v_jv_k\del_tv_i +  S^{(1)}_i[P,v]
\endaligned
$$
with
$$
\aligned
S^{(1)}_i[P,v] =& \del_t\big(P_i^{jk}(t/s)^2v_jv_k\big)\ \left(\big(Q_i^{jk} + 2P_i^{jk}\big)\del_tv_j\del_tv_k - c^2P_i^{jk}(t/s)^2v_jv_k\right)
\\
&+P_i^{jk}(t/s)^2\del_tw_i\left(2\minu(\del v_j,\del v_k) + v_jF_k + v_kF_j \right) + R_i\del_tw_i
\\
&+v_jv_k\Box\big(P_i^{jk}(t/s)^2\big) \ \del_tw_i + 2m\big(\del\big(P_i^{jk}(t/s)^2\big),\del(v_jv_k)\big)\ \del_tw_i.
\endaligned
$$

Thus
\begin{equation}\label{eq3 normal-form-energy}
\aligned
&\frac{1}{2}\del_t\left(\sum_{\alpha}|\del_\alpha w_i|^2 + c^2|w_i|^2\right) - \del_a(\del_tw_i\del_aw_i)
\\
=&(Q_i^{jk}+2P_i^{jk})\del_tv_j\del_tv_k\del_tv_i - c^2P_i^{jk}(t/s)^2v_jv_k\del_tv_i  + S^{(1)}_i[P,v]
\endaligned
\end{equation}

Next, taking \eqref{eq2 normal-form-energy} with $\omega = P_i^{jk}(t/s)^2v_i$, $v=v_j$ and $w=v_k$,
\begin{equation}\label{eq4 normal-form-energy}
\aligned
&\del_t\left(P_i^{jk}(t/s)^2v_i \big(\del_tv_j\del_tv_k + \sum_a\del_av_j\del_av_k + c^2v_jv_k\big)\right)
- \del_a\left(P_i^{jk}(t/s)^2v_i \big(\del_tv_j\del_av_k + \del_tv_k\del_a v_j\big)\right)
\\
=&P_i^{jk} \del_tv_j\del_tv_k\del_tv_i + P_i^{jk}(t/s)^2 c^2v_jv_k\del_tv_i
\\
&+\del_t((t/s)^2P_i^{jk})v_i\ (s/t)^2\del_tv_j\del_tv_k + \del_t((t/s)^2P_i^{jk})v_i\ c^2v_jv_k
\\
&+\del_t(P_i^{jk}(t/s)^2v_i)\sum_a\delu_a v_j\delu_av_k 
- (x^a/t)\del_t(P_i^{jk}(t/s)^2v_i)(\delu_av_j\del_tv_k + \delu_av_k\del_tv_j) 
\\
&-\sum_a \del_a(P_i^{jk}(t/s)^2v_i)(\del_tv_j\delu_av_k + \del_tv_k\delu_av_j)
+2(x^a/t)\delu_a(P_i^{jk}(t/s)^2v_i)\del_tv_j\del_tv_k
\\
&+ (t/s)^2\big(P_i^{jk}+P_i^{kj}\big)v_i\del_tv_j\ F_k 
\\
=:& P_i^{jk} \del_tv_j\del_tv_k\del_tv_i + c^2P_i^{jk}(t/s)^2 v_jv_k\del_tv_i + S^{(2)}_i[P,v].
\endaligned
\end{equation}

Taking the sum of \eqref{eq3 normal-form-energy} and \eqref{eq4 normal-form-energy}, we obtain
\begin{equation}\label{eq5 normal-form-energy}
\del_t V^0_i + \del_a V^a_i = \big(Q_i^{jk} + 3P_i^{jk}\big)\del_tv_j\del_tv_k + S_i^{(1)}[P,v] + S_i^{(2)}[P,v].
\end{equation}
where
$$
\aligned
V_i^0 :=& \frac{1}{2}\sum_{\alpha}|\del_\alpha w_i|^2 + \frac{1}{2}c^2|w_i|^2 
+ P_i^{jk}(t/s)^2v_i \big(\del_tv_j\del_tv_k + \sum_a\del_av_j\del_av_k + c^2v_jv_k\big),
\\
-V_i^a :=& \del_tw_i\del_aw_i + P_i^{jk}(t/s)^2v_i \big(\del_tv_j\del_av_k + \del_tv_k\del_a v_j\big).
\endaligned
$$

In the rest of this section we always take  $P_i^{jk} = -\frac{1}{3}Q_i^{jk}$, then
\begin{equation}\label{eq6 normal-form-energy}
\del_t V^0_i + \del_a V^a_i = S_i^{(1)}[P,v] + S_i^{(2)}[P,v].
\end{equation}
That is, we managed to eliminated all quadratic nonlinear terms. 

We introduce the following modified energy density for the vector $v:=(v_1,v_2,\cdots v_N)^T$:
$$
e_{Q,c}[v] := 2\sum_{i=1}^N \left(V_i^0 - (x^a/t)V_i^a\right).
$$
and recall the standard energy density for scalar $u$:
$$
e_c[u] := \sum_{\alpha}|\del_\alpha u|^2 + 2(x^a/t)\del_tu\del_au + c^2u^2 = |(s/t)\del_tu|^2 + \sum_{a}|\delu_a u|^2 + c^2u^2
$$
then we establish the following result:
\begin{lemma}\label{lem density-normal-form}
	When 
	\begin{equation}\label{eq7 normal-form-energy}
	\big| (t/s)^2Q_i^{jk}v_i\big| + \big| (t/s)^2Q_i^{jk}v_j\big| + \big|(t/s)^2Q_i^{jk}v_k\big|\leq\vep_s\ll 1,
	\end{equation}
	\begin{equation}\label{eq8 normal-form-energy}
	|(t/s)^2v_j\del_{\alpha}Q_i^{jk}| + |(t/s)^2v_k\del_{\alpha}Q_i^{jk}|\leq\vep_s\ll 1.
	\end{equation}
	then the following relation holds:
	\begin{equation}\label{eq9 normal-form-energy}
	\frac{1}{4}e_{Q,c}[v]\leq \sum_{i=1}^Ne_c[v_i]\leq 4e_{Q,c}[v].
	\end{equation}
\end{lemma}
\begin{proof}
	Denote by $w = (w_1,w_2,\cdots w_N)^T$ and $v= (v_1,v_2,\cdots v_N)^T$. Recall the definition of $w_i$,  we can write 
	\begin{equation}\label{eq10 normal-form-energy}
	w = \big(\mathcal{I} + \mathcal{P}(v)\big)v
	\end{equation}
	where $\mathcal{I}$ is the identity matrix and 
	$$
	\mathcal{P}_i^j(v) = (t/s)^2P_i^{jk}v_k.
	$$
	Furthermore,
	$$
	\aligned
	\del_{\alpha}w_i =& \del_{\alpha}v_i + (t/s)^2P_i^{jk}v_j\del_\alpha v_k + (t/s)^2P_i^{jk}v_k\del_\alpha v_j 
	\\
	&+ (t/s)^2\del_{\alpha}P_i^{jk}v_jv_k 
	+ 2(t/s) P_i^{jk}v_jv_k\ \del_{\alpha}(t/s).
	\endaligned
	$$
	Then  
	\begin{equation}
	\del_{\alpha} w = \big(\mathcal{I} + \mathcal{P}_{\alpha}\big)\del_{\alpha}v + \mathcal{R}_{\alpha}v
	\end{equation}
	with 
	$$
	\mathcal{P}_{\alpha i}^j = \mathcal{P}_{\alpha i}^j[v] := (t/s)^2\big(P_i^{jk}v_k + P_i^{kj}v_k\big)
	$$
	$$
	\mathcal{R}_{\alpha i}^j = \mathcal{R}_{\alpha i} ^j[v] :=  (t/s)^2\del_{\alpha}P_i^{jk}v_k 
	+ 2(t/s)P_i^{jk}v_k\del_{\alpha}(t/s).
	$$
	
	Remark that when \eqref{eq7 normal-form-energy} holds, the matrices $(\mathcal{I} + \mathcal{P})$ and $(\mathcal{I} + \mathcal{P}_{\alpha})$ are invertible. Taking $\vep_s$ sufficiently small and thanks to \eqref{eq8 normal-form-energy}, we will have 
	\begin{equation}
	\frac{1}{2}\sum_{i=1}^Ne_c[w_i] \leq \sum_{i=1}^Ne_c[v_i] \leq 2\sum_{i=1}^Ne_c[w_i].
	\end{equation}
	
	Now let us regard the expression of $e_{Q,c}[v]$:
	$$
	\aligned
	e_{Q,c}[v] =& \sum_{i=1}^N e_c[w_i] 
	+ 2 P_i^{jk}(t/s)^2v_i \big(\del_tv_j\del_tv_k + \sum_a\del_av_j\del_av_k + c^2v_jv_k\big) 
	\\
	&+ 2(x^a/t) P_i^{jk}(t/s)^2v_i \big(\del_tv_j\del_av_k + \del_tv_k\del_a v_j\big).
	\endaligned
	$$
	Then due to \eqref{eq7 normal-form-energy} with $\vep_s\ll 1$, \eqref{eq9 normal-form-energy} holds.
\end{proof}

Now we introduce the modified energy 
$$
E_{Q,c}(s,v) := \int_{\Hcal_s}e_{Q,c}[v] dx
$$
Following the condition \eqref{eq7 normal-form-energy} and \eqref{eq8 normal-form-energy}, 
\begin{equation}\label{eq11 normal-form-energy}
\frac{1}{4}\sum_{i=1}^NE_c(s,v_i) \leq E_{Q,c}(s,v) \leq 4\sum_{i=1}^NE_c(s,v_i)
\end{equation}

Now integrate \eqref{eq6 normal-form-energy} in $\Kcal_{[s_0,s_1]}$ and apply Stokes' formula, the following modified energy identity holds:
\begin{lemma}\label{lem modified-energy-identity}
Under the conditions \eqref{eq7 normal-form-energy}, \eqref{eq8 normal-form-energy}, the following energy identity holds:
\begin{equation}\label{eq1 modified energy}
E_{Q,c}(s_1,v) - E_{Q,c}(s_0,v) = \sum_{i=1}^N\int_{s_0}^{s_1} \big(S_i^{(1)}[P,v] + S_i^{(2)}[P,v]\big)dx.
\end{equation}
\end{lemma}

\section{Recall of basic results in hyperboloidal foliation framework}\label{sec basic-hyper}
In this section we recall some necessary notation and results for the following discussion. In Appendix \ref{App Basic} we will give a sketch of their proofs.
\subsection{Families of vector fields and multi-index}
In the region $\Kcal$, we introduce the following vector fields:
$$
L_a = x^a\del_t + t\del_a,\quad a = 1, 2.
$$
and the following notation of high-order derivatives: let $I,J$ be  multi-indices taking values in $\{0,1,2\}$ and $\{1,2\}$,
$$
I = (i_1,i_2,\cdots, i_m),\quad J = (j_1,j_2,\cdots, j_n).
$$
We define
$$
\del^IL^J = \del_{i_1}\del_{i_2}\cdots \del_{i_m}L_{j_1}L_{j_2}\cdots L_{j_n}.
$$
to be an $(m+n)-$order derivative.

We also define the following vector fields in $\Kcal$:
$$
\delu_a = \delb_a = \frac{x^a}{t}\del_t + \del_a,\quad K = s(s/t)\del_t + 2x^a\delb_a.
$$

For the convenience of discussion, we introduce the following notation on families of vector fields:
\\
{\bf 1.} Partial derivatives, denoted by $\mathscr{P} = \{\del_{\alpha}|\alpha=0,1,2\}$.
\\
{\bf 2.} Lorentzian boosts, denoted by $\mathscr{L} = \{L_a|a=1,2\}$ with $L_a := x^a\del_t + t\del_a$.
\\
{\bf 3.} Hyperbolic derivatives, denoted by $\mathscr{H} = \{\delu_a|a=1,2\}$ with $\delu_a = (x^a/t)\del_t +\del_a$.
\\
We denote by
$$
\mathscr{Z} = \mathscr{P} \cup \mathscr{L}\cup \mathscr{H}
$$
and
$$
Z_i =
\left\{
\aligned
&\del_i,  &&i=0,1,2,
\\
&L_{i-2},  &&i=3,4,
\\
&\delu_{i-4},&&i=5,6.
\endaligned
\right.
$$
Then we introduce the following notation on high-order derivatives. Let $I = (i_1,i_2,\cdots i_N)$ be a multi-index with $i_j\in\{1,2,\cdots, 6\}$ and $|I|=N$. Then
$$
Z^I : = Z_{i_1}Z_{i_2}\cdots Z_{i_N}
$$
is an $N-$orde differential operator.

Suppose that $Z^I$ is composed by $i$ partial derivatives, $j$ Lorentzian boots,  $k$ hyperbolic derivatives, then $Z^I$ is said to be of type $(i,j,k)$. If $Z^I$ is of type $(0,j,0)$, we denote by $Z^I = L^I$ and if $Z^I$ is of type $(i,0,0)$, we denote by $Z^I = \del^I$.

\subsection{Homogeneous functions}\label{func-homo}
We recall the following notion on homogeneous functions:
\begin{definition}
	Let $u$ be a $C^{\infty}$ function defined in $\{t>|x|\}$, satisfying the following properties:
	\\
	${\bf1.}$ For a $k\in\RR$,  $u(\lambda t,\lambda x) = \lambda^ku(t,x),\quad \forall \lambda>0$.
	\\
	${\bf2.}$ $\del^Iu(1,x)$ is bounded by a constant $C$ determined by $|I|$ and $u$ for $|x|< 1$.
	\\
	Then $u$ is said to be {\sl homogeneous of degree $k$}.
\end{definition}
The following properties are immediate:
\begin{proposition}\label{prop 1 homo}
	Let $u,v$ be homogeneous of degree $k,l$ respectively. Then
	\\
	{\bf 1.} When $k=l$, $\alpha u + \beta v$ is homogeneous of degree $k$ where $\alpha$ and $\beta$ are constants.
	\\
	{\bf 2.} $uv$ is homogeneous of degree $k+l$.
	\\
	{\bf 3.} $\del^IL^J u$ is homogeneous of degree $k-|I|$.
	\\
	{\bf 4.} There is a positive constant determined by $I,J$ and $u$ such that the following inequality holds in $\Kcal$:
	\begin{equation}\label{eq 1 homo}
	|\del^IL^Ju|\leq Ct^{k-|I|}.
	\end{equation}
\end{proposition}
\subsection{Analysis on $(s/t)$}
The function $(s/t) = \sqrt{t^2-r^2}/t$ plays an important role in our analysis. We recall the following properties of this function. A detailed proof is presented in Appendix \ref{subsec (s/t)-appendix}.
\begin{proposition}\label{prop (s/t)}
	Let $l,n\in \mathbb{Z}$ and $I$ be a multi-index of type $(i,j,k)$. Then in $\Kcal$,
	\begin{equation}\label{eq1 prop (s/t)}
	\big|Z^I\left((s/t)^lt^n\right)\big|\leq 
	\left\{
	\aligned
	&t^{n-k}(s/t)^l,\quad i=0,
	\\
	&t^{n-k}(s/t)^{l}(t/s^2),\quad i\geq 1.
	\endaligned
	\right.
	\end{equation}
\end{proposition}

\begin{remark}
	We list out some special cases of \eqref{eq 1 lem 3 s/t}:
	\begin{equation}\label{eq 2 lem 2 s/t}
	\big|\del^IL^J(s^n)\big|\leq
	\left\{
	\aligned
	&Cs^n,\quad |I|=0,
	\\
	&Cs^n(t/s^2),\quad |I|\geq 1,
	\endaligned
	\right.
	\quad
	\big|\del^IL^J(s^{-n})\big|\leq
	\left\{
	\aligned
	&Cs^{-n},\quad |I|=0,
	\\
	&Cs^{-n}(t/s^2),\quad |I|\geq 1.
	\endaligned
	\right.
	\end{equation}
\end{remark}

\subsection{Global Sobolev's inequality on hyperboloid}
In order to turn $L^2$ bounds to $L^{\infty}$ bounds with decreasing rates, we need the following global Sobolev type inequality
\begin{proposition}\label{prop Global-Sobolev}
	Let $u$ be a function defined in $\Kcal_{[s_0,s_1]}$, sufficiently regular and vanishing near the conical boundary $\del\Kcal_{[s_0,s_1]}$. Then
	\begin{equation}\label{eq Sobolev}
	|t^{-1}u(t,x)|^2 \leq C\sum_{|I|+|J|\leq 2}\|\del^IL^Ju\|_{L^2(\Hcal_s)}^2,\quad s = \sqrt{t^2-|x|^2}.
	\end{equation}
\end{proposition}

\subsection{Standard energy estimate}
Recall the standard energy defined on hyperboloid for flat (Minkowski) metric ($c\geq 0$):
$$
E_{m,c^2}(s,u):=\frac{1}{2} \int_{\Hcal_s}e_c[u]dx
$$
where the energy density
$$
\aligned
e_c[u]:=&|\del_tu|^2+\sum_a|\del_au|^2 + 2(x^a/t)\del_tu\del_au + c^2u^2 
\\
=&\sum_a |\delu_a|^2 + |(s/t)\del_tu|^2 + c^2u^2
\\
=&|\delu_{\perp}u|^2 + \sum_a|(s/t)\del_a u|^2 + \sum_{a<b}\big|t^{-1}\Omega_{ab}u\big|^2 + c^2u^2.
\endaligned
$$
We denote by $m^{\alpha\beta}$ the standard Minkowski metric. Let $g^{\alpha\beta} = m^{\alpha\beta} + H^{\alpha\beta}$ be a $C^1$ metric defined in the region $\Kcal_{[s_0,s_1]}$, we define
$$
E_{g,c^2}(s,u):= \frac{1}{2}\int_{\Hcal_s}\bigg(g^{00}|\del_tu|^2 - g^{ab}\del_au\del_bu - \sum_a(2x^a/t)g^{a\beta}\del_tu\del_{\beta}u + c^2u^2\bigg)dx.
$$

\begin{proposition}\label{prop 1 energy}
	We consider the $C^2$ solution $u$ to the following wave equation
	$$
	g^{\alpha\beta}\del_{\alpha}\del_{\beta}u + c^2 u = F,
	$$
	in the region $\Kcal_{[s_0,s_1]}$ and vanishes near the conical boundary $\del\Kcal_{[s_0,s_1]}$. $g^{\alpha\beta} = m^{\alpha\beta} + H^{\alpha\beta}$ is a smooth metric defined in $\RR^{1+2}$ and $H^{\alpha\beta}$ vanishes near $\del K$ and out of  $\Kcal$. Suppose that there exists a positive constant $\kappa>1$ such that
	\begin{equation}\label{ineq 1 propo 1 energy}
	\kappa^{-1} E_m(s,u)^{1/2}\leq E_g(s,u)^{1/2}\leq \kappa E_m(s,u)^{1/2}
	\end{equation}
	and
	\begin{equation}\label{ineq 2 prop 1 energy}
	\bigg|\int_{\Hcal_s}\frac{s}{t}\bigg(\frac{1}{2}\del_tg^{\alpha\beta}\del_{\alpha}u\del_{\beta}u - \del_{\alpha}g^{\alpha\beta}\del_tu\del_{\beta}u\bigg)dx\bigg|\leq M[u](s)E_m(s,u)^{1/2}
	\end{equation}
	Then the following energy estimate holds:
	\begin{equation}\label{ineq 3 prop 1 energy}
	E_c(s,u)^{1/2}\leq \kappa^2E_c(2,u)^{1/2} + \kappa^2\int_2^s\big(\|F\|_{L^2(\Hcal_\tau)} + M[u](\tau)\big)d\tau.
	\end{equation}
	%
\end{proposition}
The proof relies on the following differential identity:
\begin{equation}\label{eq differential-standard}
\aligned
\del_tu\cdot g^{\alpha\beta}\del_{\alpha}\del_{\beta}u =&\frac{1}{2}\del_t(g^{00}|\del_tu|^2 - g^{ab}\del_au\cdot\del_bu) + \del_a(\del_tug^{a\beta}\del_{\beta}u)
\\
&- \frac{1}{2}\del_tg^{00}|\del_tu|^2 + \frac{1}{2}\del_tg^{ab}\del_au\del_bu - \del_ag^{a\beta}\del_{\beta}u\del_tu.
\endaligned
\end{equation}
Then integrate this identity in the region $\Kcal_{[s_0,s_1]}$ and by Stokes' formula, we obtain the following standard energy estimate on hyperboloids (For more detail , see \cite{LM1}, \cite{M1}).


\section{Bounds with energies}\label{sec bounds-linear terms}
In this section we firstly re-state some $L^2$ and $L^{\infty}$ estimates on linear terms established in our previous work (e.g.\cite{LM1},\cite{LM2}) with notation which are more convenience for sub sequential discussion.  Then we complete the conformal energy estimate and sharp decay estimate on wave equation established in previous sections with 
\subsection{Notation}\label{subsec notation-bounds}
Let $u$ be a function defined in the region $\Kcal_{[s_0,s_1]}$ and $T = T^{\alpha\beta}\del_{\alpha}\otimes\del_{\beta}$. Let $\mathcal{I}_{p,k} = \{I| I \text{ is of type }(p,p-k,0) \}$.
\begin{equation}\label{eq1 notation}
\aligned
|u|_{p,k} &:= \max_{K\in \mathcal{I}_{p,k}}|Z^K u|,\quad &&|u|_p := \max_{0\leq k\leq p}|u|_{p,k},
\\
|T|_{p,k} &:= \max_{\alpha,\beta}|T^{\alpha\beta}|_{p,k}, &&|T|_p := \max_{0\leq k\leq p}|T|_{p,k},
\\
|\del u|_{p,k} &:= \max_{\alpha=0,1,2}|\del_{\alpha} u|_{p,k}, &&|\del u|_p := \max_{0\leq k\leq p}|\del u|_{p,k},
\\
|\del^m u|_{p,k} &:= \max_{|I|=m}|\del^I u|_{p,k}, &&|\del^m u|_p := \max_{0\leq k\leq p}|\del^I u|_{p,k},
\\
|\dels u|_{p,k} &:= \max\{|\delu_1 u|_{p,k},|\delu_2u|_{p,k}\}, &&|\dels u|_p := \max_{0\leq k\leq p}|\dels u|_{p,k}
\\
|\del\dels  u|_{p,k} &:=\max_{a,\alpha} \{|\delu_a\del_{\alpha} u|_{p,k},|\del_{\alpha}\delu_a u|_{p,k}\},
&&| \del\dels u|_p :=\max_{0\leq k\leq p}|\dels \del u|_{p,k}.
\endaligned
\end{equation}

Furthermore, we have the following results:
\begin{lemma}\label{lem 1 notation}
Let $L$ be a multi-index of type $(p-k+m,k,0)$, then in $\Kcal_{[s_0,s_1]}$,
\begin{equation}\label{eq1 lem 1 notation}
|Z^L u|\leq C|\del^m u|_{p,k}.
\end{equation}
Inversely,
\begin{equation}\label{eq2 lem 1 notation}
|\del^m u|_{p,k}\leq C\max_{|I|= m\atop L\in\mathcal{I}_{p,k}}|\del^IZ^Lu|
\end{equation}
Here the constant $C$ is determined by $L,m$.

Let $L$ be a multi-index of type $(p-k-1,k,1)$, then in $\Kcal_{[s_0,s_1]}$,
\begin{equation}\label{eq3 lem 1 notation}
|Z^L u|\leq
\left\{
\aligned
&Cs^{-1}(s/t)|\del u|_{p,k+1} ,\quad p\geq k+2,
\\
&C\sum_{a,|J|\leq k}|\delu_a L^J u|,\quad p=k+1.
\endaligned
\right.
\end{equation}
\end{lemma}
\begin{proof}
\eqref{eq1 lem 1 notation} and \eqref{eq2 lem 1 notation} are deduced from \eqref{eq 1 lem 2 high-order}. For \eqref{eq3 lem 1 notation}, we can write
$$
Z^L u = Z^{L_1}\delu_aZ^{L_2}u 
$$
with $L_1$ and $L_2$ type of $(p_1-k_1,k_1,0)$ and $(p_2-k_2,k_2,0)$ with $p=p_1+p_2+1$ and $k=k_1+k_2$. Then
$$
Z^L u =  Z^{L_1}\delu_aZ^{L_2}u  = Z^{L_1}\big(t^{-1}L_aZ^{L_2}u\big) = \sum_{L_{11}+L_{12}=L_1} Z^{L_{11}}(t^{-1})\ Z^{L_{12}}L_aZ^{L_2}u.
$$
Then we distinguish between the following cases. 

First, when $p-k-1 = 0$, i.e., in $Z^L$ there is no partial derivative. Thus in $Z^{L_{11}}$ and $Z^{L_{12}}L_aZ^{L_2}$ there is partial derivative. Denote by $Z^{L_{11}} = L^{J_1}$ and $Z^{L_{12}}L_aZ^{L_2} = L^{J_2}$.  Then by homogeneity:
$$
|Z^L u|\leq Ct^{-1}|L^{J_2}u|
$$
Observe that $|J_2|\geq 1$, 
$$
|Z^L u| \leq Ct^{-1}|L_aL^{J_2'}u| = C|\delu_a L^{J_2'}u |
$$
with $|J_2'|\leq k$ which concludes the case $p = k+1$. 

When $p\geq k+2$, in $ Z^L$ there is at least one partial derivative. When $Z^{L_{12}}L_aZ^{L_2}$ does not contain partial derivative, $Z^{L_{11}}$ contains at least one partial derivative. Then
$$
|Z^{L_{11}}(t^{-1})|\leq Ct^{-2}\quad \Rightarrow \quad |Z^L u|\leq Ct^{-2} |Z^{L_{12}}L_aZ^{L_2}u|\leq Ct^{-1}|\del Z^{L'}u|
$$
with $L'$ being of type $(k',k',0)$ where $p'\leq p-1$, $k'\leq k$.
\\
When  $Z^{L_{12}}L_aZ^{L_2}$ contains at least one partial derivative, we apply \eqref{eq1 lem 1 notation} on $Z^{L_{12}}L_aZ^{L_2}u$
$$
|Z^{L_{12}}L_aZ^{L_2}u|\leq C|\del u|_{p,k+1}.
$$
Thus we conclude by \eqref{eq3 lem 1 notation}.
\end{proof}

We introduce the notion of ``linear combination''.
$$
A \simeq B_1+B_2+\cdots B_n,\quad \text{or}\quad A\simeq \sum_{\alpha\in\lambda}B_{\alpha}
$$
for ``$A$ is a finite linear combination of $B_i, i=1,2,\cdots,n$ or $B_{\alpha}, \alpha\in\Lambda$ with homogeneous coefficients of degree zero''. When $\Lambda = \emptyset$, we take $A=0$.  Then the following result is obvious:
\begin{lemma}\label{lem 2 notation}
Let $A$ and $B_i,i=1,2,\cdots, n$ be functions defined in $\Kcal_{[s_0,s_1]}$. Suppose that
$$
A \simeq B_1+B_2+\cdots B_n.
$$
Then
$$
|A|_{p,k}\leq C \sum_{i=1}^n|B_i|_{p,k}
$$
where $C$ is a positive constant determined by the coefficients of linear combination. 
\end{lemma}

The following estimate on multi-linear form is trivial, we omit the proof.
\begin{lemma}\label{lem multi-linear}
	Let $u_i, i=1,2,\cdots,m$ be functions defined in $\Kcal_{[s_0,s_1]}$, sufficiently regular. Let $U = \prod_{i=1}^m u_i$ and denote by $p_1 = [p/2], k_1=[k/2]$, then
	\begin{equation}\label{eq1 multi-linear}
	\||U|_{p,k}\|_{L^2(\Hcal_s)}\leq \sum_{j=1}^m\||u_j|_{p,k}\|_{L^2(\Hcal_s)}\prod_{i=1,i\neq j}^m \||u_i|_{p_1,k_1}\|_{L^{\infty}(\Hcal_s)}.
	\end{equation}
	and especially:
	\begin{equation}\label{eq1 bilinear}
	\||uv|_{p,k}\|_{L^2(\Hcal_s)}\leq \||u|_{p,k}\|_{L^2(\Hcal_s)}\||v|_{p_1,k_1}\|_{L^{\infty}(\Hcal_s)} 
	+ \||v|_{p,k}\|_{L^2(\Hcal_s)}\||u|_{p_1,k_1}\|_{L^{\infty}(\Hcal_s)}.
	\end{equation}
\end{lemma}

\subsection{Basic bounds on linear terms}\label{subsec linear-bounds}
With the above notation and recall the definition \eqref{eq 3 01-01-2019} and \eqref{eq 4 01-01-2019}, we write the following bounds that are frequently applied in the subsequential discussion:
\begin{lemma}\label{lem 3 notation}
Let $u$ be a function defined in $\Kcal_{[s_0,s_1]}$, sufficiently regular. Let $N\geq 2$, then the following quantities are bounded by $C\Ecal^N(s,u)^{1/2}$ with $C$ a constant determined by $N$ :
\begin{equation}\label{list1 lem 3 notation}
\aligned
&\|(s/t)|\del u|_N\|_{L^2(\Hcal_s^*)},\quad \||(s/t)\del u|_N\|_{L^2(\Hcal_s^*)},\quad \||\dels u|_N\|_{L^2(\Hcal_s^*)},\quad \|s|\del\dels u|_{N-1}\|_{L^2(\Hcal_s^*)}
\\
&\|s|\del u|_{N-2}\|_{L^{\infty}(\Hcal_s^*)},\quad \|t|\dels u|_{N-2}\|_{L^{\infty}(\Hcal_s^*)},\quad \|st|\del\dels u|_{N-3}\|_{L^2(\Hcal_s^*)}
\endaligned
\end{equation}
	
For $c>0$, the following quantities are bounded by $C\Ecal_c^N(s,u)^{1/2}$ with $C$ a constant determined by $N$:
\begin{equation}\label{list2 lem 3 notation}
\aligned
&\||u|_N\|_{L^2(\Hcal_s^*)},\quad \|t|\dels u|_{N-1}\|_{L^2(\Hcal_s^*)},
\\
&\|t|u|_{N-2}\|_{L^{\infty}(\Hcal_s^*)},\quad \|t^2|\dels u|_{N-3}\|_{L^\infty(\Hcal_s^*)}.
\endaligned
\end{equation}
	
The following quantities are bounded by $C\Fcon^N(s,u)$ with $C$ a constant determined by $N$:
\begin{equation}\label{list3 lem 3 notation}
\aligned
&\|(s/t)|u|_N\|_{L^2(\Hcal_s^*)},\quad \||(s/t)u|_N\|_{L^2(\Hcal_s^*)},\quad \|s(s/t)^2|\del u|_N\|_{L^2(\Hcal_s^*)},\quad \|s(s/t)|(s/t)\del u|_N\|_{L^2(\Hcal_s^*)}  
\\
&\|s|\dels u|_N\|_{L^2(\Hcal_s^*)},
\\
&\|s|u|_{N-2}\|_{L^{\infty}(\Hcal_s^*)},\quad \|s^2(s/t)|\del u|_{N-2}\|_{L^{\infty}(\Hcal_s^*)},\quad 
\|st|\dels u|_{N-2}\|_{L^{\infty}(\Hcal_s^*)}.
\endaligned
\end{equation}
\end{lemma}
\begin{proof}
These are direct results of proposition \ref{lem 1 esti-high} and proposition \ref{lem 3 esti-high} except the bound on $|(s/t)u|_N$ and $|(s/t)\del u|_N$. For this term we only need to remark the following calculation. Let $K$ be type $(i,j,k)$, $|K|=N$. Then
$$
Z^K((s/t)u) = \sum_{K_1+K_2=Z}Z^{K_1}(s/t)\ Z^{K_2}u.
$$
Then recall \eqref{eq 1 lem 3 s/t}, we obtain
$$
|Z^K((s/t)u)|\leq C(s/t)|u|_{p,k}
$$	
where $C$ is determined by $N$. Then combined with proposition \ref{lem 1 esti-high} and proposition \ref{lem 3 esti-high}, the bounds on  $|(s/t)u|_N$ and $|(s/t)\del u|_N$ are established.
\end{proof}

\subsection{Conformal energy estimate with curved back-ground metric}\label{subsec conforml-bounds}
Now based on the differential identity \eqref{eq 1 21-05-2017} and the notation introduced in subsection \ref{subsec notation-bounds}, we establish the following energy estimate:
\begin{proposition}\label{prop1 conformal}
	Let $u$ be a function defined in $\Kcal_{[s_0,s_1]}$, sufficiently regular and vanishes near the conical boundary $\del\Kcal_{[s_0,s_1]}$. Suppose that $g^{\alpha\beta} = m^{\alpha\beta} + h^{\alpha\beta}$ is a metric defined in $\Kcal_{[s_0,s_1]}$, sufficiently regular with $h^{\alpha\beta}$ vanishes near $\del\Kcal_{[s_0,s_1]}$. Let
	\begin{equation}\label{eq1 conformal}
	F = g^{\alpha\beta}\del_{\alpha}\del_{\beta}u.
	\end{equation}
	Then if $0\leq \vep_s\ll 1$ and
	\begin{equation}\label{eq1 lem 1 conformal}
	\aligned
	&|\hb^{00}| + |h|\leq \vep_s(s/t),
	\\
	&|\del \hb^{00}| + |\del h| + (t/s)|\dels \hb^{00}| + (t/s)|\dels h|\leq \vep_ss^{-1},
	\\
	& |\del L\hb^{00}| \leq \vep_s ts^{-2},\quad  (s/t)|\del\del\hb^{00}|\leq \vep_s s^{-2}.
	\endaligned
	\end{equation}
	Then the following bound holds:
	\begin{equation}\label{eq3 conformal}
	\aligned
	\Ec(s_1,u)^{1/2} \leq& CE_{\text{con},g}(s_0,u)^{1/2} + C\int_{s_0}^{s_1}\|sF\|_{L^2(\Hcal_s)}\ ds 
	\\
	&+C\vep_s \int_{s_0}^{s_1}s^{-1} \big(\Ec(s,u)^{1/2} + \Fc(s_0;s,u)\big).
	\endaligned
	\end{equation}
\end{proposition}
\begin{remark}
	The estimate \eqref{eq3 conformal} seems to be not very reasonable: both side contain $\Ec(s,u)^{1/2}$ and $\Fc(s_0,s,u)$ is in fact an integration of $\Ec(s,u)^{1/2}$. However it is satisfactory for our bootstrap argument. In fact we will suppose that $\Ec(s,u)^{1/2}\sim C\vep s^{\delta}$ which leads to $\Fc(s_0;s,u)\sim C\vep s^{\delta}$. Then if we can prove that 
	$$
	\aligned
	&\|sF\|_{L^2(\Hcal_s)}\lesssim C\vep^2 s^{-1+\delta},
	\\
	&\|R_g\|_{L^{\infty}(\Hcal_s)} + \|A_g\|_{L^{\infty}(\Hcal_s)} + \|(t/s)\dels N_g\|_{L^{\infty}(\Hcal_s)} + \|(t/s)\dels\hb^{00}\|_{L^\infty(\Hcal_s)} \lesssim C\vep s^{-1},
	\endaligned 
	$$
	then the above estimate will give desired refined bound $\Ec(s,u)^{1/2}\sim C\vep^2 s^{\delta}$.
\end{remark}

In order to prove Proposition \ref{prop1 conformal}, we firstly analyze the objects appears in \eqref{eq 1 21-05-2017} 
$$
\mathscr{K}_g + N_g,\quad R_g^{ab}, \quad S_g[u],\quad T_g^a[u].
$$
For the convenience of discussion, we introduce the following functions of ``energy density'':
$$
\ec[u] := \sum_a|s\delu_a u|^2 + |(K+1)u|^2,\quad \ebc[u]:= \sum_a|s\delu_au|^2 + \big|(s/t)^2s\del_tu\big|^2 + (s/t)^2u^2.  
$$
and 
$$
e_{\text{con},g}[u] := |\mathscr{K}_gu + N_gu|^2  - s^2\gb^{00}\gb^{ab}\delb_au\delb_bu.
$$
Thanks to \eqref{eq conformal-energy} and \eqref{eq 2 01-01-2019},
$$
\int_{\Hcal_s}\ec[u]\ dx = \Ec(s,u),\quad \int_{\Hcal_s}\ebc[u]\ dx\leq F^2_{\text{con}}(s_0;s,u).
$$
And we have the following result:
\begin{lemma}\label{lem 1 conformal}
If \eqref{eq1 lem 1 conformal} holds, then
\begin{equation}\label{eq3 lem 1 conformal}
|S_g[u]|^2 + |T_g^a[u]|^2\leq C\vep_s^2 s^{-2}\ebc[u].
\end{equation}
\begin{equation}\label{eq4 lem 1 conformal}
s^2|R_g^{ab}\delu_au\delu_bu|\leq C\vep_s s^{-1}\ec[u],
\end{equation}
\begin{equation}\label{eq2 lem 1 conformal}
\ec[u]\leq Ce_{\text{con},g}[u] + C\vep_s \ebc[u].
\end{equation}
\end{lemma}
\begin{proof}
Recall \eqref{eq 1-3 21-05-2017},  \eqref{eq3 lem 1 conformal} demands the following bounds:
$$
\aligned
|s^{-1}\delb_s(s\hb^{a0})| +  |\delb_a\hb^{ab}| + (t/s)|\delb_sN_g| + (t/s)|\delb_aN_g| + (t/s)|\delb_a\hb^{00}|\leq \vep_s s^{-1}.
\endaligned
$$
Recall \eqref{eq 1-4 21-05-2017}, \eqref{eq4 lem 1 conformal} demands taht the following terms
$$
\aligned
&|\hb^{00}\gb^{ab}|,\quad   |\mb^{00}\hb^{ab}|,\quad |N_g-1|,
\\
&|s\delb_c(\hb^{0c}\gb^{ab})|,\quad |s\delb_c(\mb^{0c}\hb^{ab})|,\quad  |s\delb_c(\gb^{cb}\hb^{0a})|,\quad  |s\delb_c(\hb^{cb}\mb^{0a})|,\quad |s\delb_s(\hb^{00}\gb^{ab})|,\quad  |s\delb_s(\hb^{ab})|
\endaligned
$$	
are bounded by $\vep_s$.
These bounds are guaranteed by \eqref{eq1 lem 1 conformal} and the following relations:
$$
\hb^{ab} = h^{ab},\quad \hb^{a0} = (t/s)\hu^{a0},\quad \hb^{00} = (t/s)^2\hu^{00}.
$$
and 
$$
\del_t(t/s) = -r^2/s^3,\quad \delu_a(t/s) = (x^a/t)s^{-1}.
$$

On the other hand,
$$
\aligned
(\mathscr{K}_g + N_g)u 
= (K+1)u + s(\hb^{00}\delb_s + 2\hb^{a0}\delb_a)u + (N_g-1)u
\endaligned
$$
and
$$
-s^2\gb^{00}\gb^{ab}\delb_au\delb_bu = (1+\hb^{00})\sum_{a}|s\delb_a u|^2 - s^2\gb^{00}\hb^{ab}\delb_au\delb_bu .
$$
Then under the assumption \eqref{eq1 lem 1 conformal}
\begin{equation}\label{eq0 pf conformal}
\aligned
&|Ku + u|^2\leq  |\mathscr{K}_gu + N_g u|^2 + C\vep_s \big(|s\delu_au|^2 + |(s/t)u|^2 + |(s/t)u|^2\big),
\\
&\sum_a|s\delb_au|^2\leq -Cs^2\gb^{00}\gb^{ab}\delb_au\delb_bu
\endaligned
\end{equation}
which lads to \eqref{eq2 lem 1 conformal}.
\end{proof}

\begin{proof}[Proof of Proposition \ref{prop1 conformal}]
By integrating \eqref{eq 1 21-05-2017} in the region $\Kcal_{[s_0,s]}$ with Stokes' formula:
\begin{equation}\label{eq1 pf conformal}
\aligned
E_{\text{con},g}(s,u) - E_{\text{con},g}(s_0,u)
=&\int_{s_0}^s\int_{\Hcal_{s'}} 2s(\mathscr{K}_gu + N_gu)F\ dxds' 
\\
&- \int_{s_0}^s\int_{\Hcal_{s'}}2 \big(s^2R_g^{ab}\delb_au\delb_bu + (\mathscr{K}_g+N_g)u\cdot S_g[u]  + s\delb_bu\cdot T_g^b[u]\big)\ dxds'.
\endaligned
\end{equation}
Differentiate the above identity with respect to $s$, we obtain:
\begin{equation}\label{eq2 pf conformal}
\aligned
\frac{d}{ds}E_{\text{con},g}(s,u)
=&\int_{\Hcal_s} 2s(\mathscr{K}_gu + N_gu)F\ dx 
\\
&- \int_{\Hcal_s}2 \big((s^2R_g^{ab}\delb_au\delb_bu + (\mathscr{K}_g+N_g)u\cdot S_g[u]  + s\delb_bu\cdot T_g^b[u]\big)\ dx.
\endaligned
\end{equation}
Now we remark that
$$
\aligned
\|\big(\mathscr{K}_g +N_g\big) u\cdot S_g[u]\|_{L^1(\Hcal_s)}\leq& \|\big(\mathscr{K}_g +N_g\big) u\|_{L^2(\Hcal_s)}\|S_g[u]\|_{L^2(\Hcal_s)}
\\
\leq& C\vep_s s^{-1} \Ec(s,u)^{1/2}\Fc(s_0;s,u)
\endaligned
$$
$$
\aligned
\|s\delb_bu\cdot T_g^b[u]\|_{L^1(\Hcal_s)}
\leq & \|s\delb_bu\|_{L^2(\Hcal_s)}\|T_g^b[u]\|_{L^2(\Hcal_s)}\leq C\vep_s s^{-1}\Ec(s,u)^{1/2}\Fc(s_0;s,u).
\endaligned
$$
$$
\|s^2R_g^{\alpha\beta}\delb_au\delb_bu\|_{L^1(\Hcal_s)}\leq C\vep_s s^{-1}\Ec(s,u).
$$
Combine the above bounds with \eqref{eq2 pf conformal}, 
$$
\aligned
2E_{\text{con},g}(s,u)^{1/2}&\frac{d}{ds}E_{\text{con},g}(s,u)^{1/2}
\\
\leq& C\Ec(s,u)^{1/2}\big(\|sF\|_{L^2(\Hcal_s)} + \vep_s s^{-1}\Ec(s,u)^{1/2} + \vep_s s^{-1}\Fc(s_0;s,u)\big).
\endaligned
$$
This leads to the desired estimate.
\end{proof}

\subsection{Bounds on commutators}\label{subsec commu}
In this subsection we recall the estimates of the following terms:
$$
[\del^IL^J, H^{\alpha\beta}\del_{\alpha}\del_{\beta}]u.
$$
These terms appear when we derive the wave equation with respect to $\del^IL^J$. In \cite{LM1}  the following estimate is (implicitly) proved:
\begin{proposition}\label{prop commu}
Let $u$ be a sufficiently regular function defined in $\Kcal_{[s_0,s_1]}$. Then
\begin{equation}\label{eq 1 prop commu}
\aligned
\big|[\del^IL^J,H^{\alpha\beta}\del_{\alpha}\del_{\beta}] u \big|\leq &
\sum_{|I_1|+|I_2|\leq |I|,|J_1|+|J_2|\leq|J|\atop|I_2|+|J_2|\geq 1}\!\!\!\!\!\!\!\!\!\!
|\del^{I_2}L^{J_2}\Hu^{00}||\del_t\del_t\del^{I_1}L^{J_1}u|
+ |\Hu^{00}|\sum_{0\leq |J'|<|J|}\!\!\!\!|\del_t\del_t\del^IL^{J'}u| 
\\
&+t^{-1}\sum_{p_1+p_2\leq p,p_1<p\atop k_1+k_2\leq k}|H|_{p_2,k_2}|\del u|_{p_1+1,k_1+1} + t^{-1}|H||\del u|_{p,k}.
\endaligned
\end{equation}
	
\end{proposition}

We remark that
\begin{equation}\label{eq1 curved-Hessian}
\aligned
H^{\alpha\beta}\del_{\alpha}\del_{\beta} u 
=&\Hu^{00}\del_t\del_tu + \sum_{(\alpha,\beta)\neq(0,0)}\Hu^{\alpha\beta}\delu_{\alpha}\delu_{\beta}u + H^{\alpha\beta}\del_{\alpha}\big(\Psiu_{\beta}^{\beta'}\big)\delu_{\beta'}u
\\
=:&\Hu^{00}\del_t\del_tu + \Hu(\del\del,\del)u.
\endaligned
\end{equation}
The ``good'' component $\Hu(\del\del,\del)u$ can be written as:
\begin{equation}\label{eq 1 28-06-2019}
\aligned
\Hu(\del\del,\del)u=& t^{-1}\left(2\Hu^{a0}\del_tL_a + \Hu^{ab}\delu_aL_b\right)u 
\\
&+ t^{-1}\left(-2\Hu^{a0}\delu_a - \Hu^{ab}(x^a/t)\delu_b + \Hu^{a0}(x^a/t)\del_t + H^{\alpha\beta}t\del_{\alpha}\big(\Psiu_{\beta}^{\beta'}\big)\delu_{\beta'}\right)u
\\
=:& T_1[H,u] + T_2[H,u].
\endaligned
\end{equation}
Then
\begin{equation}\label{eq1' curved-Hessian}
\aligned
\,&[\del^IL^J,H^{\alpha\beta}\del_{\alpha}\del_{\beta}]u 
=
[\del^IL^J,\Hu^{00}\del_t\del_t]u + [\del^IL^J,\Hu(\del\del,\del)]u.
\endaligned
\end{equation}

Then we have the following result for the ``good components''. Its proof is contained in Appendix \ref{Appendix lem null-commu}.
\begin{lemma}[Good components of commutator]\label{lem null-commu}
Let $u$ be a function defined in $\Kcal_{[s_0,s_1]}$, sufficiently regular. Then
\begin{equation}\label{eq1 lem null-commu}
|[\del^IL^J,\Hu(\del\del,\del)]u| \leq Ct^{-1}\sum_{p_1+p_2=p,p_1<p\atop k_1+k_2=k}\!\!\!\!|\del u|_{p_1+1,k_1+1}|H|_{p_2,k_2} + Ct^{-1}|H||\del u|_{p,k}.
\end{equation}
\end{lemma}

Then we focus on the most interesting component of commutator: $[\del^IL^J,\Hu^{00}\del_t\del_t]u$. We establish the following result:
\begin{lemma}[Essential components of commutator]\label{lem essential-commu}
	Let $u$ be a function defined in $\Kcal_{[s_0,s_1]}$, sufficiently regular. Then
	\begin{equation}\label{eq1 lem essential-commu}
	\aligned
	|[\del^IL^J,\Hu^{00}\del_t\del_t]u|\lesssim &
	\sum_{|I_1|+|I_2|\leq |I|,|J_1|+|J_2|\leq|J|\atop|I_2|+|J_2|\geq1}\!\!\!\!\!\!\!\!\!\!
	|\del^{I_2}L^{J_2}\Hu^{00}||\del_t\del_t\del^{I_1}L^{J_1}u|
	+ |\Hu^{00}|\sum_{0\leq |J'|<|J|}\!\!\!\!|\del_t\del_t\del^IL^{J'}u| 
	\\
	&+t^{-1}\sum_{p_1+p_2\leq p,p_1<p\atop k_1+k_2\leq k}|\Hu^{00}|_{p_2,k_2}|\del u|_{p_1+1,k_1+1} + Ct^{-1}|\Hu^{00}||\del u|_{p,k}
	\endaligned
	\end{equation}
	where $|I|+|J| = p$, $|J|=k$.
\end{lemma}
\begin{proof}
	We make the following calculation:
	\begin{equation}\label{eq1 pr lem essential-commu}
	\aligned
	\,[\del^IL^J,\Hu^{00}\del_t\del_t]u =& \sum_{I_1+I_2=I,|I_2|+|J_2|\geq 1\atop J_1+J_2=J}\!\!\!\!\!\! \del^{I_2}L^{J_2}\Hu^{00}\  \del^{I_1}L^{J_1}\del_t\del_tu+ \Hu^{00}[\del^IL^J,\del_t\del_t]u
	\\
	\simeq &\sum_{I_1+I_2=I,|I_2|+|J_2|\geq 1\atop |J_1|+|J_2|\leq|J|}\!\!\!\!\!\! \del^{I_2}L^{J_2}\Hu^{00}\  \del_\alpha\del_\beta\del^{I_1}L^{J_1}u 
	+ \Hu^{00}\sum_{|J'|<|J|}\del_{\alpha}\del_{\beta}\del^IL^{J'}u
	\endaligned
	\end{equation}
	where in the second equality \eqref{eq 2 comm} is applied.
	
	Remark that 
	\begin{equation}\label{eq 1 lem 1 Hessian}
	\del_t\del_a u = -\frac{x^a}{t}\del_t\del_t u + t^{-1}\big(\del_tL_a - \delu_a + (x^a/t)\del_t\big)u
	\end{equation}
	\begin{equation}\label{eq 2 lem 1 Hessian}
	\del_a\del_bu = \frac{x^ax^b}{t^2}\del_t\del_tu + t^{-1}\big(\del_aL_b - (x^b/t)\del_tL_a + (x^b/t)\delu_a - \delta_{ab}\del_t - (x^ax^b/t^2)\del_t\big)u
	\end{equation}
	Thus by \eqref{eq1 lem 1 notation},
	\begin{equation}\label{eq 3 lem 1 Hessian}
	|\del_{\alpha}\del_{\beta}\del^IL^Ju|\lesssim |\del_t\del_t\del^IL^Ju| + t^{-1}|\del u|_{p+1,k+1}.
	\end{equation}
	
	Then, substitute the above bound into \eqref{eq1 pr lem essential-commu}, the desired result is established.
\end{proof}

Now proposition \ref{prop commu} is direct by combining \eqref{eq1' curved-Hessian} together with \eqref{eq1 lem null-commu} and \eqref{eq1 lem essential-commu} (Remark that $\Hu^{00}$ is a linear combination of $H^{\alpha\beta}$). 
\section{Normal form transform : bounds and estimates}\label{sec normal-form transform-bounds}
Based on the notation and estimates established in the previous section, we will complete the discussion on normal-form transform. In this section we follow the notation applied in section \ref{sec normal-form transform}.

\subsection{Modified energy estimate on Klein-Gordon system}
\begin{proposition}\label{prop normal-form-energy}
	Let $(v_i)_{i=1,2,\cdots,N}$ be a solution of \eqref{eq-main normal-form-energy}, sufficiently regular and vanishes near the conical boundary $\del\Kcal_{[s_0,s_1]}$.  $Q_i^{jk}$ and $R_i$ sufficiently regular in $\Kcal_{[s_0,s_1]}$ and satisfy \eqref{eq7 normal-form-energy} and \eqref{eq8 normal-form-energy}. Furthermore, suppose that
	\begin{equation}\label{eq1 prop normal-form-energy}
	|L_av_j| + |\del_\alpha v_j| + (t/s)|v_j|\leq \kappa  (s/t)s^{-1+\delta},\quad j=1,\cdots N, \quad \alpha = 0,1,2,\quad a=1,2,\quad \kappa\leq 1.
	\end{equation} 
	\begin{equation}\label{eq2 prop normal-form-energy}
	|Q_i^{jk}| + (s/t)^2|L((t/s)^2Q_i^{jk})| + s(s/t)^3|\del((t/s)^2Q_i^{jk})| + s^2(s/t)^2\big|\Box \big((t/s)^2Q_i^{jk}\big)\big| \leq C,
	\end{equation}
	then the following estimate holds:
	\begin{equation}\label{eq3 prop normal-form-energy}
	\aligned
	E_{Q,c}(s_1,v)^{1/2}\leq& E_{Q,c}(s_0,v)^{1/2} + C\sum_{i=1}^N\int_{s_0}^{s_1}\|R_i\|_{L^2(\Hcal_s)}\ ds
	\\
	&+ C\sum_{i=1}^N\int_{s_0}^{s_1}s^{-2+2\delta}\kappa \big(E_c(s,\del v_i)^{1/2} + E_c(s,v_i)^{1/2}\big)\ ds.
	\endaligned
	\end{equation}
\end{proposition}
\begin{remark}
The fact that the right-hand-side of \eqref{eq3 prop normal-form-energy} contains $E_c(s,\del v_i)^{1/2}$ and $E_c(s,v_i)^{1/2}$ seems to be not very satisfactory, however, the importance is the convergent factor $s^{-2+2\delta}$. This shows that even if the standard energy is increasing (no too fast), the modified energy will remain globally bounded.
\end{remark}
\begin{proof}
Differentiate \eqref{eq1 modified energy} with respect to $s_1$, we obtain
	\begin{equation}\label{eq1 pr-normal-form-energy}
	\frac{d}{ds}E_{Q,c}(s,v) = \sum_{i=1}^N\int_{\Hcal_s}(s/t)\left(S_i^{(1)}[P,v] + S_i^{(2)}[P,v]\right)dx
	\end{equation}
	
	Now we analyse $S_i^{(1)}$ and $S_i^{(2)}$.
	%
	By \eqref{eq1 prop normal-form-energy} and \eqref{eq2 prop normal-form-energy},
	\begin{equation}
	|\del_{\alpha}\big(P_i^{jk}(t/s)^2v_l\big)|\leq C\kappa (t/s)s^{-1+\delta},\quad \big|\delu_a\big(P_i^{jk}(t/s)^2v_l\big)\big|\leq C\kappa s^{-2+\delta}
	\end{equation}
	and
	\begin{equation}
	\aligned
	\big|\del_\alpha\big(P_i^{jk}(t/s)^2v_jv_k\big)\big| \leq C\kappa (t/s)s^{-1+\delta}\sum_{i=1}^N|v_i|.
	\endaligned
	\end{equation}
	$$
	\aligned
	&\big\|(s/t)\del_t\big(P_i^{jk}(t/s)^2v_jv_k\big)\ \left(\big(Q_i^{jk} + 2P_i^{jk}\big)\del_tv_j\del_tv_k - c^2P_i^{jk}(t/s)^2v_jv_k\right)\big\|_{L^1(\Hcal_s)}
	\\
	\leq& C\kappa s^{-2+2\delta} \sum_{i=1}^N\|v_i\|_{L^2(\Hcal_s)}\left(\sum_{i=1}^N\|v_i\|_{L^2(\Hcal_s)} + \sum_{i=1}^N\|(s/t)\del_tv_i\|_{L^2(\Hcal_s)}\right)
	\\
	\leq& C\kappa s^{-2+2\delta}E_{Q,c}(s,v)
	\endaligned
	$$
	where for the last inequality we have applied \eqref{eq9 normal-form-energy}.
	
	Remark that
	$$
	\aligned
	&\|(s/t)\del_tw_i\ P_i^{jk}(t/s)^2\minu(\del v_j,\del v_k)\|_{L^1(\Hcal_s)}
	\\
	\leq&E_{Q,c}(s,v)^{1/2}\|P_i^{jk}(t/s)^2\minu(\del v_j,\del v_k)\|_{L^2(\Hcal_s)}
	\\
	\leq& C E_{Q,c}(s,v)^{1/2}\left(\|(t/s^2)L_av_j\del_tv_k\|_{L^2(\Hcal_s)} + 
	\|(t/s^2)\del_tv_jL_av_k\|_{L^2(\Hcal_s)} + \|(t/s^2)L_av_j\delu_bv_k\|_{L^2(\Hcal_s)}\right)
	\\
	\leq&  C\kappa s^{-2+\delta}E_{Q,c}(s,v)^{1/2}\sum_{i=1}^N E_c(s,\del_{\alpha} v_i)^{1/2},
	\endaligned
	$$ 
	$$
	\aligned
	&\|(s/t)\del_tw_i (P_i^{jk}(t/s)^2v_jF_k)\|_{L^1(\Hcal_s)}
	\\
	\leq& \|(s/t)\del_tw_i\|_{L^2(\Hcal_s)}\left( \|(t/s)^2v_jQ_k^{j'k'}\del_tv_{j'}\del_tv_{k'}\|_{L^2(\Hcal_s)} + \|(t/s)^2v_jR_k\|_{L^2(\Hcal_s)}\right)
	\\
	\leq& C\kappa s^{-2+2\delta}E_{Q,c}(s,v) + C\kappa s^{-1+\delta}E_{Q,c}(s,v)^{1/2}\sum_{j=1}^N\|R_j\|_{L^2(\Hcal_s)},
	\endaligned
	$$
	$$
	\aligned
	&\big\|(s/t)\del_t w_i\ m\big(\del(P_i^{jk}(t/s)^2),\del(v_jv_k)\big)\big\|_{L^1(\Hcal_s)}
	\\
	\leq&\|(s/t)\del_tw_i\|_{L^2(\Hcal_s)}\ \big(\|(s/t)^2\del_t((t/s)^2P_i^{jk})\del_t(v_jv_k)\|_{L^2(\Hcal_s)} 
	+ \|\minu(\del(P_i^{jk}(t/s)^2), \del(v_jv_k))\|_{L^2(\Hcal_s)}\big)
	\\
	\leq&C\kappa s^{-2+\delta}\|(s/t)\del_tw_i\|_{L^2(\Hcal_s)}\sum_{i=1}^N\|v_i\|_{L^2(\Hcal_s)} 
	\\
	&+ \|(s/t)\del_tw_i\|_{L^2(\Hcal_s)}\left(\|\dels ((P_i^{kj}+P_i^{jk})(t/s)^2)v_j\del_tv_k\|_{L^2(\Hcal_s)} 
	+ \|\del_t((t/s)^2P_i^{jk})t^{-1}L_a(v_jv_k)\|_{L^2(\Hcal_s)}\right)
	\\
	\leq& Cs^{-2+2\delta}E_{Q,c}(s,v).
	\endaligned
	$$
	
	So we obtain
	\begin{equation}\label{eq2 pr-normal-form-energy}
	\aligned
	\|(s/t)S_i^{(1)}\|_{L^1(\Hcal_s)}\leq& CE_{Q,c}(s,v)^{1/2}\sum_{i=1}^N\|R_i\|_{L^2(\Hcal_s)}
	\\
	& + C\kappa E_{Q,c}(s,v)^{1/2}\sum_{\alpha,i}\left( s^{-2+2\delta}E_c(s,\del v_i)^{1/2} + s^{-1+\delta}\|R_i\|_{L^2(\Hcal_s)}\right)
	\\
	& + C\kappa s^{-2+2\delta} E_{Q,c}(s,v).
	\endaligned
	\end{equation}
	
	In the same manner,
	\begin{equation}\label{eq3 pr-normal-form-energy}
	\aligned
	\|(s/t)S_i^{(2)}\|_{L^1(\Hcal_s)}
	\leq& C\kappa E_{Q,c}(s,v)^{1/2}\sum_{\alpha,i}\left( s^{-2+2\delta}E_c(s,\del v_i)^{1/2} + s^{-1+\delta}\|R_i\|_{L^2(\Hcal_s)}\right)
	\\
	& + C\kappa s^{-2+2\delta} E_{Q,c}(s,v).
	\endaligned
	\end{equation}
	So combine \eqref{eq1 pr-normal-form-energy} with \eqref{eq2 pr-normal-form-energy} and \eqref{eq3 pr-normal-form-energy} and remark that (thanks to \eqref{eq9 normal-form-energy})
	$$
	E_{Q,c}(s,v)^{1/2}\leq C\sum_{i=1}^NE_c(s,v_i)^{1/2} ,
	$$
	then the desired estimate is proved.
\end{proof}

\subsection{High-order energy estimate on semi-linear Klein-Gordon equation}
In this subsection we will establish a version of high-order estimate on \eqref{eq starting normal-form}, i.e., we will bound the quantity
$$
E_c(s,\del^IL^J v)
$$
via the above modified energy estimate.

We consider the following semi-linear Klein-Gordon equation:
\begin{equation}\label{eq1 27-04-2019}
\Box v + c^2 v = A \del_tv\del_tv + R
\end{equation}
where $A$ and $R$ are regular functions defined in $\Kcal_{[s_0,s_1]}$. This is the equation \eqref{eq ending normal-form} after normal form transform. The idea is to differentiate \eqref{eq1 27-04-2019} with respect to $\del^IL^J$, $|I|+|J|\leq N$. This will leads to a system in the form \eqref{eq-main normal-form-energy} with $v_k = \del^IL^J v$ and then we apply Proposition \ref{prop normal-form-energy}. To do so, we need the following technical preparations.

Remark the following special case of  \eqref{eq 2 comm} :
\begin{equation}\label{eq 1 comm}
[L^J,\del_\alpha] = \sum_{\beta,|J'|<|J|}\Gamma_{\alpha J'}^{J\beta}\del_{\beta}L^{J'}
\end{equation}
with $\Gamma_{\alpha J'}^{J\beta}$ constants and the following identity:
$$
\del_a = \delu_a - (x^a/t)\del_t = -(x^a/t)\del_t + t^{-1}L_a.
$$
So we obtain:

\begin{lemma}\label{lem1 high-order-normal}
Let $v$ be a function defined in $\Kcal_{[s_0,s_1]}$, sufficiently regular. Then
\begin{equation}\label{eq1 lem1 high-order-normal}
\del^IL^J \del_tv = \sum_{|J'|\leq |J|}\Theta_{J'}\del_t\del^IL^{J'}v + t^{-1}\sum_{a,|J'|<|J|}\Gamma_{0J'}^{Ja}L_a\del^IL^{J'}v.
\end{equation} 
where $\Theta_{J'}$ are homogeneous of degree zero.
\end{lemma}

Then we are ready to establish the following result:
\begin{lemma}\label{lem2 high-order-normal}
Let $v$ be a sufficiently regular solution to \eqref{eq1 27-04-2019} in $\Kcal_{[s_0,s_1]}$. Then
\begin{equation}\label{eq1 lem2 high-order-normal}
\aligned
\big(\Box + c^2\big) \del^IL^J v
=&\sum_{I_1+I_2+I_3=I\atop |J_1|+|J_2|+|J_3|\leq|J|} Q^{IJ}_{I_1J_1,I_2J_2}\del_t\del^{I_1}L^{J_1}v\,\del_t\del^{I_2}L^{J_2}v 
\\
& + t^{-1}\mathscr{B}^{IJ}  + \del^IL^JR
\endaligned
\end{equation}
where $\forall I_1+I_2+I_3=I$,
\begin{equation}\label{eq2 lem2 high-order-normal}
Q^{IJ}_{I_1J_1,I_2J_2} = \Theta_{J_1}\Theta_{J_2}\del^{I_3}L^{J_3}A 
\end{equation}
and
$$
\aligned
\mathscr{B}^{IJ} 
 =\sum_{{I_1+I_2+I_3=I\atop J_1+J_2+J_3=J}\atop |J_1'|\leq|J_1|,|J_2'|\leq|J_2|}\del^{I_3}L^{J_3}A\ \bigg(&\Theta_{J_1}\Gamma_{0J_2'}^{J_2a} \del_t\del^{I_1}L^{J_1}v\ L_a\del^{I_2}L^{J_2'}v
 + \Theta_{J_2}\Gamma_{0J_1'}^{J_1a} \del_t\del^{I_2}L^{J_2}v\ L_a\del^{I_1}L^{J_1'}v
\\
 &+ \Gamma_{0J_1'}^{J_1a}\Gamma_{0J_2'}^{J_2b}L_a\del^{I_1}L^{J_1'}v\ \delb_b\del^{I_2}L^{J_2'}v\bigg).
\endaligned 
$$
\end{lemma}
\begin{proof}
Differentiate \eqref{eq1 27-04-2019} with respect to $\del^IL^J$, we obtain
$$
\aligned
\big(\Box + c^2\big)\del^IL^J v = &
\sum_{I_1+I_2+I_3=I\atop J_1+J_2+J_3=J}\del^{I_3}L^{J_3}A\ \del^{I_1}L^{J_1}\del_tv\ \del^{I_2}L^{J_2}\del_t v 
 + \del^IL^J R.
\endaligned
$$
Then substitute \eqref{eq1 lem1 high-order-normal} into the above expression, the desired result is proved.
\end{proof}

Now we apply proposition \ref{prop normal-form-energy} on \eqref{eq1 lem2 high-order-normal}.
\begin{proposition}\label{prop 2 normal-form-energy}
Let $v$ be the regular solution to \eqref{eq1 27-04-2019}. Let $0<\kappa\leq 1$ and $0<\vep_s\ll 1$ be constants. Suppose that $A$ is of the following form:
\begin{equation}\label{eq1 prop 2 normal-form-energy}
A = A_0(s/t)^2 + A_1
\end{equation}
with $A_0$ a constant and $A_1$ a homogeneous function of degree zero. 

Suppose furthermore that for $|I|+|J|\leq N$, $N\in\mathbb{N}$
\begin{equation}\label{eq2 prop 2 normal-form-energy}
|\del^IL^Jv|\leq \vep_s(s/t)^2,
\end{equation}
and
\begin{equation}\label{eq3 prop 2 normal-form-energy}
|L\del^IL^J v| + |\del \del^IL^Jv| + (t/s)|\del^IL^Jv|\leq \kappa (s/t)s^{-1+\delta}.
\end{equation}
Then the following estimate holds:
\begin{equation}
\aligned
\Ecal^N_c(s_1,v)^{1/2}
\leq& C\Ecal^N_c(s_0,v)^{1/2} 
+C \kappa  \int_{s_0}^{s_1}s^{-2+2\delta} \Ecal^{N+1}_c(s,v)^{1/2} ds
\\
&+ C\int_{s_0}^{s_1}\||R|_N\|_{L^2(\Hcal_s)}ds.
\endaligned
\end{equation}
\end{proposition}
\begin{proof}
Consider \eqref{eq1 lem2 high-order-normal} with $|I|+|J|\leq N$. These equations forms a system of semi-linear Klein-Gordon equation of $\del^IL^J v$ in the form of \eqref{eq-main normal-form-energy} where $\del^IL^J v$ take the role of $v_i$ and $Q_i^{jk}$ is replaced by $Q^{IJ}_{I_1J_1I_2J_2}$.

Recall \eqref{eq2 lem2 high-order-normal} combined with \eqref{eq1 prop 2 normal-form-energy} and \eqref{eq2 prop 2 normal-form-energy}, we have the following bounds:
\begin{subequations}
\begin{equation}
|(t/s)^2 Q^{IJ}_{I_1J_1I_2J_2}\del^IL^Jv|+|(t/s)^2 Q^{IJ}_{I_1J_1I_2J_2}\del^{I_1}L^{J_1}v| + |(t/s)^2  Q^{IJ}_{I_1J_1I_2J_2}\del^{I_2}L^{J_2}v|\leq C\vep_s\ll 1,
\end{equation}
\begin{equation}
|(t/s)^2\del^{I_1}L^{J_1}v\del_{\alpha}Q^{IJ}_{I_1J_1I_2J_2}| + |(t/s)^2\del^{I_2}L^{J_2}v\del_{\alpha}Q^{IJ}_{I_1J_1I_2J_2}|\leq\vep_s\ll 1,
\end{equation}
\end{subequations}
i.e., \eqref{eq7 normal-form-energy} and \eqref{eq8 normal-form-energy} are verified. Furthermore, \eqref{eq1 prop normal-form-energy} is guaranteed by \eqref{eq3 prop 2 normal-form-energy}. Direct calculation based on \eqref{eq1 prop 2 normal-form-energy} and \eqref{eq2 lem2 high-order-normal} shows that \eqref{eq2 prop normal-form-energy} holds. Then \eqref{eq3 prop normal-form-energy} is applied. 
Substitute \eqref{eq3 prop 2 normal-form-energy}, we obtain:
\begin{equation}
\| t^{-1}\mathscr{B}^{IJ}\|_{L^2(\Hcal_s)}\lesssim s^{-2+\delta}\Ecal^N_c(s,v)^{1/2}.
\end{equation}
Recall \eqref{eq11 normal-form-energy} guaranteed by  \eqref{eq7 normal-form-energy} and \eqref{eq8 normal-form-energy} which implies the equivalence between the modified energy and the standard energy. Then by \eqref{eq3 prop normal-form-energy}, the desired result is established.

\end{proof}

\subsection{Bounds of $\mathscr{R}$}
Once we have established energy estimate on \eqref{eq ending normal-form}, we need to regard the $L^2$ norm of $\mathscr{R}$. Recall its definition
\eqref{eq R normal-form}. This term is ``good'' in the following sens:
\begin{lemma}\label{lem R normal-form}
	Following the conditions \eqref{eq1 lem main normal-form} and suppose that
	\begin{equation}\label{eq1 lem rest-1 normal-form}
	|\del\del v|_p + |Lv|_p + (t/s)|\del v|_p + (t/s)|v|_p\leq \kappa (s/t)s^{-1+\delta}, \quad \kappa\ll 1.
	\end{equation}
	Then
	\begin{equation}\label{eq1 lem R normal-form}
	\aligned
	\||\mathscr{R}|_p\|_{L^2(\Hcal_s)} 
	\leq& C\kappa s^{-2+2\delta}\Ecal_c^{p+3}(s,v)^{1/2} + C\||R_0|_p\|_{L^2(\Hcal_s)} + C\kappa s^{-1+\delta}\|\del_t R_0\|_{L^2(\Hcal_s)}.
	\endaligned
	\end{equation}
\end{lemma}

\begin{proof}
	
	First, remark that \eqref{eq1 lem main normal-form}  combined with \eqref{eq1 prop (s/t)} leads to
	\begin{subequations}\label{eqs ab normal-form}
		\begin{equation}
		|a|_{p,k}+|b|_{p,k}\leq C(t/s)^2,\quad |\del a|_{p,k} + |\del b|_{p,k}\leq C(t/s)^3s^{-1},\quad |\dels a|_{p,k} + |\dels b|_{p,k}\leq C(t/s)s^{-1},
		\end{equation}
		\begin{equation}
		|\del\del a|_{p,k} + |\del\del b|_{p,k}\leq C(t/s)^4s^{-2},\quad |\Box a|_{p,k} + |\Box b|_{p,k}\leq C(t/s)^2s^{-2}
		\end{equation}
		where $C$ are determined by $p,k$.
	\end{subequations}
	These bounds leads to (combined with \eqref{eq1 lem rest-1 normal-form})
	\begin{equation}
	|a\del_tv|_p + |a v|_p + |bv|_p \leq C\kappa s^{-1+\delta}.
	\end{equation}
	\begin{equation}\label{eq2 pr lem rest-2 normal-form}
	\||av\del_tv + bv^2|_p\|_{L^{\infty}(\Hcal_s)} + \||\del (av\del_tv + bv^2)|_p\|_{L^{\infty}(\Hcal_s)}\leq C\kappa s^{-2+2\delta},
	\end{equation}
	\begin{equation}\label{eq3 pr lem rest-2 normal-form}
	\||\del (av\del_tv + bv^2)|_p\|_{L^2(\Hcal_s)}\leq C\kappa s^{-1+\delta}\Ecal_c^{p+2}(s,v)^{1/2}
	\end{equation}
	and
	\begin{equation}\label{eq1 pr lem rest-2 normal-form}
	\||\del\del(av\del_tv + bv^2)|_p\|_{L^2(\Hcal_s)}\leq C\kappa s^{-1+\delta} \Ecal_c^{p+3}(s,v)^{1/2}.
	\end{equation}
	
	Now for the terms in $\mathscr{R}_1$, we substitute the bounds \eqref{eqs ab normal-form} combined with \eqref{eq2 pr lem rest-2 normal-form}, \eqref{eq3 pr lem rest-2 normal-form}, \eqref{eq1 pr lem rest-2 normal-form} and \eqref{eq1 lem rest-1 normal-form} into its expression. We only need to point out that for the terms
	$$
	2\del_t vm^{\alpha\beta}\del_{\alpha}a\del_{\beta} v,\quad  2vm^{\alpha\beta}\del_{\alpha}a\del_{\beta}\del_tv,\quad  
	4vm^{\alpha\beta}\del_{\alpha}b\del_{\beta}v
	$$
	the null structure should be evoked. For example
	$$
	\aligned
	m^{\alpha\beta}\del_{\alpha}a\del_{\beta} v =& \minu^{00}\del_ta\del_tv + \minu(\del v,\del v)
	\\
	=&(s/t)^2\del_ta\del_tv +t^{-1}\left(\minu^{a0}L_aa\del_tv + \minu^{0a}\del_t a L_a v + \minu^{ab}\delu_aa L_bv\right)
	\endaligned
	$$
	So we obtain
	$$
	\||\del_t vm^{\alpha\beta}\del_{\alpha}a\del_{\beta} v|_p\|_{L^2(\Hcal_s)}\leq C\kappa s^{-2+\delta}\Ecal^{p+1}_c(s, v)^{1/2}.
	$$
	We also remark the term in $\mathscr{R}_1$ concerning $f$:
	$$
	(a\del_tv + bv)f + av\del_t f.
	$$
	Remark in the case of \eqref{eq f starting normal-form}, we have
	$$
	\aligned
	\||(a\del_tv + bv)f + av\del_t f|_p\|_{L^2(\Hcal_s)} \leq& C\kappa s^{-2+2\delta}\Ecal^{p+3}_c(s,v)^{1/2} 
	\\
	&+ C \kappa s^{-1+\delta}\big(\||R_0|_p\|_{L^2(\Hcal_s)}  + \||\del_t R_0|_p\|_{L^2(\Hcal_s)}\big).
	\endaligned
	$$
	For the rest terms in $\mathscr{R}_1$, we omit the detail.
	
	For terms in $\mathscr{R}_2$, remark the following bounds:
	$$
	\sum_{(\alpha,\beta,\gamma)\neq (0,0,0)}|\hu_1^{\alpha\beta\gamma}\delu_{\gamma}v\delu_{\alpha}\delu_{\beta}w|_p + \sum_{(\alpha,\beta)\neq(0,0)}|v\hu_0^{\alpha\beta}\delu_{\alpha}\delu_{\beta}w|_p\leq Ct^{-1}|v|_{p+1}|v|_{p+3}\leq C\kappa s^{-2+\delta}|v|_{p+3}.
	$$
	This is because that in each term there is at least one hyperbolic derivative, and
	$$
	\delu_a = t^{-1}L_a,\quad \del_{\alpha}\delu_a = t^{-1}\del_{\alpha} L_a - t^{-1}\del_{\alpha}t\ \delu_a,\quad \delu_a\del_{\alpha} = t^{-1}L_a\del_{\alpha}.
	$$
	For the same reason:
	$$
	|v\Bu(\del v)|_p + |\Au(\del v,\del v)|_p\lesssim \kappa s^{-2+\delta}|v|_{p+1}.
	$$
	
	For the rest terms in $\mathscr{R}_2$,  we recall \eqref{eq1 pr lem rest-2 normal-form} and the fact that $\del_{\alpha}\big(\Psiu_{\beta}^{\beta'}\big)$ is homogeneous of degree $(-1)$ which supplies additional decay.

	For the terms in $\mathscr{R}_3$, remark that \eqref{eq1 lem rest-1 normal-form} leads to
	$$
	|h[a,v]|_p\leq C\kappa s^{-1+\delta}\leq 1/2
	$$
	thus (thanks to Fa\`a di Bruno's formula)
	\begin{equation}\label{eq4 pr lem rest-2 normal-form}
	\aligned
	&\big|(1+h[a,v])^{-1}\big|_p \leq C,\quad \big|1 - (1+h[a,v])^{-1}\big|_p\leq C\kappa s^{-1+\delta},
	\\
	&\big|1-(1+h[a,v])^{-1} - h[a,v]\big|_p\leq C\kappa s^{-2+2\delta}.
	\endaligned
	\end{equation}
	
	Then substitute the above bounds into the expression of $\mathscr{R}_3$, the desired bound is established.
\end{proof}

\subsection{Normal-form transform: conclusion}
\begin{proposition}\label{prop 3 normal-form}
Let $v$ be a sufficiently regular solution in $\Kcal_{[s_0,s_1]}$ to the following equation: 
\begin{equation}\label{eq main normal-form}
\Box v + (h_0^{\alpha\beta}v + h_1^{\alpha\beta\gamma}\del_{\gamma}v)\del_{\alpha}\del_{\beta}v + c^2v
= A^{\alpha\beta}\del_{\alpha}v\del_{\beta}v + B^{\alpha}v\del_{\alpha}v + Rv^2 + R_0,
\end{equation}
where $h_0,h_1, A,B,R$ are supposed to be constant-coefficient multi-linear forms.  $R_0$ is sufficiently regular. 

Suppose furthermore that
\begin{subequations}
\begin{equation}\label{cond 1 prop 3 normal-form}
|v| + |\del v|\leq \vep_s(s/t)^2
\end{equation}
\begin{equation}\label{cond 2 prop 3 normal-form}
|\del\del v|_N + |Lv|_N + (t/s)|\del v|_N + (t/s)|v|_N\leq \kappa (s/t)s^{-1+\delta}, \quad \kappa\ll 1.
\end{equation}
\end{subequations}
Then
\begin{equation}\label{conc prop 3 normal-form}
\aligned
\Ecal_c^N(s_1,v)^{1/2}\leq& C \Ecal_c^N(s_0,v)^{1/2} 
+ C\kappa \int_{s_0}^{s_1}s^{-2+2\delta}\Ecal_c^{N+3}(s,v)^{1/2} ds 
\\
&+ C\int_{s_0}^{s_1}s^{-1+\delta}\kappa \||\del_t R_0|_N\|_{L^2(\Hcal_s)} + \||R_0|_N\|_{L^2(\Hcal_s)}ds.
\endaligned
\end{equation}
\end{proposition}
\begin{remark}
The main interest of this estimate is to obtain uniform bounds on lower order energy. In right-hand-side a higher order energy appears, however, it is multiplied by a fast decreasing factor.
\end{remark}

\begin{proof}
Recall the calculation made in subsection \ref{subsec normal-form-identity}. \eqref{eq 2.5 normal-form} is guaranteed by \eqref{cond 1 prop 3 normal-form}. So we obtain:
\begin{equation}\label{eq1 pr prop 3 normal-form}
\Box w + c^2w = \left(2(s/t)^2c^{-2}R + 2\hu_0^{00} + \Au^{00}\right)\del_tw\del_tw + \mathscr{R}.
\end{equation}
with 
\begin{equation}\label{eq2 pr prop 3 normal-form}
a = \frac{1}{3c^2}\left(\Bu^0 + c^2(t/s)^2\hu_1^{000}\right),\quad 
b = \frac{1}{c^2}\left(R + c^2(t/s)^2\hu_0^{00}\right),
\end{equation}
and 
$$
w = v + av\del_tv + bv^2.
$$
By \eqref{cond 2 prop 3 normal-form} combined with \eqref{eq2 pr prop 3 normal-form}, 
\begin{equation}
|w| + |\del w|\leq \vep_s (s/t)^2, \vep_s\ll 1,
\end{equation}
and
\begin{equation}
|L\del^IL^J w| + |\del \del^IL^Jw| + (t/s)|\del^IL^Jw|\leq \kappa (s/t)s^{-1+\delta}.
\end{equation}
Now we apply Proposition \ref{prop 2 normal-form-energy} on \eqref{eq1 pr prop 3 normal-form}. \eqref{eq2 prop 2 normal-form-energy} and \eqref{eq3 prop 2 normal-form-energy} are guaranteed by the above bounds. \eqref{eq1 prop 2 normal-form-energy} is verified by the expression.  For the bound of $\mathscr{R}$, recall lemma \ref{lem R normal-form} where \eqref{eq1 lem rest-1 normal-form} is guaranteed by \eqref{cond 2 prop 3 normal-form} and \eqref{eq2 pr prop 3 normal-form}. 
\end{proof}
\section{Other estimates based on semi-hyperboloidal decomposition of wave operator}\label{sec tec-last}
\subsection{Estimates on Hessian form for wave component}
In this section, we concentrate on the estimates on the following terms:
$$
\del_{\alpha}\del_{\beta}Z^Ku,\quad Z^K\del_{\alpha}\del_{\beta}u.
$$
With a bit abuse of notation, we call these terms the {Hessian form} of $u$ of order $|K|$ . Observe that by \eqref{eq 3 lem 1 Hessian}, the only essential component of $\del_{\alpha}\del_{\beta}Z^Ku$ is $\del_t\del_t Z^Ku$. In the following we will give an estimate on this component.

We have the following decomposition of the D'Alembert operator with respect to SHF:
\begin{equation}\label{eq 1 decompo-box-semi}
\Box = (s/t)^2\del_t\del_t + t^{-1}\underbrace{\left((2x^a/t)\del_tL_a - \sum_a\delu_aL_a - (x^a/t)\delu_a + (2+(r/t)^2)\del_t\right)}_{A_m[u]}
\end{equation}
here in $A_m[u]$  in the index $m$ represents the Minkowski metric. We remark that 
\begin{equation}\label{eq 3 decompo-box-semi}
|A_m[u]|\leq C|\del u|_{1,1}.
\end{equation}
Then we establish the following estimate for Hessian components with flat background metric:
\begin{lemma}\label{lem Hessian-flat}
Let $u$ be a function defined in $\Kcal_{[s_0,s_1]}$, sufficiently regular. Then 
\begin{equation}\label{eq1 lem Hessian-flat}
(s/t)^2|\del_\alpha\del_\beta Z^K u| \lesssim |\Box u|_{p,k} + t^{-1}|\del u|_{p+1,k+1}.
\end{equation}
\begin{equation}\label{eq2 lem Hessian-flat}
(s/t)^2|\del\del  u|_{p,k} \lesssim |\Box u|_{p,k} + t^{-1}|\del u|_{p+1,k+1}.
\end{equation}
\end{lemma}
\begin{proof}
Differentiate $\Box u = f$ with respect to $Z^K$ with $K$ of type $(p-k,k,0)$, one obtains:
$$
Z^Kf = \Box Z^Ku = (s/t)^2\del_t\del_tZ^Ku + A_m[Z^K u].
$$
Apply \eqref{eq1 lem 1 notation} (with $m=1$) on $A_m[Z^K u]$, one obtains:
$$
(s/t)^2|\del_t\del_t Z^K u| \lesssim |\Box u|_{p,k} + t^{-1}|\del u|_{p+1,k+1}.
$$
Then recall the relation \eqref{eq 3 lem 1 Hessian}, \eqref{eq1 lem Hessian-flat} is established.

\eqref{eq2 lem Hessian-flat} is direct by \eqref{eq1 lem Hessian-flat} combined with \eqref{eq2 lem 1 notation}.
\end{proof}

\subsection{Fast decay of Klein-Gordon component near light-cone}
In this section we recall the following bound on Klein-Gordon component:
\begin{proposition}\label{prop fast-KG}
Let $v$ be a regular solution to
\begin{equation}\label{eq1 prop fast-KG}
\Box v + c^2 v = f.
\end{equation}
Then 
\begin{equation}\label{eq2 prop fast-KG}
c^2|v|_{p,k}\lesssim (s/t)^2|\del v|_{p+1,k+1} + |f|_{p,k}.
\end{equation}
\end{proposition}
\begin{proof}
Differentiate \eqref{eq1 prop fast-KG} with respect to $Z^I$ with $I$ of type $(p-k,k,0)$
$$
\Box Z^I v + c^2Z^Iv = Z^I f.
$$
Then by \eqref{eq 1 decompo-box-semi},
\begin{equation}
c^2Z^I v = -(s/t)^2\del_t\del_tZ^Iv - t^{-1}A_m[Z^Iv] + Z^If
\end{equation}
And this leads to the desired result (thanks to \eqref{eq1 lem 1 notation})
\end{proof}

\section{Bootstrap argument}\label{sec bootstrap}
\subsection{Bootstrap bounds}
This section is devoted to the proof of theorem \ref{thm 1 weak-coupling}. As explained in introduction, we suppose that on time interval $[2,s_1]$, the following bounds hold:
\begin{equation}\label{eq1 bootstrap}
\Ecal^N(s,u)^{1/2} + \Ecal_c^N(s,v)^{1/2}\leq C_1\vep s^{\delta}.
\end{equation}
\begin{equation}\label{eq2 bootstrap}
\Ecal^{N-4}(s,u)^{1/2} + \Ecal_c^{N-4}(s,v)^{1/2}\leq C_1\vep
\end{equation}
\begin{equation}\label{eq3 bootstrap}
\Econ^{N-4}(s,u)^{1/2}\leq C_1\vep s^{\delta}
\end{equation}
with $0<\delta\leq \frac{1}{100}$ and $N\geq 15$. We will prove, when
\begin{equation}\label{eq choice-of-constants}
 C_1\geq 2C_0,\quad 0\leq \vep<\frac{\delta}{2CC_1}
\end{equation}
where $C = C(N), C_0 = C_0(N)$ are constants determined by $N$,  then the following {\sl improved} energy bounds hold:
\begin{equation}\label{eq1' bootstrap}
\Ecal^N(s,u)^{1/2} + \Ecal_c^N(s,v)^{1/2} < C_1\vep s^{\delta},
\end{equation}
\begin{equation}\label{eq2' bootstrap}
\Ecal^{N-4}(s,u)^{1/2} + \Ecal_c^{N-4}(s,v)^{1/2} < C_1\vep  ,
\end{equation}
\begin{equation}\label{eq3' bootstrap}
\Econ^{N-4}(s,u)^{1/2} < C_1\vep s^{\delta}.
\end{equation}
Then standard bootstrap argument leads to global existence.

For the convenience of expression, we collect the linear terms to be bounded

\begin{subequations}\label{list 1 basic-bounds-bootstrap}
\begin{equation}\label{list 1a basic-bounds-bootstrap}
\|(s/t)|\del u|_p\|_{L^2(\Hcal_s^*)},\quad \||(s/t)\del u|_p\|_{L^2(\Hcal_s^*)},\quad \||\dels u|_p\|_{L^2(\Hcal_s^*)},\quad \|s|\del\dels u|_{p-1}\|_{L^2(\Hcal_s^*)},
\end{equation}
\begin{equation}\label{list 1b basic-bounds-bootstrap}
\|s|\del u|_{p-2}\|_{L^{\infty}(\Hcal_s^*)},\quad \|t|\dels u|_{p-2}\|_{L^{\infty}(\Hcal_s^*)},\quad \|st|\del\dels u|_{p-3}\|_{L^2(\Hcal_s^*)}.
\end{equation}
\end{subequations}
\begin{subequations}\label{list 2 basic-bounds-bootstrap}
\begin{equation}\label{list 2a basic-bounds-bootstrap}
\aligned
&\|(s/t)|\del v|_p\|_{L^2(\Hcal_s^*)},\quad \||(s/t)\del v|_p\|_{L^2(\Hcal_s^*)},\quad \||\dels v|_p\|_{L^2(\Hcal_s^*)},\quad 
\|s|\del\dels v|_{p-1}\|_{L^2(\Hcal_s^*)}
\\
&\||v|_p\|_{L^2(\Hcal_s^*)},\quad \|t|\dels v|_{p-1}\|_{L^2(\Hcal_s^*)},
\endaligned
\end{equation}
\begin{equation}\label{list 2b basic-bounds-bootstrap}
\aligned
&\|s|\del v|_{p-2}\|_{L^{\infty}(\Hcal_s^*)},\quad \|t|\dels v|_{p-2}\|_{L^{\infty}(\Hcal_s^*)},\quad \|st|\del\dels v|_{p-3}\|_{L^2(\Hcal_s^*)}.
\\
&\|t|v|_{p-2}\|_{L^{\infty}(\Hcal_s^*)},\quad \|t^2|\dels v|_{p-3}\|_{L^\infty(\Hcal_s^*)},
\endaligned
\end{equation}
\end{subequations}

\begin{subequations}\label{list 3 basic-bounds-bootstrap}
\begin{equation}\label{list 3a basic-bounds-bootstrap}
\aligned
&
\|(s/t)|u|_{N-4}\|_{L^2(\Hcal_s^*)},\quad \||(s/t)u|_{N-4}\|_{L^2(\Hcal_s^*)},\quad 
\\
&\|s(s/t)^2|\del u|_{N-4}\|_{L^2(\Hcal_s^*)},\quad \|s(s/t)|(s/t)\del u|_{N-4}\|_{L^2(\Hcal_s^*)},
\\
&\|s|\dels u|_{N-4}\|_{L^2(\Hcal_s^*)},
\endaligned
\end{equation}
\begin{equation}\label{list 3b basic-bounds-bootstrap}
\|s|u|_{N-6}\|_{L^{\infty}(\Hcal_s^*)},\quad \|s^2(s/t)|\del u|_{N-6}\|_{L^{\infty}(\Hcal_s^*)},\quad 
\|st|\dels u|_{N-6}\|_{L^{\infty}(\Hcal_s^*)}.
\end{equation}
\end{subequations}

Then based on \eqref{eq1 bootstrap} and apply lemma \ref{lem 3 notation}, we have the following bounds:
\begin{lemma}\label{lem basic-linear}
When $p=N$,	the quantities listed in \eqref{list 1 basic-bounds-bootstrap} and \eqref{list 2 basic-bounds-bootstrap} are bounded by $CC_1\vep 
s^{\delta}$.
\\
When $p=N-4$, the quantities listed in \eqref{list 1 basic-bounds-bootstrap} and \eqref{list 2 basic-bounds-bootstrap} are bounded by $CC_1\vep$.
\\
The quantities listed in \eqref{list 3 basic-bounds-bootstrap} are bounded by $CC_1\delta^{-1}\vep s^{-1 + \delta}$.
\end{lemma}
\begin{proof}
One only needs to remark that \eqref{eq3 bootstrap} leads to 
\begin{equation}\label{eq F_con bootstrap}
\Fcon^{N-4}(s,v)\leq CC_1\delta^{-1}\vep s^{\delta}.
\end{equation}
And this combined with lemma  \ref{lem 3 notation} (list \eqref{list3 lem 3 notation}) leads to the bounds for terms in \eqref{list 3 basic-bounds-bootstrap}.
\end{proof}

For wave component, the decay on $\dels u$ can be improved as following:
$$
\aligned
|\del_r\delu_a& \del^IL^J u| = |t^{-1}(x^b/r)\del_bL_a \del^IL^J u|
\\
\leq& \left\{
\aligned
&CC_1\vep t^{-1}s^{-1+\delta}\sim CC_1\vep t^{-3/2+\delta/2}(t-r)^{-1/2+\delta/2},\quad &&|I|+|J|\leq N-3
\\
&CC_1\vep t^{-1}s^{-1}\sim CC_1\vep t^{-3/2}(t-r)^{-1/2},\quad && |I|+|J|\leq N-7.
\endaligned
\right.
\endaligned
$$
Integrate this bound along radial direction and recall that $\delu_a \del^IL^J u$ vanishes when $r=t-1$, one obtains:
\begin{equation}\label{list 1b'' basic-bounds-bootstrap}
|\dels \del^IL^J u|\leq
\left\{
\aligned
& CC_1\vep (s/t)^2s^{-1+\delta},&&\quad |I|+|J|\leq N-3,
\\
& CC_1\vep (s/t)^2s^{-1},&&\quad |I|+|J|\leq N-7,
\endaligned
\right.
\end{equation}
In the same manner, integrate $\del_r\del^IL^J u$ we obtain
\begin{equation}\label{list 1b''' basic-bounds-bootstrap}
|\del^IL^J u|\leq
\left\{
\aligned
& CC_1\vep (s/t)s^{\delta},&&\quad |I|+|J|\leq N-2,
\\
& CC_1\vep (s/t),&&\quad |I|+|J|\leq N-6,
\endaligned
\right.
\end{equation}
Thus by \eqref{eq3 lem 1 notation},
\begin{equation}\label{list 1b' basic-bounds-bootstrap}
\|t(t/s)|\dels u|_{p-3}\|_{L^{\infty}(\Hcal_s^*)}\leq
\left\{
\aligned
& CC_1\vep s^{\delta},\quad && p=N,
\\
&CC_1\vep, \quad &&p=N-4.
\endaligned
\right.
\end{equation}

Also for wave component, remark that for $|I|+|J| = p$,
$$
|\del^IL^J u|\leq \left\{
\aligned
&C|\del u|_{p-1},\quad  |I|\geq 1,
\\
&Ct|\dels u|_{p-1},\quad |I| = 0, |J|\geq 1.
\endaligned
\right.
$$
Then for $Z=L_a,\del_{\alpha}$,
\begin{equation}\label{eq u bootstrap}
\|t^{-1}|Zu|_p\|_{L^2(\Hcal_s)}\leq 
\left\{
\aligned
&CC_1\vep s^{\delta}, \quad &&p=N,
\\
&CC_1\vep ,\quad &&p=N-4,
\endaligned
\right.
\end{equation}

%
%

\subsection{Basic multi-linear estimates}
We apply lemma \ref{lem multi-linear}, especially \eqref{eq1 bilinear} combined with lemma \ref{lem basic-linear}. For the convenience of expression, we list out the quantities of interest:
\begin{equation}\label{list L2 bilinear}
\aligned
&\||P_1^{\alpha\beta\gamma}\del_{\gamma}u\del_{\alpha}\del_{\beta}u|_{p-1}\|_{L^2(\Hcal_s)},\quad
\||P_3^{\alpha\beta\gamma}\del_{\gamma}v\del_{\alpha}\del_{\beta}u|_{p-1}\|_{L^2(\Hcal_s)},\quad
\||P_4^{\alpha\beta}v\del_{\alpha}\del_{\beta}u|_{p-1}\|_{L^2(\Hcal_s)},
\\
&\||P_5^{\alpha\beta\gamma}\del_{\gamma}u\del_{\alpha}\del_{\beta}v|_{p-1}\|_{L^2(\Hcal_s)},\quad
\||P_7^{\alpha\beta\gamma}\del_{\gamma}v\del_{\alpha}\del_{\beta}v|_{p-1}\|_{L^2(\Hcal_s)},\quad
\||P_8^{\alpha\beta}v\del_{\alpha}\del_{\beta}v|_{p-1}\|_{L^2(\Hcal_s)},
\\
&\||A_1^{\alpha\beta}\del_{\alpha}u\del_{\beta}u|_p\|_{L^2(\Hcal_s)},\quad
\||A_3^{\alpha\beta}\del_{\alpha}u\del_{\beta}v|_p\|_{L^2(\Hcal_s)},\quad
\||A_4^{\alpha}v\del_{\alpha}u|_p\|_{L^2(\Hcal_s)},
\\
&\||A_5^{\alpha\beta}\del_{\alpha}u\del_{\beta}u|_p\|_{L^2(\Hcal_s)},\quad
\||A_7^{\alpha\beta}\del_{\alpha}u\del_{\beta}v|_p\|_{L^2(\Hcal_s)},\quad
\||A_8^{\alpha}v\del_{\alpha}u|_p\|_{L^2(\Hcal_s)},
\\
&\||B_3^{\alpha\beta}\del_{\alpha} v\del_{\beta} v|_p\|_{L^2(\Hcal_s)},\quad
\||B_4^{\alpha}v\del_{\alpha}v|_p\|_{L^2(\Hcal_s)},\quad
\||K_2v^2|_p\|_{L^2(\Hcal_s)},
\endaligned
\end{equation}
\begin{equation}\label{list L2' bilinear}
\||P_2^{\alpha\beta}u\del_{\alpha}\del_{\beta}u|_{p-1}\|_{L^2(\Hcal_s)},\quad 
\||A_6^{\alpha}u\del_{\alpha}u|_{p}\|_{L^2(\Hcal_s)}
\end{equation}
\begin{equation}\label{list decay bootstrap}
\aligned
&\|t|P_1^{\alpha\beta\gamma}\del_{\gamma}u\del_{\alpha}\del_{\beta}u|_{p-3}\|_{L^\infty(\Hcal_s)},\quad
\|t|P_3^{\alpha\beta\gamma}\del_{\gamma}v\del_{\alpha}\del_{\beta}u|_{p-3}\|_{L^\infty(\Hcal_s)},\quad
\|t|P_4^{\alpha\beta}v\del_{\alpha}\del_{\beta}|_{p-3}\|_{L^\infty(\Hcal_s)},
\\
&\|t|P_5^{\alpha\beta\gamma}\del_{\gamma}u\del_{\alpha}\del_{\beta}v|_{p-3}\|_{L^\infty(\Hcal_s)},\quad
\|t|P_7^{\alpha\beta\gamma}\del_{\gamma}v\del_{\alpha}\del_{\beta}v|_{p-3}\|_{L^\infty(\Hcal_s)},\quad
\|t|P_8^{\alpha\beta}v\del_{\alpha}\del_{\beta}v|_{p-3}\|_{L^\infty(\Hcal_s)},
\\
&\|t|A_1^{\alpha\beta}\del_{\alpha}u\del_{\beta}u|_{p-2}\|_{L^\infty(\Hcal_s)},\quad
\|t|A_3^{\alpha\beta}\del_{\alpha}u\del_{\beta}v|_{p-2}\|_{L^\infty(\Hcal_s)},\quad
\|t|A_4^{\alpha}v\del_{\alpha}u|_{p-2}\|_{L^\infty(\Hcal_s)},
\\
&\|t|A_5^{\alpha\beta}\del_{\alpha}u\del_{\beta}u|_{p-2}\|_{L^\infty(\Hcal_s)},\quad
\|t|A_7^{\alpha\beta}\del_{\alpha}u\del_{\beta}v|_{p-2}\|_{L^\infty(\Hcal_s)},\quad
\|t|A_8^{\alpha}v\del_{\alpha}u|_{p-2}\|_{L^\infty(\Hcal_s)},
\\
&\|t|B_3^{\alpha\beta}\del_{\alpha} v\del_{\beta} v|_p\|_{L^2(\Hcal_s)},\quad
\|t|B_4^{\alpha}v\del_{\alpha}v|_p\|_{L^2(\Hcal_s)},\quad
\|t|K_2v^2|_p\|_{L^2(\Hcal_s)}
\endaligned
\end{equation}
\begin{equation}\label{list decay' bilinear}
\|t|P_2^{\alpha\beta}u\del_{\alpha}\del_{\beta}u|_{p-3}\|_{L^\infty(\Hcal_s)},\quad 
\|t|A_6^{\alpha}u\del_{\alpha}u|_{p-2}\|_{L^\infty(\Hcal_s)}
\end{equation}
Then we state the following bounds:
\begin{lemma}\label{lem basic-bootstrap}
Under the assumption of \eqref{eq1 bootstrap} and \eqref{eq2 bootstrap},
$$
\text{Quantities listed in \eqref{list L2 bilinear} and \eqref{list decay bootstrap}}\leq 
\left\{
\aligned
&C(C_1\vep)^2s^{-1+\delta}, \quad &&p = N,
\\
&C(C_1\vep)s^{-1} ,\quad &&p=N-4.
\endaligned
\right.
$$

Under the assumption \eqref{eq1 bootstrap} and \eqref{eq3 bootstrap}, the quantities listed in \eqref{list L2' bilinear} and \eqref{list decay' bilinear} with $p=N-4$ are bounded by $C\delta^{-1}(C_1\vep)^2s^{-1+2\delta}$.

\end{lemma}
\begin{proof}
Consider firstly the terms in \eqref{list L2 bilinear}. For the term $P_1,P_3,P_5,A_1,A_3, A_5$ and $A_7$, we need to evoke their null structure. We only show how to bound $P_1$ for $p=N$, the rest terms are similar. 
\begin{equation}\label{eq example null}
P_1^{\alpha\beta\gamma}\del_{\gamma}u\del_{\alpha}\del_{\beta}u
= \Pu_1^{000}\del_tu\del_t\del_tu 
+ \sum_{(\alpha,\beta,\gamma)\neq(0,0,0)}\Pu_1^{\alpha\beta\gamma}\delu_{\alpha}u\delu_{\beta}\delu_{\gamma}u 
+ P_1^{\alpha\beta\gamma}\del_{\gamma}u\del_{\alpha}(\Psiu_{\beta}^{\beta'})\delu_{\beta'}u.
\end{equation}
For the first term in right-hand-side, remark the null conditions leads to $|\Pu_1^{000}|_N\leq C(s/t)^2$. Then substitute the bounds \eqref{list 1a basic-bounds-bootstrap} (with $p=N$) and \eqref{list 1b basic-bounds-bootstrap} (with $p=N-4$) into \eqref{eq1 bilinear} (Remark that when $N\geq 13$, $[N/2]\leq N-7$). The second term, containing at least one hyperbolic derivative, will have sufficient decay/$L^2$ bounds and can be bounded by \eqref{list 1a basic-bounds-bootstrap} and \eqref{list 1b basic-bounds-bootstrap} (with $p=N$). The last term has additional decreasing factor $\del_{\alpha}(\Psiu_{\beta}^{\beta'})$ which is homogeneous of degree $(-1)$. Thus $P_1$ bounded as desired.

Terms other than the null terms are bounded directly via \eqref{eq1 bilinear}, we omit the detail.

For terms in \eqref{list L2' bilinear} and \eqref{list decay' bilinear}, we need to remark that the terms $P_2$ and $A_6$ are bounded by applying \eqref{eq u bootstrap} combined with \eqref{list 3b basic-bounds-bootstrap} and \eqref{list 1a basic-bounds-bootstrap} combined with \eqref{list 3a basic-bounds-bootstrap}, that is why they have a factor $\delta$ (provided by $\Fcon^{N-4}(s,u)^{1/2}$). Here we show how to bound $P_2$:
$$
P_2^{\alpha\beta}u\del_{\alpha}\del_{\beta}u = \Pu_2^{00}u\del_t\del_t u + \sum_{(\alpha,\beta)\neq(0,0)}\Pu_2^{\alpha\beta}u\delu_{\alpha}\delu_{\beta}u + P_2^{\alpha\beta}u\del_{\alpha}(\Psiu_{\beta}^{\beta'})\delu_{\beta'}u.
$$
For the first term, due to the null condition, 
$$
\aligned
\||\Pu_2^{00}u\del_t\del_t u|_{N}\|_{L^2(\Hcal_s)}\leq& C\|(s/t)t^{-1}|\del_t\del_t u|_{N}\|_{L^2(\Hcal_s)}\|s|u|_{[N/2]}\|_{L^{\infty}(\Hcal_s)}
\\
&+C\|(s/t)|\del_t\del_t u|_{[N/2]}\|_{L^\infty(\Hcal_s)}\|(s/t)|u|_N\|_{L^2(\Hcal_s)}
\\
\leq& C\delta^{-1}(C_1\vep)^2 s^{-1+\delta}.
\endaligned
$$
The second term contains at least one hyperbolic derivative, we apply \eqref{list 1b''' basic-bounds-bootstrap} combined with \eqref{list 1a basic-bounds-bootstrap} or \eqref{list 2b basic-bounds-bootstrap} together with \eqref{eq u bootstrap}. The last term has decreasing factor $\del_{\alpha}(\Psiu_{\beta}^{\beta'})$. We omit the detail.
\end{proof}
\subsection{Bounds on Hessian form of wave component}
In this subsection we will establish the following bounds:
\begin{equation}\label{eq1 Hessian boostrap}
\|s(s/t)^2|\del \del u|_{p-1}\|_{L^2(\Hcal_s)} + \|s^2(s/t)|\del\del u|_{p-3}\|_{L^{\infty}(\Hcal_s)}\leq 
\left\{
\aligned
&CC_1\vep s^{\delta},\quad &&p=N,
\\
&CC_1\vep,\quad &&p=N-4. 
\endaligned
\right.
\end{equation}
This is by lemma \ref{lem Hessian-flat}. We first remark that by lemma \ref{lem basic-bootstrap}, all terms in $F_1$ {\bf except $P_2$} satisfies the following bounds:
\begin{equation}\label{eq3 Hessian boostrap}
\||T|_{p-1}\|_{L^2(\Hcal_s)} + \|t|T|_{p-3}\|_{L^{\infty}(\Hcal_s)}\leq
\left\{
\aligned
&C(C_1\vep)^2 s^{-1+\delta},\quad &&p=N,
\\
&C(C_1\vep)^2s^{-1},\quad &&p=N-4
\endaligned
\right.
\end{equation}
where $T$ represents any term in $F_1$ other than $P_2$.

The only problematic term is $P_2$. We recall the null structure of $P_2$:
$$
P_2^{\alpha\beta}u\del_{\alpha}\del_{\beta}u = \Pu_2^{00}u\del_t\del_tu +\sum_{(\alpha,\beta)\neq(0,0)}\Pu_2^{\alpha\beta}u\delu_{\alpha}\delu_{\beta}u + P_2^{\alpha\beta}u\del_{\alpha}\big(\Psiu_{\beta}^{\beta'}\big)\delu_{\beta'}u
$$
and for the last two terms, thanks to \eqref{list 1b''' basic-bounds-bootstrap}, \eqref{list 1a basic-bounds-bootstrap}, \eqref{list 2b basic-bounds-bootstrap} and \eqref{eq u bootstrap},
\begin{equation}\label{eq4 Hessian boostrap}
\||T|_{p-1}\|_{L^2(\Hcal_s)}+\|t|T|_{p-3}\|_{L^2(\Hcal_s)}\leq 
\left\{
\aligned
&C(C_1\vep)^2 s^{-1+\delta},\quad &&p=N,
\\
&C(C_1\vep)^2s^{-1},\quad &&p=N-4
\endaligned
\right.
\end{equation}
where $T$ represents one of the terms other than the first in right-hand-side. 

Combing \eqref{eq2 lem Hessian-flat} with \eqref{eq3 Hessian boostrap} and \eqref{eq4 Hessian boostrap}, we  obtain
\begin{equation}\label{eq5 Hessian boostrap}
\|(s/t)^2|\del\del u|_{p-1}\|_{L^2(\Hcal_s)} \leq C\||\Pu_2^{00}u\del_t\del_tu|_{p-1}\|_{L^2(\Hcal_s)} + 
\left\{
\aligned
&C(C_1\vep)^2 s^{-1+\delta},\quad &&p=N,
\\
&C(C_1\vep)^2s^{-1},\quad &&p=N-4
\endaligned
\right.
\end{equation}
\begin{equation}\label{eq6 Hessian boostrap}
(s/t)^2|\del\del u|_{p-3}\leq 
C|\Pu_2^{00}u\del_t\del_tu|_{p-3}
+
\left\{
\aligned
&C(C_1\vep)^2 (s/t)s^{-2+\delta},\quad &&p=N,
\\
&C(C_1\vep)^2(s/t)s^{-2},\quad &&p=N-4.
\endaligned
\right.
\end{equation}

We will first establish the $L^{\infty}$ bound. To do so, remark that in \eqref{eq6 Hessian boostrap} for $p\leq N$,
$$
\aligned
|\Pu_2^{00}u\del_t\del_tu|_{p-3}\leq& C(s/t)^2\sum_{0\leq p_1\leq N-6}|u|_{p_1}|\del\del u|_{p-p_1-3} + C(s/t)^2\sum_{N-5\leq p_1\leq p-3}|u|_{p_1}|\del\del u|_{p-p_1-3}
\\
\leq &CC_1\vep (s/t)^2|\del\del u|_{p-3}+ C(s/t)^2s^{\delta}|\del\del u|_{3}.
\endaligned
$$
where \eqref{list 1b''' basic-bounds-bootstrap} is applied. The last term does not exist if $N-6>p-3\Leftrightarrow p<N-2$. When $C_1\vep\ll 1$
\eqref{eq6 Hessian boostrap} together with the above bound leads to
$$
(s/t)^2|\del\del u|_{p-3}\leq 
\left\{
\aligned
&C(C_1\vep)^2 (s/t)s^{-2+\delta} + C(s/t)^2s^{\delta}|\del\del u|_{3},\quad &&p=N,
\\
&C(C_1\vep)^2(s/t)s^{-2},\quad &&p=N-4.
\endaligned
\right.
$$
So we conclude by ($3\leq N-4$)
\begin{equation}\label{eq7 Hessian boostrap}
(s/t)^2|\del\del u|_{p-3}\leq 
\left\{
\aligned
&C(C_1\vep)^2 (s/t)s^{-2+\delta} ,\quad &&p=N,
\\
&C(C_1\vep)^2(s/t)s^{-2},\quad &&p=N-4.
\endaligned
\right.
\end{equation}

For the $L^2$ bounds, remark that
$$
\aligned
&\||\Pu_2^{00}u\del_t\del_tu|_{p-1}\|_{L^2(\Hcal_s)}
\\
\leq& C\|(s/t)^2|u|_{N-6}|\del\del u|_{p-1}\|_{L^2(\Hcal_s)} + C\sum_{N-5\leq|I'|\leq p-1}\|(s/t)^2|\del\del u|_{p+4-N}|Z^{I'}u|\|_{L^2(\Hcal_s)}
\\
\leq& CC_1\vep \|(s/t)^2|\del\del u|_{p-1}\|_{L^2(\Hcal_s)} + C\sum_{N-5\leq|I'|\leq p-1}\|(s/t)^2|\del\del u|_4|Z^{I'}u|\|_{L^2(\Hcal_s)}
\\
\leq& CC_1\vep \|(s/t)^2|\del\del u|_{p-1}\|_{L^2(\Hcal_s)} + CC_1\vep s^{-1}\sum_{N-5\leq|I'|\leq p-1}\|t^{-1}|Z^{I'}u|\|_{L^2(\Hcal_s)}
\\
\leq& CC_1\vep \|(s/t)^2|\del\del u|_{p-1}\|_{L^2(\Hcal_s)} + 
\left\{
\aligned
&CC_1\vep s^{-1+\delta},\quad N-2\leq |I'|\leq p-1,
\\
&CC_1\vep s^{-1},\quad N-5\leq |I'|\leq N-3
\endaligned
\right.
\endaligned
$$
where in the third inequality \eqref{eq7 Hessian boostrap} is applied on $|\del\del u|_4$ (recall that $N-7\geq 4$) and in the last inequality \eqref{eq u bootstrap} on $|Z^{I'}u|$. Remark that when $N-2>p-1\Leftrightarrow p<N-1$, 
$$
\||\Pu_2^{00}u\del_t\del_tu|_{p-1}\|_{L^2(\Hcal_s)}
\leq CC_1\vep \|(s/t)^2|\del\del u|_{p-1}\|_{L^2(\Hcal_s)} + 
\left\{
\aligned
&CC_1\vep s^{-1+\delta},\quad N-1\leq p \leq N, 
\\
&CC_1\vep s^{-1},\quad p\leq N-2.
\endaligned
\right.
$$
This combined with \eqref{eq5 Hessian boostrap} (and suppose that $C_1\vep \ll 1$) leads to 
\begin{equation}
\|(s/t)^2|\del\del u|_{p-1}\|_{L^2(\Hcal_s)}\leq \left\{
\aligned
&C(C_1\vep)^2 s^{-1+\delta},\quad &&p=N,
\\
&C(C_1\vep)^2s^{-1},\quad &&p=N-4
\endaligned
\right.
\end{equation}

Thus \eqref{eq1 Hessian boostrap} is established. 


\subsection{Improved energy bound for KG component: lower order}\label{subsec kg-low}
\subsubsection{objective}
This section is devoted to the following improved energy bound:
\begin{equation}\label{eq1 kg-low}
\Ecal_c^{N-4}(s,v)^{1/2}\leq C_0\vep + C\delta^{-1}(C_1\vep)^2
\end{equation}
where $C_0$ is a constant determined by $N$.
\eqref{eq1 kg-low} is proved by Proposition \ref{prop 3 normal-form}.  The following section is devoted to the verification of \eqref{cond 1 prop 3 normal-form} and \eqref{cond 2 prop 3 normal-form}, and estimates on $\del_t R_0$ and $R_0$  (according to the notation of Proposition \ref{prop 3 normal-form}). 
\subsubsection{Fast decay of KG component near light-cone}
First, we need to guarantee \eqref{cond 1 prop 3 normal-form} and \eqref{cond 2 prop 3 normal-form}. In fact we will prove that
\begin{equation}\label{eq1 fast-decay-kg}
|\del \del v|_{N-4} + |L v|_{N-3} + (t/s)|\del v|_{N-4} + (t/s)|v|_{N-3} \leq CC_1\vep(s/t) s^{-1+\delta}.
\end{equation} 
The bound on first two terms are included in \eqref{list 2b basic-bounds-bootstrap}. The bounds on last two terms are guaranteed by
\begin{equation}\label{eq2 fast-decay-kg}
|v|_{N-3}\leq CC_1\vep (s/t)^2s^{-1+\delta}.
\end{equation}

This is done by application of Proposition \ref{prop fast-KG}. From \eqref{list 2b basic-bounds-bootstrap}
$$
(s/t)^2|\del v|_{N-2}\leq CC_1\vep (s/t)^2s^{-1+\delta}.
$$
Then we need to bound $F_2(\del u,u,\del\del v,\del v,v)$  (who take the role of $f$, following the notation of Proposition \ref{prop fast-KG}). This is concluded in the following lemma:
\begin{lemma}\label{lem 1 energy-kg-low}
Under the assumption \eqref{eq1 bootstrap} and \eqref{eq2 bootstrap},
\begin{equation}\label{eq1 kg-fast-bootstrap}
|F_2|_{N-3}\leq C(C_1\vep)^2(s/t)^2s^{-1+\delta}.
\end{equation}
\end{lemma}
\begin{proof}
This is by substitution of the bounds in \eqref{list 1b basic-bounds-bootstrap}, \eqref{list 1b'' basic-bounds-bootstrap} and \eqref{list 2b basic-bounds-bootstrap} into the expression. Among these terms we pay special attention to $P_5, A_5,A_6$ and $A_7$, which null terms and their structure need to be evoked.  

We first write the bound on $P_7^{\alpha\beta\gamma}\del_{\gamma}v\del_{\alpha}\del_{\beta}v$ as an example. For this term we need to remark that 
$$
(t/s)|\del v|_{p-3} +  |\del\del v|_{p-3}
\leq
\left\{
\aligned
&CC_1\vep s^{-1+\delta},\quad &&p=N,
\\
&CC_1\vep s^{-1} ,\quad &&p=N-4.
\endaligned
\right. 
$$
Thus
$$
\aligned
|P_7^{\alpha\beta\gamma}\del_{\gamma}v\del_{\alpha}\del_{\beta}v|_{N-3}
\leq& 
C|\del v|_{[(N-3)/2]}|\del\del v|_{N-3} + C|\del v|_{N-3}|\del\del v|_{[(N-3)/2]}
\\
\leq&C(C_1\vep)^2 (s/t)s^{-2+\delta}\leq C(C_1\vep)^2 (s/t)^2s^{-1+\delta} 
\endaligned
$$
where we have remark the relation $s^{-1}\leq (s/t)$ in $\Kcal_{[s_0,s_1]}$.

For null terms, take $P_5$ as example:
$$
P_5^{\alpha\beta\gamma}\del_{\gamma}u\del_{\alpha}\del_{\beta}v = \Pu_5^{000}\del_tu\del_t\del_tv
 + \sum_{(\alpha,\beta,\gamma)\neq(0,0,0)}\Pu_5^{\alpha\beta\gamma}\delu_{\gamma}u\delu_{\alpha}\delu_{\beta}v 
 + P_5^{\alpha\beta\gamma}\del_{\gamma}u\del_{\alpha}(\Psiu_{\beta}^{\beta'})\delu_{\beta'}v.
$$
Due to the null condition, $\Pu_5^{000} = \Lambda(s/t)^2$ with $\Lambda$ homogeneous of degree zero. In the second term of right-hand-side, there is at least one hyperbolic derivative. In the last term the factor $\del_{\alpha}(\Psiu_{\beta}^{\beta'})$ is homogeneous of degree $(-1)$. Taking these into consideration rather than substituting na\"ively the bounds of $\del u$ and $\del\del v$, we obtain (with one factor bounded by bounds in with $p=N$ and the other bounded by those with $p=N-4$)
$$
|P_5^{\alpha\beta\gamma}\del_{\gamma}u\del_{\alpha}\del_{\beta}v|_{N-3}\leq CC_1\vep (s/t)^2s^{-2+\delta}.
$$
\end{proof}

Then we conclude by \eqref{eq2 fast-decay-kg}. 
\subsubsection{$L^2$ bounds on $R_0$}
In this subsection we show how to bound $R_0$ and $\del_t R_0$ (according to the notation of Proposition \ref{prop 3 normal-form}). A fist result is
\begin{lemma}\label{lem 2 energy-kg-low}
Following the notation of proposition \ref{prop 3 normal-form} and assume that \eqref{eq1 bootstrap}, \eqref{eq2 bootstrap} and \eqref{eq3 bootstrap} hold, then
\begin{equation}\label{eq1 lem 2 energy-kg-low}
s^{-1}\||\del_t R_0|_{N-4}\|_{L^2(\Hcal_s)} + \||R_0|_{N-4}\|_{L^2(\Hcal_s)} \leq C\delta^{-1}(C_1\vep)^2s^{-2+2\delta}
\end{equation}
where 
$$
R_0 =  P_5^{\alpha\beta\gamma}\del_{\gamma}u\del_{\alpha}\del_{\beta}v  + (A_5^{\alpha\beta}\del_{\beta}u + A_6^{\alpha}u +  A_7^{\alpha\beta}\del_{\beta}v + A_8^{\alpha}v)\del_{\alpha}u.
$$
\end{lemma}
\begin{proof}
The bound on $R_0$ is by bilinear estimate \eqref{eq1 bilinear} combined with the bounds \eqref{list 3a basic-bounds-bootstrap}, \eqref{list 3b basic-bounds-bootstrap} and \eqref{eq1 fast-decay-kg}. We need to evoke the null structure of $P_5^{\alpha\beta}$, $A_5^{\alpha\beta}$ and $A_6^{\alpha}$ exactly as in \eqref{eq example null} and below. 

For the bound on $\del_t R_0$ a similar discussion based on  \eqref{list 3a basic-bounds-bootstrap}, \eqref{list 3b basic-bounds-bootstrap}, \eqref{list 2b basic-bounds-bootstrap}, \eqref{list 1b''' basic-bounds-bootstrap}, \eqref{eq5 Hessian boostrap} and \eqref{eq1 fast-decay-kg} leads to the desired bound.

\end{proof}

Now substitute \eqref{eq1 lem 2 energy-kg-low} into \eqref{conc prop 3 normal-form}, remark that the initial energy $\Ecal_c^{N-4}(s_0,v)^{1/2}$ is bounded by $C_0\vep$ with $C_0$ a constant determined only by $N$. Then \eqref{eq1 kg-low} is established.

\subsection{Improved energy bound for wave component: low order}
\subsubsection{Objective}
In this subsection we will establish the following bound:
\begin{equation}\label{eq5 conform-bootstrap}
\Ecal^{N-4}(s,u)^{1/2} \leq C_0\vep  + C\delta^{-1}(C_1\vep)^2.
\end{equation}
This is by energy estimate Proposition \ref{prop 1 energy} applied on 
$$
\Box Z^Iu = Z^IF_1, \quad I\text{ of type }(p,k,0).
$$ 
We only need to establish the following bound:
\begin{equation}\label{eq2 conform-bootstrap}
\||F_1|_{N-4}\|_{L^2(\Hcal_s)}\leq C\delta^{-1}(C_1\vep)^2s^{-1+2\delta}. 
\end{equation}
It is done in the next subsubsection.
\subsubsection{Bound on $\||F_1|_{N-4}\|_{L^2(\Hcal_s)}$}
For the convenience of discussion, we denote by 
$$
F_1 = f_1 + P_2^{\alpha\beta}u\del_{\alpha}\del_{\beta}u
$$
With $f_1$ all terms except $P_2$. Then recall lemma \ref{lem basic-bootstrap}, 
\begin{equation}\label{eq f_1-low}
\||f_1|_{N-4}\|_{L^2}\leq C(C_1\vep)^2 s^{-2+\delta}.
\end{equation}

However, the bound on $P_2$ can not be bounded as $f_1$. We do null decomposition:
$$
P_2^{\alpha\beta}u\del_{\alpha}\del_{\beta}
= \Pu_2^{00}u\del_t\del_tu 
+ \sum_{(\alpha,\beta)\neq(0,0)}\Pu^{\alpha\beta}u\delu_\alpha\delu_\beta u + P_2^{\alpha\beta}u\del_{\alpha}(\Psiu_{\beta}^{\beta'})\delu_{\beta'}u
$$
The last two terms can be bounded by $C(C_1\delta^{-1}\vep)^2s^{-2+\delta}$, while the fist term is bounded as following: 
$$
\aligned
\||\Pu_2^{00}u\del_t\del_tu|_{N-4}\|\leq&
C\||(s/t) u|_{N-4}\|_{L^2(\Hcal_s)}\||(s/t)\del_t\del_tu|_{[(N-4)/2]}\|_{L^{\infty}(\Hcal_s)} 
\\
&+ C\|s^{-1}|u|_{[(N-4)/2]}\|_{L^{\infty}(\Hcal_s)}\ \|s(s/t)^2|\del_t\del_tu|_{N-4}\|_{L^2(\Hcal_s)} 
\\
\leq & C\delta^{-1}(C_1\vep)^2s^{-2+2\delta}
\endaligned
$$
where on $|\del\del u|$ we have applied \eqref{eq1 Hessian boostrap} and on $|u|_{[(N-4)/2]}$ we have applied \eqref{list 3b basic-bounds-bootstrap}.
We thus obtain 
\begin{equation}\label{eq3 conform-bootstrap}
\||F_1|_{N-4}\|_{L^2(\Hcal_s)}\leq C\delta^{-1}(C_1\vep)^2s^{-2+2\delta}.
\end{equation}
Take this bound and apply \eqref{ineq 3 prop 1 energy}, we obtain \eqref{eq5 conform-bootstrap}, where we remark that \eqref{ineq 1 propo 1 energy} and \eqref{ineq 2 prop 1 energy} holds automatically with $\kappa =1$.

\subsection{Improved conformal energy bound}
\subsubsection{Objective}
In this subsection we will establish the following bound:
\begin{equation}\label{eq1 conform-bootstrap}
\Econ^{N-4}(s,u)^{1/2}\leq C_0\vep + C\delta^{-1}(C_1\vep)^2s^{\delta}.
\end{equation}
In order to establish this bound, we write the wave equation in \eqref{eq 1 main} into the following form:
\begin{equation}\label{eq10 conform-bootstrap}
\big(\Box - P_2^{\alpha\beta}u\del_{\alpha}\del_{\beta}\big)u = f_1.
\end{equation}
Then differentiate this equation with respect to $\del^IL^J$, we obtain:
\begin{equation}\label{eq11 conform-bootstrap}
\big(\Box - P_2^{\alpha\beta}\del_{\alpha}\del_{\beta}\big) \del^IL^J u = [\del^IL^J, P_2^{\alpha\beta}u\del_{\alpha}\del_{\beta}]u + \del^IL^J f_1.
\end{equation}
Then we apply \eqref{eq3 conformal}. \eqref{eq f_1-low} supplies sufficient $L^2$ bound on $f_1$. We only need to verify \eqref{eq1 lem 1 conformal} and then give a sufficient $L^2$ bounds on $ [\del^IL^J, P_2^{\alpha\beta}u\del_{\alpha}\del_{\beta}]u$.
\subsubsection{Verification of \eqref{eq1 lem 1 conformal}}
Remark that in our case, $h^{\alpha\beta} = P_2^{\alpha\beta}u$. Then thanks to \eqref{list 1b basic-bounds-bootstrap} \eqref{list 1b''' basic-bounds-bootstrap} and \eqref{eq1 Hessian boostrap} (with $p=N-4$), 
\begin{equation}\label{eq12 conform-bootstrap}
|\del(P_2^{\alpha\beta}u)|\leq CC_1\vep s^{-1},\quad |P_2^{\alpha\beta}u|\leq CC_1\vep(s/t).
\end{equation}


Furthermore, recall the null condition satisfied by $P_2$, 
$$
\hb^{00} = \bar{P}_2^{00}u = (t/s)^2\Pu_2^{00}u
$$
where $(t/s)^2\Pu_2^{00}$ is homogeneous of degree zero. Then \eqref{eq1 lem 1 conformal} is verified.

\begin{remark}
	Remark that in \eqref{eq12 conform-bootstrap}, $CC_1\vep$ takes the role of $\vep_s$ in \eqref{eq1 lem 1 conformal}.
\end{remark}

\subsubsection{Bounds on commutator}
In this subsubsection we establish the following bounds:
\begin{equation}\label{eq13 conform-bootstrap}
\| [\del^IL^J, P_2^{\alpha\beta}u\del_{\alpha}\del_{\beta}]u\|_{L^2(\Hcal_s)}\leq C\delta^{-1}(C_1\vep)^2s^{-2+\delta}.
\end{equation}
To do so, we rely on Proposition \ref{prop commu}. It is clear that by \eqref{list 3a basic-bounds-bootstrap} and \eqref{list 3b basic-bounds-bootstrap}:
\begin{equation}\label{eq14 conform-bootstrap}
\|(s/t)|P^{\alpha\beta}u|_{N-4}\|_{L^2(\Hcal_s)} + \|s|P^{\alpha\beta}u|_{N-6}\|_{L^{\infty}(\Hcal_s)}\leq C\delta^{-1}C_1\vep s^{\delta}.
\end{equation}
Recall that $P_2$ is a null quadratic form, thus
\begin{equation}\label{eq15 conform-bootstrap}
\|(t/s)|\Pu^{00}u|_{N-4}\|_{L^2(\Hcal_s)} + \|s(t/s)^2|\Pu^{00}u|_{N-6}\|_{L^{\infty}(\Hcal_s)}\leq C\delta^{-1}C_1\vep s^{\delta}.
\end{equation}
Null recall Proposition \ref{prop commu}, apply the above bounds together with \eqref{eq1 Hessian boostrap} (with $p=N-4$) and \eqref{list 3 basic-bounds-bootstrap} on the first two terms in right-hand-side of \eqref{eq 1 prop commu}, and \eqref{list 1b''' basic-bounds-bootstrap}, \eqref{eq u bootstrap} together with \eqref{list 1 basic-bounds-bootstrap}(with $p=N-4$) on the last two terms. Then we obtain \eqref{eq13 conform-bootstrap}.

Now apply \eqref{eq3 conformal} together with  \eqref{eq f_1-low} and \eqref{eq13 conform-bootstrap} (remark that $CC_1\vep$ takes the role of $\vep_s$ therein), \eqref{eq1 conform-bootstrap} is proved.


%
%

\subsection{Improved energy bounds: high-order}
This subsection is devoted to the final step: improved energy estimates for high-order:
\begin{equation}\label{improved energy-high}
\Ecal^N(s,u)^{1/2} + \Ecal_c^N(s,v)^{1/2}\leq C_0\vep + C\delta^{-1}(C_1\vep)^2s^{\delta}.
\end{equation}
We differentiate \eqref{eq 1 main} with respect to $\del^IL^J$ and obtain:
\begin{equation}\label{eq1 energy-high}
\aligned
&\Box \del^IL^J u - \mathcal{P}_w^{\alpha\beta}\del_{\alpha}\del_{\beta}\del^IL^J u = [\del^IL^J,\mathcal{P}_w^{\alpha\beta}\del_{\alpha}\del_{\beta}]u + \del^IL^J\left(\mathcal{A}_{w}^{\alpha}\del_{\alpha}u\right)
\\
&\Box \del^IL^J v -  \mathcal{P}_{kg}^{\alpha\beta}\del_{\alpha}\del_{\beta}\del^IL^J v + c^2\del^IL^Jv = [\del^IL^J,\mathcal{P}_{kg}^{\alpha\beta}\del_{\alpha}\del_{\beta}]u
 + \del^IL^J\left(\mathcal{A}_{kg}^{\alpha}\del_{\alpha}u + \mathcal{B}_{kg}^{\alpha}\del_{\alpha}v + K_2v^2\right)
\endaligned
\end{equation}
and then apply Proposition \ref{prop 1 energy}. To do so, it is sufficient to guarantee \eqref{ineq 1 propo 1 energy} and \eqref{ineq 2 prop 1 energy} and give sufficient bonds on source terms. The following subsubsections are devoted to these.

\subsubsection{Verification of  \eqref{ineq 1 propo 1 energy} and \eqref{ineq 2 prop 1 energy}}
Remark that these two conditions are posed on the quasilinear part of the system.
\\ 
We first concentrate on \eqref{ineq 1 propo 1 energy}.
Suppose that we can prove:
\begin{equation}\label{eq 1 curved-enerngy}
(s/t)^2\big(\big|\mathcal{\Pu}_w^{00}\big| + \big|\mathcal{\Pu}_{kg}^{00}\big|\big)
 + \big| \mathcal{\Pu}_w^{ab} \big| + \big| \mathcal{\Pu}_{kg}^{ab} \big| \leq \kappa \ll 1 . 
\end{equation}
Let $w$ be a sufficiently regular function defined on $\Kcal_{[s_0,s_1]}$. Taking the difference of $E_{g,c}(s,w)$ and $E_c(s,w)$, one has:
\begin{equation}\label{eq 2 curved-enerngy}
\aligned
\big|E_{g,c}(s,w) - E_c(s,w)\big|
\leq &  \int_{\Hcal_s} \big|\mathcal{P}_{w}^{00}|\del_tw|^2 - \mathcal{P}_w^{ab}\del_aw\del_bw - \sum_a(2x^a/t)\mathcal{P}_w^{a\beta}\del_tw\del_{\beta}w\big|\ dx
 \\
= & \int_{\Hcal_s} \big|\mathcal{\Pu}_w^{00}|\del_tw|^2 - \mathcal{\Pu}_w^{ab}\delu_aw\delu_bw\big|\ dx.
\\
\leq&  C\kappa \int_{\Hcal_s} \big|(s/t)^2\del_tw\big|^2 + \sum_a|\delu_a w|^2\ dx\leq C\kappa E(s,w).
\endaligned
\end{equation}
which leads to \eqref{ineq 1 propo 1 energy}.

Then we concentrate on \eqref{eq 1 curved-enerngy}. We will only show haw to bound $\mathcal{P}_w$ and omit the bound on $\mathcal{P}_{kg}$ which is similar. Recall the expression of $\mathcal{P}_w$ and the bound \eqref{eq1 fast-decay-kg}, $P_3$ and $P_4$ are easily bounded. For $P_2$, the null condition leads to $|\Pu_2^{00}|\leq C(s/t)^2$, and then  recall \eqref{list 1b''' basic-bounds-bootstrap}. For $P_1$, the $00$ component is written as 
$$
\Pu^{00\gamma}\delu_{\gamma}u = \Pu^{000}\del_tu + \Pu^{00c}\delu_c u.
$$
Also by null condition, $|\Pu^{000}\del_tu|\leq C(s/t)^2C_1\vep$. Recall \eqref{list 1b' basic-bounds-bootstrap} for the second term. Then $|\mathcal{\Pu}_{w}^{00}|$ and $\mathcal{\Pu}_w^{ab}$ are correctly bounded as in \eqref{eq 1 curved-enerngy}.

The verification of \eqref{ineq 2 prop 1 energy} is similar. We will prove that 
\begin{equation}
\|(s/t)\del_\mu(\mathcal{P}_w^{\alpha\beta})\del_{\alpha}w\del_{\beta}w\|_{L^1(\Hcal_s)} + \|(s/t)\del_\mu(\mathcal{P}_{kg}^{\alpha\beta})\del_{\alpha}w\del_{\beta}w\|_{L^1(\Hcal_s)} \leq CC_1\vep s^{-1}E(s,w).
\end{equation}
We will only write the estimate on $\del_\mu\big(\mathcal{P}_w^{\alpha\beta}\big)\del_{\alpha}w\del_{\beta}w$. Recall the expression of $\mathcal{P}_w^{\alpha\beta}$, we need to bound $P_1,P_2,P_3,P_4$. In $P_3$ and $P_4$, due to the bound \eqref{list 2b basic-bounds-bootstrap} with $p=N-4$, 
$$
\big| P_3^{\alpha\beta\gamma}\del_\mu\del_{\gamma}v\big| + |P_4^{\alpha\beta}\del_{\mu}v|\leq CC_1t^{-1}\sim CC_1\vep (s/t)s^{-1}
$$
For the term $P_1$ and $P_2$, we need to evoke their null structure: 
$$
\aligned
&\del_\mu\big(P_1^{\alpha\beta\gamma}\del_{\gamma}u\big)\del_{\alpha}w\del_{\beta}w
\\
 =&P_1^{\alpha\beta\gamma}\del_{\gamma}\del_\mu u\del_{\alpha}w\del_{\beta}w
= \Pu_1^{000}\del_t\del_\mu u\del_tw\del_tw 
+ \sum_{(\alpha,\beta,\gamma)\neq(0,0,0)}\Pu_1^{\alpha\beta\gamma}\delu_{\gamma}\del_\mu u\delu_{\alpha}w\delu_{\beta}w 
\endaligned
$$
Then 
$$
|\Pu_1^{000}\del_t\del_\mu u|\leq C(s/t)^2 s^{-1},
$$
and this leads to
$$
\|(s/t) \Pu_1^{000}\del_t\del_\mu u\del_tw\del_tw\|_{L^1(\Hcal_s)}\leq CC_1\vep s^{-1}E(s,w).
$$
$$
\Pu_1^{\alpha\beta\gamma}\delu_{\gamma}\del_\mu u\leq
\left\{
\aligned
&CC_1\vep(s/t)s^{-1},\quad \gamma>0,
\\
&CC_1\vep s^{-1},\quad \gamma = 0.
\endaligned
\right.
$$
And this leads to
$$
\sum_{(\alpha,\beta,\gamma)\neq(0,0,0)}\|\Pu_1^{\alpha\beta\gamma}\delu_{\gamma}\del_\mu u\delu_{\alpha}w\delu_{\beta}w\|_{L^1(\Hcal_s)}
\leq CC_1\vep s^{-1}E(s,w).
$$
The verification on $P_2$ is similar, we omit the detail.
\subsubsection{Bounds on source terms}
Recall lemma \ref{lem basic-bootstrap}, all semilinear terms in  $F_1$ and $F_2$ (i.e., $\mathcal{A}_w, \mathcal{A}_{kg},\mathcal{B}_{kg}$ and $v^2$) are bounded as following:

\begin{equation}\label{source-semilinear}
\||\mathcal{A}_w^{\alpha}\del_{\alpha}u|_N\|_{L^2(\Hcal_s)} + \||\mathcal{A}_{kg}^{\alpha}\del_{\alpha}u|_N\|_{L^2(\Hcal_s)}
 + \||\mathcal{B}_{kg}^\alpha\del_{\alpha}v|_N\|_{L^2(\Hcal_s)} + \||v^2|_N\|_{L^2(\Hcal_s)}\leq C(C_1\vep)^2s^{-1+\delta}.
\end{equation}
\begin{subequations}
The analysis on commutators is based on Proposition \ref{prop commu}. We will prove the following bounds:
\begin{equation}\label{L2 source-quasilinear}
\|(t/s^2)|\mathcal{\Pu}_w^{00}|_N\|_{L^2(\Hcal_s)} + \|(s/t)|\mathcal{\Pu}_{kg}^{00}|_N\|_{L^2(\Hcal_s)} \leq C(C_1\vep)^2 s^{\delta},
\end{equation}
\begin{equation}\label{decay source-quasilinear}
\|(t/s)^2|\mathcal{\Pu}_w^{00}|_{N-7}\|_{L^{\infty}(\Hcal_s)} + \|t|\mathcal{\Pu}_{kg}^{00}|_{N-7}\|_{L^{\infty}(\Hcal_s)}\leq C(C_1\vep)^2.
\end{equation}
\end{subequations}
The terms other than $P_2$ are bounded directly by \eqref{list 1a basic-bounds-bootstrap} with $p=N$ and \eqref{list 1b basic-bounds-bootstrap} with $p=N-4$ while $P_2^{00} u$ is bounded by \eqref{eq u bootstrap} and \eqref{list 1b''' basic-bounds-bootstrap}.

In the same manner, the following bounds hold:
\begin{subequations}
\begin{equation}\label{L2' source-quasilinear}
\|t^{-1}|\mathcal{P}_w|_N\|_{L^2(\Hcal_s)} + \|st^{-2}|\mathcal{P}_{kg}|_N\|_{L^2(\Hcal_s)} \leq C(C_1\vep)^2 s^{\delta},
\end{equation}
\begin{equation}\label{decay' source-quasilinear}
\||\mathcal{P}_w|_{N-7}\|_{L^{\infty}(\Hcal_s)} + \||\mathcal{P}_{kg}|_{N-7}\|_{L^{\infty}(\Hcal_s)}\leq C(C_1\vep)^2.
\end{equation}
\end{subequations}

\begin{subequations}
Now we are ready to bound the commutator for wave equation. By Proposition \ref{prop commu}:
\begin{equation}\label{commutator-w}
\aligned
\|\big|[\del^IL^J,\mathcal{P}_w^{\alpha\beta}\del_{\alpha}\del_{\beta}]u\big|_N\|_{L^2(\Hcal_s)}
\leq& \||\mathcal{\Pu}_w^{00}|_{[N/2]}|\del\del u|_{N-1}\|_{L^2(\Hcal_s)} + \||\mathcal{\Pu}_w^{00}|_N|\del\del u|_{[N/2]}\|_{L^2(\Hcal_s)}
\\
&+\|t^{-1}|\mathcal{\Pu}_w|_{[N/2]}|\del u|_N\|_{L^2(\Hcal_s)} + \|t^{-1}|\mathcal{P}_w|_{N}|\del u|_{[N/2]}\|_{L^2(\Hcal_s)}
\\
\leq& C(C_1\vep)^2s^{-1+\delta}.
\endaligned
\end{equation}

In the same manner, we can establish the same bound for Klein-Gordon equation:
\begin{equation}\label{commutator-kg}
\|[\del^IL^J,\mathcal{P}_{kg}^{\alpha\beta}\del_{\alpha}\del_{\beta}]v\|\leq C(C_1\vep)^2s^{\delta}
\end{equation}
by applying the following bounds:
$$
|\del\del v|_{N-8}\leq CC_1\vep t^{-1},\quad \|(s/t)|\del\del v|_{N-1}\|_{L^2(\Hcal_s)}\leq CC_1\vep s^{-1+\delta}.
$$ 
where the first is due to \eqref{list 2b basic-bounds-bootstrap} for $p=N-4$.
\end{subequations}

Now, substitute \eqref{source-semilinear}, \eqref{commutator-w} and \eqref{commutator-kg} into \eqref{ineq 3 prop 1 energy}, \eqref{improved energy-high} is verified.
\subsection{Conclusion of bootstrap argument}
Now, recalling \eqref{eq1 kg-low}, \eqref{eq5 conform-bootstrap}, \eqref{eq1 conform-bootstrap} and \eqref{improved energy-high}, we only need to make the following choice:
\begin{equation}
C_0< \frac{C_1}{2}, \quad  0<\delta<\frac{1}{100}, \quad \vep<\frac{\delta}{2CC_1}
\end{equation}
and 
\begin{equation}
CC_1\vep\ll 1
\end{equation}
where $C$ is a constant determined by $N$. 
Then 
\begin{equation}
C_0\vep + C\delta^{-1}(C_1\vep)^2<C_1\vep
\end{equation}
which leads to \eqref{eq1' bootstrap}, \eqref{eq2' bootstrap}, \eqref{eq3' bootstrap}.

\begin{appendix}
\section{A Sketch on the basic results of hyperboloidal foliation framework}\label{App Basic}
\subsection{Weak Leibniz rule and Fa\`a di Bruno's formula}
The following two results are not sharp but enough for our analysis. Their proof is by induction, we omit the detail.

\begin{lemma}[Weak Leibniz Rules]\label{lem 1 Leinbiz}
	If $u_k$ are functions defined in $\Kcal$, sufficiently regular, then
	\begin{equation}\label{eq 1' Leibniz}
	Z^I\big(u_1\cdot u_2\cdots u_m\big)
	\end{equation}
	is a finite linear combination (with constant coefficients determined by $I$) of the terms
	$$
	Z^{I_1}u_1\cdot Z^{I_2}u_2\cdots Z^{I_m}u_m
	$$
	where  $I_n$ is of type $(i_n,j_n,k_n)$ and $I$ is of type $(i,j,k)$ with
	\begin{equation}\label{eq 1 condition-compatible}
	(i,j,k) = \sum_{n=1}^m(i_n,j_n,k_n).
	\end{equation}

	Furthermore
	\begin{equation}\label{eq 2' Leibniz}
	\del^IL^J(u_1\cdot u_2\cdots u_m\big)
	\end{equation}
	is a finite linear combination (with constant coefficients determined by $I,J$) of the terms
	$$
	\del^{I_1}L^{J_1}u_1\del^{I_2}L^{J_2}u_2\cdots \del^{I_m}L^{J_m}u_m
	$$
	with
	\begin{equation}\label{eq 2 condition-compatible}
	\sum_{n=1}^m|I_n| = |I|,\quad \sum_{n=1}^m|J_n| = |J|.
	\end{equation}
\end{lemma}

\begin{lemma}[Weak F\`aa di Bruno's formula]\label{lem 1 faa}
	Let $u$ be a function defined in $\Kcal$, sufficiently regular. Let $f$ be a $C^{\infty}$ function defined on an open interval $(a,b)$ of $\RR$ which contains the image of $u$. Then $Z^I(f(u)) $ is a finite linear combination of the following terms (with constant coefficients determined by $I$):
	\begin{equation}\label{eq 1' faa}
	f^{(k)}(u)Z^{I_1}uZ^{I_2}u\cdots Z^{I_k}u
	\end{equation}
	where $1\leq k\leq |I|$, $I_n$ is of type $(i_n,j_n,k_n)$ and $I$ is of type $(i,j,k)$ with
	\begin{equation}\label{eq 2' faa}
	(i,j,k) = \sum_{n=1}^m(i_n,j_n,k_n), \quad i_n+j_n+k_n\geq 1.
	\end{equation}
	Furthermore, $\del^IL^J(f(u))$ is a finite linear combination of
	\begin{equation}\label{eq 3' faa}
	f^{(k)}(u)\ \del^{I_1}L^{J_1}u\ \del^{I_2}L^{J_2}u\cdots \del^{I_k}L^{J_k}u
	\end{equation}
	with
	\begin{equation}\label{eq 4' faa}
	1\leq k\leq |I|+|J|,\quad \sum_{i=1}^k|I_i| = |I|,\quad \sum_{i=1}^k |J_i| = |J|,\quad |I_i|+|J_i|\geq 1.
	\end{equation}
\end{lemma}

For the convenience of expression,  we denote by
\begin{equation}\label{eq 1 Leibniz}
Z^I(u_1\cdot u_2\cdots u_m) = \sum_{I_1+I_2+\cdots +I_m = I}\!\!\!\! Z^{I_1}u_1\cdot Z^{I_2}u_2\cdots Z^{I_m}u_m,
\end{equation}
\begin{equation}\label{eq 2 Leibniz}
\del^IL^J (u_1\cdot u_2\cdots u_m) = \sum_{I_1+\cdots I_m=I\atop J_1+\cdots +J_m = J}\!\!\!\!\del^{I_1}L^{J_1}u_1\del^{I_2}L^{J_2}u_2\cdots \del^{I_m}L^{J_m}u_m
\end{equation}
and
\begin{equation}\label{eq 1 faa}
Z^I(f(u)) = \sum_{k=1}^{|I|}f^{(k)}(u)\sum_{I_1+\cdots +I_k \seq I}Z^{I_1}uZ^{I_2}u\cdots Z^{I_k}u
\end{equation}
\begin{equation}\label{eq 2 faa}
\del^IL^J(f(u)) = \sum_{k=1}^{|I|+|J|}f^{(k)}(u)\sum_{I_1+\cdots +I_k \seq I\atop J_1+\cdots+J_k\seq J}\del^{I_1}L^{J_1}u\ \del^{I_2}L^{J_2}u\cdots \del^{I_k}L^{J_k}u
\end{equation}
for the fact that the left-hand-side being finite linear combinations of the terms in right-hand-side with the conditions \eqref{eq 1 condition-compatible}, \eqref{eq 2 condition-compatible} or \eqref{eq 2' faa}, \eqref{eq 4' faa}.

\subsection{Ordering lemma of high-order derivative}
The main result of this subsection is the following lemma, which shows that a high-order derivative $Z^I$ can be written in a ``standard'' form. 
\begin{lemma}[Decomposition of high-order derivative]\label{lem 2 high-order}
Let $u$ be a function defined in $\Kcal_{[s_0,s_1]}$, sufficiently regular. Let $Z^K$ be a $N-$order operator of type $(i,j,k)$ and $j+k\geq 1$. Then the following bound holds:
\begin{equation}\label{eq 1 lem 2 high-order}
Z^Ku = \sum_{|I|\leq i, |J|\leq j+k\atop |I|+|J|\geq 1}t^{-k-i+|I|}\Delta_{IJ}^K\del^IL^Ju
\end{equation}
with $\Delta_{IJ}^K$ homogeneous functions of degree zero.
\end{lemma}

Before prove this, we state the following special case:

\begin{lemma}\label{lem 1 high-order}
	Let $u$ be a function defined in $\Kcal_{[s_0,s_1]}$, sufficiently regular. Let $Z^K$ be a $N-$order operator of type $(i,j,0)$. Then the following bound holds:
	\begin{equation}\label{eq 1 lem 1 high-order}
	Z^Ku = \sum_{|I|=i\atop|J|\leq j}\Gamma^K_{IJ}\del^IL^Ju
	\end{equation}
	with $\Gamma_{IJ}^K$ constants determined by $K$ and $I,J$.
\end{lemma}
\begin{proof}[Sketch of proof]
We need the following relation:
	\begin{equation}\label{eq 2 comm}
[L^J,\del^I] = \sum_{|I'|=|I|\atop|J'|<|J|}\Gamma_{I'J'}^{JI}\del^{I'}L^{J'}
\end{equation}
where $\Gamma_{\alpha J'}^{J\,\beta}$ and $\Gamma_{I'J'}^{JI}$ are constants. This is firstly proved in \cite{LM1} and can be observed easily by making induction on $(I,J)$ (to get start, verify the case $|I|=|J|=1$). 

Then let $K$ be of type $(i,j,0)$, then it can be written as
$$
Z^K = \del^{I_1}L^{J_1}\del^{I_2}L^{J_2}\cdots \del^{I_r}L^{J_r}
$$
where $|I_1|$ and $|J_r|$ may be zero. Then apply \eqref{eq 2 comm} :
$$
\del^{I_1}L^{J_1}\del^{I_2}L^{J_2}\cdots \del^{I_r}L^{J_r} = \underbrace{\del^{I_1}\del^{I_2}}_{\del^{I_1'}}\underbrace{L^{J_1}L^{J_2}}_{L^{J_1'}}\del^{I_3}L^{J_3}\cdots \del^{I_r}L^{J_r} + \del^{I_1}\big([L^{J_1},\del^{I_2}]\del^{I_3}\cdots \del^{I_r}L^{J_r}\big).
$$
Then by induction on $r$, one can obtain the desired result.
\end{proof}

\begin{proof}[Proof of lemma \ref{lem 2 high-order}]
When $k=0$, we apply \eqref{eq 1 lem 1 high-order}. 
	
Suppose that $k\geq 1$, then we proceed by induction on $k$. Suppose that \eqref{eq 1 lem 2 high-order} holds for $k\leq k_0$. Let $Z^K$ be of type $(i,j,k)$ with $k = k_0+1$. Suppose that $K = (k_1,k_2,\cdots k_m,\cdots ,k_N)$ with
$$
k_1,k_2,\cdots k_{m-1} \in \{0,1,2,3,4,\}, \quad k_m,k_{m+1},\cdots k_{N}\in\{5,6\}.
$$	
In another word, $Z_{k_m}$ is the first hyperbolic derivative in $Z^K$. We denote by $\delu_a = Z_{k_m}$. Then
$$
Z^Ku = Z^{K_1}\delu_aZ^{K_2}u
$$
with $Z^{K_1}$ being $(i_1,j_1,0)$ and $Z^{K_2}$ being $(i_2,j_2,k_0)$ with $i_1+i_2 = i, j_1+j_2=j$. Then
\begin{equation}\label{eq 1 pr lem 2 high-order}
Z^{K_1}\delu_aZ^{K_2}u = Z^{K_1}\big(t^{-1}L_aZ^{K_2}\big)u = \sum_{K_{11}+K_{12}=K_1}\!\!\!\!Z^{K_{11}}t^{-1}\cdot Z^{K_{12}}L_aZ^{K_2}u.
\end{equation}
Suppose that $K_{11}$ is of type $(i_{11},j_{11},0)$ and $K_{12}$ is of type $(i_{12},j_{12},0)$ with $i_{11}+i_{12} = i_1$ and $j_{11}+j_{12} = j_1$. Denote by $Z^{K_{11}'} = Z^{K_{12}}L_aZ^{K_2}$ and remark that $Z^{K_{11}'}$ is of type $(i_{11}',j_{11}',k_0)$ with	
$$
i_{11}' = i_{12} + i_2,\quad j_{11}' = j_{12} + j_2 + 1.
$$
Then $i_{11}' + j_{11}' + k_0\geq 1$.  Then by the assumption of induction:
$$
Z^{K_{12}}L_aZ^{K_2}u = Z^{K_{11}'}u = \sum_{|I|\leq i'_{11},|J|\leq j_{11}' + 1 + k_0\atop |I|+|J|\geq 1}t^{-k_0 -i'_{11} + |I|}\Delta_{IJ}^{K_{11}'}\del^IL^Ju
$$
	
On the other hand, by the homogeneity of $t^{-1}$:
$$
|Z^{K_{11}}t^{-1}|\leq t^{-1-i_{11}}\theta
$$
where $\theta$ is a homogeneous function of degree zero. So for each term in right-hand-side of \eqref{eq 1 pr lem 2 high-order},
$$
\aligned
Z^{K_{11}}t^{-1}\cdot Z^{K_{12}}L_aZ^{K_2}u
=& \theta\sum_{|I|\leq i'_{11}|J|\leq j_{11}' + 1 + k_0\atop|I|+|J|\geq 1}\Delta_{IJ}^{K_{11}'}t^{-k_0-1-(i_{11}+i_{11}')+|I|}\del^IL^Ju
\\
=&\sum_{|I|\leq i'_{11} |J|\leq j_{11}' + k\atop |I|+|J|\geq 1}\theta\Delta_{IJ}^{K_{11}'}t^{-k-i+|I|}\del^IL^Ju
\endaligned
$$
and we remark that $\theta\Delta_{IJ}^{K_{11}'}$ are homogeneous functions of degree zero. Now we take the sum over $K_{11}+K_{12}=K_1$, and see that the case for $k = k_0+1$ is guaranteed (here remark that a sum of finite homogeneous functions of degree zero is again homogeneous of degree zero).
\end{proof}

\subsection{Sketch of Proof for proposition \ref{prop (s/t)}}\label{subsec (s/t)-appendix}
\begin{lemma}\label{lem 1 s/t}
	In the region $\Kcal$, the following decompositions hold:
	\begin{equation}\label{eq 1 lem 1 s/t}
	L^J(s/t) = \Lambda^J(s/t),\quad \del^I(s/t) = \sum_{k=1}^{|I|}\Lambda_k^I(s/t)^{1-2k}
	\end{equation}
	with $\Lambda^J$ homogeneous of degree zero, $\Lambda_k^I$ homogeneous of degree $-|I|$. Furthermore,
	\begin{equation}\label{eq 2 lem 1 s/t}
	\big|\del^IL^J(s/t)\big|\leq \left\{
	\aligned
	&C(s/t), \quad &&|I|=0,
	\\
	&Cs^{-1},\quad &&|I|>0
	\endaligned
	\right.
	\end{equation}
	with $C$ a constant determined by $I,J$.
\end{lemma}
\begin{proof}
	The first decomposition in \eqref{eq 1 lem 1 s/t} is by induction. We just remark that
	$$
	L_a(s/t) = \frac{-x^a}{t}(s/t)
	$$
	where $(-x^a/t)$ is homogeneous of degree zero.
	
	For the second decomposition of \eqref{eq 1 lem 1 s/t}, we recall the Fa\`a di Bruno's formula and take $u = s^2/t^2 = (1-r^2/t^2)$ and
	$$
	\aligned
	f:\RR^+&\rightarrow \RR
	\\
	x&\rightarrow x^{1/2}.
	\endaligned
	$$
	Then
	$$
	\del^I(s/t) = \sum_{k=1}^{|I|}\sum_{I_1+\cdots+I_k\seq I}C_ku^{-k+1/2}\cdot \del^{I_1}u\del^{I_2}u\cdots \del^{I_k}u.
	$$
	Also recall that $(1-r^2/t^2)$ is homogeneous of degree zero, $\del^{I_1}u \del^{I_2}u \cdots \del^{I_k}u$ is homogeneous of degree $-|I|$. So the desired decomposition is established.
	
	Furthermore, recall proposition \ref{prop 1 homo} (the last point) and the fact that in $\Kcal$, $s\leq t\leq s^2$,
	$$
	\del^I(s/t)\leq C\sum_{k=1}^{|I|}(s/t)^{1-2k}t^{-|I|}\leq Cs^{-1}(t/s^2)^{|I|-1}\leq Cs^{-1}.
	$$
	Then by \eqref{eq 1 lem 1 s/t},
	$$
	\del^IL^J(s/t) = \del^I\big(\Lambda^J(s/t)\big) = \sum_{I_1+I_2=I}\del^{I_1}L^{J_1}\Lambda^J\cdot \del^{I_2}L^{J_2}(s/t).
	$$
	Recall the homogeneity of $\Lambda^{J}$, \eqref{eq 2 lem 1 s/t} is proved.
\end{proof}

Then we prove the following results:
\begin{lemma}\label{lem 3 s/t}
	In the region $\Kcal$, the following bounds hold for $k,l\in\mathbb{Z}$:
	\begin{equation}\label{eq 1 lem 3 s/t}
	\big|\del^IL^J\big((s/t)^kt^l\big)\big|\leq
	\left\{
	\aligned
	&C(s/t)^kt^l,\quad &&|I|=0,
	\\
	&C(s/t)^kt^l(t/s^2),\quad &&|I|\geq 1.
	\endaligned
	\right.
	\end{equation}
\end{lemma}
\begin{proof}
	We first establish the following bound, for $n\in\mathbb{Z}$:
	\begin{equation}\label{eq 2' lem 1 s/t}
	\big|\del^IL^J\big((s/t)^n\big)\big|\leq \left\{
	\aligned
	&C(s/t)^n, \quad &&|I|=0,
	\\
	&C(s/t)^n(t/s^2),\quad &&|I|\geq 0.
	\endaligned
	\right.
	\end{equation}
	When $n\in \mathbb{N}$, this is based on \eqref{eq 2 lem 1 s/t} combined with the weak Leibniz rule.
	
	Then consider $(s/t)^{-n}$. This is also by Fa\`a di Bruno's formula. We denote by $u = (s/t)$ and
	$$
	\aligned
	f: \RR^+ &\rightarrow \RR
	\\
	x&\rightarrow x^{-n}
	\endaligned
	$$
	We denote by $Z^{I'} = \del^IL^J$. Then $Z^{I'}$ is of type $(i,j,0)$ with $i=|I|$ and $j = |J|$. Then
	$$
	\del^IL^J\big((s/t)^{-n}\big) = Z^{I'}(f(u)) =
	\sum_{k=1}^{|I|+|J|}\sum_{I'_1+\cdots I'_k \seq I'}f^{(k)}(u)\cdot Z^{I_1'}(s/t)\cdots Z^{I_k'}(s/t).
	$$
	Here
	$$
	Z^{I'_l} = \del^{I_l}L^{J_l},\quad 1\leq l\leq k.
	$$
	Then by \eqref{eq 2 lem 1 s/t}: suppose that among $\{I_1,I_2\cdots I_k\}$ there are $i_0$ indices of positive order . Then when $i\geq 1$, there are  at least one index with order $\geq 1$. Then
	$$
	\big|f^{(k)}(u)\cdot\del^{I_1}L^{J_1}(s/t)\cdots \del^{I_k}L^{J_k}(s/t)\big|\leq C_n(s/t)^{-n-k}\cdot (s/t)^{k-i_0}s^{-i_0} = C(s/t)^{-n-i_0}s^{-i_0}.
	$$
	Recall that $s^{-1}\leq s/t$, then the bound on $\del^IL^J\big((s/t)^{-n}\big)$ is established.
	
	Now for \eqref{eq 1 lem 3 s/t}, remark that
	$$
	\del^IL^J\big((s/t)^kt^l\big) = \sum_{I_1+I_2=I\atop J_1+J_2=J}\del^{I_1}L^{J_1}(s/t)^k\cdot \del^{I_2}L^{J_2}t^l.
	$$
	Then apply \eqref{eq 2' lem 1 s/t} and the homogeneity of $t^l$, the desired result is established.
\end{proof}

Now proposition \ref{prop (s/t)} is direct by combining \eqref{eq 1 lem 2 high-order} and \eqref{eq 1 lem 3 s/t}.

\subsection{Estimates of high-order derivatives}\label{subsec esti-high-order-app}
Recall the following notation:
$$
\Fcon^N(s_0;s,u)^{1/2}:=\sum_{|I|+|J|\leq N}\Fc(s_0;s,\del^IL^Ju)^{1/2}.
$$
We also recall $\Ecal^N(s,u)$ and $\Ecal_c^N(s,u)$ in \eqref{eq 3 01-01-2019} and \eqref{eq 4 01-01-2019}.
\begin{proposition}\label{lem 1 esti-high}
	Let $u$ be a function defined in $\Kcal_{[s_0,s_1]}$, sufficiently regular. Let $Z^K$ be a operator of type $(i,j,k)$, and let $|K|=N+1\geq 1$. Then the following bounds hold:
	\begin{equation}\label{eq 0 lem 1 esti-high}
	\|t^{k-1}Z^K u\|_{L^2(\Hcal^*_s)}\leq C\Ecal^N(s,u)^{1/2},\quad i=0,
	\end{equation}
	\begin{equation}\label{eq 1 lem 1 esti-high}
	\|(s/t)t^k Z^Ku\|_{L^2(\Hcal^*_s)}\leq C\Ecal^N(s,u)^{1/2},\quad i\geq 1.
	\end{equation}
	When $c>0$, the following bound holds for $|K| = N\geq 0$:
	\begin{equation}\label{eq 3 lem 1 esti-high}
	\|ct^k Z^K u\|_{L^2(\Hcal^*_s)}\leq C\Ecal^N_c(s,u)^{1/2}.
	\end{equation}
	
	Let $J$ be a multi-index of type $(i,j,k)$ with $|J| = N \geq 1$,
	\begin{equation}\label{eq 0' lem 1 esti-high}
	\|(s/t)t^kZ^Ju\|_{L^2(\Hcal_s^*)}\leq C\Fcon^{N-1}(s,u),\quad i=0, 
	\end{equation}
	\begin{equation}\label{eq 1' lem 1 esti-high}
	\|(s/t)^3t^{k+1}Z^Ju\|_{L^2(\Hcal_s^*)}\leq C\Fcon^{N-1}(s,u),\quad i\geq 1,
	\end{equation}
	and when $|J|=0$,
	\begin{equation}\label{eq 2' lem 1 esti-high}
	\|(s/t)u\|_{L^2(\Hcal_s^*)}\leq C\Fc(s,u).
	\end{equation}
\end{proposition}
\begin{proof}
	\eqref{eq 0 lem 1 esti-high} is direct by \eqref{eq 1 lem 2 high-order}. To see this let us consider
	$$
	t^{-k-i+|I|}\Delta_{IJ}^K\del^IL^Ju,\quad |I|+|J|\geq 1.
	$$
	
	Recall that $|I|=i=0$, then $|J|\geq 1$. We denote by $L^J = L_aL^{J'}$. Then (recall $i\geq 0$)
	$$
	\aligned
	\|t^{k-1}\big(t^{-k-i+|I|}\Delta_{IJ}^K\del^IL^Ju\big)\|_{L^2(\Hcal^*_s)} \leq& C \|t^{-1}L_aL^{J'}u\|_{L^2(\Hcal^*_s)}
	= C\|\delu_aL^{J'}u\|_{L^2(\Hcal^*_s)}
	\\
	\leq& CE(s,L^{J'}u)^{1/2}\leq C\Ecal^{N-1}(s,u)^{1/2}.
	\endaligned
	$$
	

	For \eqref{eq 1 lem 1 esti-high}, remark that in this case $i\geq 1$. By \eqref{eq 1 lem 2 high-order}, we consider
	$$
	t^{-k-i+|I|}\Delta_{IJ}^K\del^IL^Ju,\quad |I|+|J|\geq 1.
	$$
	As in discussion on \eqref{eq 0 lem 1 esti-high}, when $|I|\geq 1$, we denote by $\del^I = \del_{\alpha}\del^{I'}$. Then (recall that $i\geq|I|$)
	$$
	\aligned
	\|t^k(s/t)\cdot t^{-k-i+|I|}\Delta_{IJ}^K\del^IL^Ju\|_{\Hcal^*_s}
	\leq& \|t^{-i+|I|}(s/t)\del_{\alpha}\del^{I'}L^Ju\|_{\Hcal^*_s}
	\\
	\leq& CE(s,\del^{I'}L^Ju)^{1/2}\leq C\Ecal^{N-1}(s,u)^{1/2}.
	\endaligned
	$$
	
	When $|I|=0$, then $|J|\geq 1$. We denote by $L^J = L_aL^{J'}$. Then (recall $i\geq 1$)
	$$
	\aligned
	\|t^k(s/t)\big(t^{-k-i+|I|}\Delta_{IJ}^K\del^IL^Ju\big)\|_{L^2(\Hcal^*_s)} \leq& C \|t^{-i}L_aL^{J'}u\|_{L^2(\Hcal^*_s)}
	= C\|t^{-i+1}\delu_aL^{J'}u\|_{L^2(\Hcal^*_s)}
	\\
	\leq& CE(s,L^{J'}u)^{1/2}\leq C\Ecal^{N-1}(s,u)^{1/2}.
	\endaligned
	$$
	
	\eqref{eq 3 lem 1 esti-high} is direct by \eqref{eq 1 lem 2 high-order} and the expression of the energy, we omit the detail.
	
	For the bounds \eqref{eq 0' lem 1 esti-high}, \eqref{eq 1' lem 1 esti-high} and \eqref{eq 2' lem 1 esti-high}, we combine proposition \ref{proposition 2 01-01-2019} and \eqref{eq 1 lem 2 high-order}, we omit the detail.
\end{proof}

The following result is to be combined Klainerman-Sobolev inequality in order to establish decay estimates. 
\begin{lemma}\label{lem 2 esti-high}
	Let $u$ be a function defined in $\Kcal_{[s_0,s_1]}$, sufficiently regular.  Let $|I_0|+|J_0|\leq 2$, then the following bounds hold for $Z^K$ of type $(i,j,k)$ with $1\leq |K|\leq N-1$:
	\begin{equation}\label{eq 0 lem 2 esti-high}
	\big\|\del^{I_0}L^{J_0}\big(t^{k-1}Z^Ku\big)\big\|_{L^2(\Hcal^*_s)}\leq C\Ecal^N(s,u)^{1/2},\quad i=0
	\end{equation}
	\begin{equation}\label{eq 1 lem 2 esti-high}
	\big\|\del^{I_0}L^{J_0}\big(t^k(s/t)Z^Ku\big)\big\|_{L^2(\Hcal^*_s)}\leq C\Ecal^N(s,t)^{1/2},\quad i\geq 1.
	\end{equation}
	When $c>0$ and $|K|\leq N - 2$,
	\begin{equation}\label{eq 3 lem 2 esti-high}
	\big\|c\del^{I_0}L^{J_0}\big(t^kZ^Ku\big)\big\|_{L^2(\Hcal^*_s)}\leq C\Ecal_c^N(s,t)^{1/2}.
	\end{equation}
	
	Let $J$ be a multi-index of type $(i,j,k)$ with $|J| = N \geq 1$,
	\begin{equation}\label{eq 0' lem 2 esti-high}
	\big\|\del^{I_0}L^{J_0}\big((s/t)t^kZ^Ju\big)\big\|_{L^2(\Hcal_s^*)}\leq C\Fcon^{N-1}(s,u),\quad i=0, 
	\end{equation}
	\begin{equation}\label{eq 1' lem 2 esti-high}
	\big\|\del^{I_0}L^{J_0}\big((s/t)^3t^{k+1}Z^Ju\big)\big\|_{L^2(\Hcal_s^*)}\leq C\Fcon^{N-1}(s,u),\quad i\geq 1,
	\end{equation}
	and when $|J|=0$,
	\begin{equation}\label{eq 2' lem 2 esti-high}
	\big\|\del^{I_0}L^{J_0}\big((s/t)u\big)\big\|_{L^2(\Hcal_s^*)}\leq C\Fc(s,u).
	\end{equation}
\end{lemma}
\begin{proof}
	These are by proposition \ref{lem 1 esti-high} and the following calculation. Recall \eqref{eq 1 lem 3 s/t} and the fact that $(t/s^2)\leq C$ in $\Kcal$. Then
	$$
	\del^{I_0}L^{J_0}\big(t^{k-1}Z^Ku\big) 
	= \sum_{I_{01}+I_{02}=I_0\atop J_{01}+J_{02}=J_0}\del^{I_{01}}L^{J_{01}}t^{k-1}\cdot\del^{I_{02}}L^{J_{02}}Z^Ku
	$$
	Then each term in right-hand-side, we apply \eqref{eq 1 lem 3 s/t} on the first factor. For second factor, remark that
	$$
	\del^{I_{02}}L^{J_{02}}Z^K
	$$
	is of order $\leq N+2$. Then by proposition \ref{lem 1 esti-high}, the above bounds are established.
	
	\eqref{eq 3 lem 2 esti-high} are established in the same manner, we omit the detail.
\end{proof}

Then, based on this lemma, we can establish the following $L^{\infty}$ bounds via global Sobolev's inequality (proposition \ref{prop Global-Sobolev})
\begin{proposition}\label{lem 3 esti-high}
	Let $u$ be a function defined in $\Kcal_{[s_0,s_1]}$, sufficiently regular. then the following bounds hold for $Z^K$ of type $(i,j,k)$ with $1\leq |K|\leq N-1$:
	\begin{equation}\label{eq 0 lem 3 esti-high}
	\|t^kZ^Ku\|_{L^\infty(\Hcal^*_s)}\leq C\Ecal^N(s,u)^{1/2},\quad i=0
	\end{equation}
	\begin{equation}\label{eq 1 lem 3 esti-high}
	\|(s/t)t^{k+1}Z^Ku\|_{L^\infty(\Hcal^*_s)}\leq C\Ecal^N(s,t)^{1/2},\quad i\geq 1.
	\end{equation}
	When $c>0$ and $|K|\leq N - 2$,
	\begin{equation}\label{eq 3 lem 3 esti-high}
	\|ct^{k+1}Z^Ku\|_{L^2(\Hcal^*_s)}\leq C\Ecal_c^N(s,t)^{1/2}.
	\end{equation}
	
	Let $J$ be a multi-index of type $(i,j,k)$ with $|J| = N \geq 1$,
	\begin{equation}\label{eq 0' lem 3 esti-high}
	\|(s/t)t^{k+1}Z^Ju\|_{L^2(\Hcal_s^*)}\leq C\Fcon^{N-1}(s,u),\quad i=0, 
	\end{equation}
	\begin{equation}\label{eq 1' lem 3 esti-high}
	\|(s/t)^3t^{k+2}Z^Ju\|_{L^2(\Hcal_s^*)}\leq C\Fcon^{N-1}(s,u),\quad i\geq 1,
	\end{equation}
	and when $|J|=0$,
	\begin{equation}\label{eq 2' lem 3 esti-high}
	\|su\|_{L^\infty(\Hcal_s^*)}\leq C\Fc(s,u).
	\end{equation}
	
\end{proposition}

\subsection{Proof of lemma \ref{lem null-commu}}\label{Appendix lem null-commu}
First, we need the following decomposition:
\begin{lemma}\label{lem1 commu}
	Let $u$ be a function defined in $\Kcal_{[s_0,s_1]}$, sufficiently regular. Then 
	\begin{equation}\label{eq1 lem1 commu}
	[\del^IL^J,L_a]u = \sum_{0\leq |I'|\leq|I|\atop 1\leq |J'|\leq |J|}t^{|I'|-|I|}\Lambda^{IJ}_{aI'J'}\del^{I'}L^{J'}u + \sum_{|I'|=|I|}\Gamma^{I}_{aI'}\del^{I'}L^J u.
	\end{equation}
	where $\Lambda^{IJ}_{aI'J'}$ are homogeneous of degree zero and $\Gamma^{I}_{aI'}$ are constants. Furthermore:
	\begin{equation}\label{eq2 lem1 commu}
	\big|[\del^IL^J,\del_{\alpha}L_b]u\big| \leq  C\sum_{\beta,0\leq|I'|\leq|I|\atop 0\leq |J'|\leq|J|}|\del_{\beta}\del^{I'}L^{J'}u|
	\end{equation}
	and
	\begin{equation}\label{eq3 lem1 commu}
	\big|[\del^IL^J,\del_{\alpha}]u\big| \leq  C\sum_{\beta,0\leq |J'|<|J|}|\del_{\beta}\del^IL^{J'}u|
	\end{equation}
	where $C$ is determined by $(I,J)$.
\end{lemma}
\begin{proof}[Proof of lemma \ref{lem1 commu}]
This is an induction on $(I,J)$. We first remark that
	$$
	[L_a,L_b] = (x^a/t)L_b - (x^b/t)L_a
	$$
	We denote this by 
	\begin{equation}\label{eq1 pr lem1 commu}
	[L_a,L_b] = \lambda_{ab}^cL_c
	\end{equation}
	where $\lambda_{ab}^c$ is homogeneous of degree zero. 
	
	Then we establish the following decomposition:
	\begin{equation}\label{eq2 pr lem1 commu}
	[L^J,L_a] = \sum_{1\leq |J'|\leq|J|}\Lambda_{aJ'}^JL^{J'}.
	\end{equation}
	This is by induction on $|J|$. When $|J|=1$ this is guaranteed by \eqref{eq1 pr lem1 commu}. Then we remark the following calculation:
	$$
	\aligned
	\,[L^JL_a,L_b]u =& L^J([L_a,L_b]u) + [L^J,L_b]L_au
	\\
	=&L^J(\lambda_{ab}^cL_cu) + \sum_{1\leq |J'|\leq|J|}\Lambda_{bJ'}^JL^{J'}L_au
	\\
	=&\sum_{J_1+J_2=J}L^{J_1}\lambda_{ab}^c\ L^{J_2}L_cu + \sum_{1\leq |J'|\leq|J|}\Lambda_{bJ'}^JL^{J'}L_au.
	\endaligned
	$$
	Remark that $L^{J_1}\lambda_{ab}^c$ and $\Lambda_{bJ'}^J$ are homogeneous of degree zero. Then the above calculation proves the desired result for $|J|+1$ case.
	
	Then we consider $[\del^IL^J,L_a]$.
	$$
	\aligned
	\,[\del^IL^J,L_a]u =& \del^I([L^J,L_a]u) + [\del^I,L_a]L^Ju
	\\
	=&\sum_{1\leq|J'|\leq|J|}\del^I(\Lambda_{aJ'}^JL^{J'}u) + \sum_{ |I'|= |I|}\Gamma_{aI'}^I\del^{I'}L^{J}u
	\\
	=&\sum_{1\leq|J'|\leq|J|\atop I_1+I_2=I}\del^{I_1}\Lambda_{aJ'}^J\ \del^{I_2}L^{J'}u + \sum_{ |I'|= |I|}\Gamma_{aI'}^I\del^{I'}L^{J}u
	\endaligned
	$$
	where for the second line we have applied \eqref{eq2 pr lem1 commu} and \eqref{eq 2 comm}. Now remark that $\del^{I_1}\Lambda_{aJ'}^J$ is homogeneous of degree $-|I_1| = |I_2|-|I|$ and $\Gamma_{aI'}^I$ are constants. Then \eqref{eq1 lem1 commu} is established with coefficients of linear combination determined by $(I,J)$.
	
	Now let us consider \eqref{eq2 lem1 commu}. Recall \eqref{eq 2 comm}
	$$
	\aligned
	\,&[\del^IL^J,\del_{\alpha}L_b]u
	\\
	=& [\del^IL^J,\del_{\alpha}]L_b u + \del_{\alpha}([\del^IL^J,L_b]u)
	\\
	=&\sum_{\beta, 0\leq |J'|<|J|}\Gamma^{J\beta}_{\alpha J'}\del_{\beta}\del^IL^{J'}L_b u
	+ \sum_{0\leq |I'|\leq|I|\atop 1\leq |J'|\leq |J|}\del_{\alpha}\big(t^{|I'|-|I|}\Lambda^{IJ}_{bI'J'}\del^{I'}L^{J'}u\big) 
	+ \sum_{|I'|=|I|}\Gamma^{I}_{bI'}\del_{\alpha}\del^{I'}L^J u
	\\
	=&\sum_{\beta, 0\leq |J'|<|J|}\uline{\Gamma^{J\beta}_{\alpha J'}}\del_{\beta}\del^IL^{J'}L_b u  + \sum_{|I'|=|I|}\uline{\Gamma^{I}_{bI'}}\del_{\alpha}\del^{I'}L^J u
	+\sum_{0\leq |I'|\leq|I|\atop 1\leq |J'|\leq |J|}t^{|I'|-|I|}\uline{\Lambda^{IJ}_{bI'J'}}\del_{\alpha}\del^{I'}L^{J'}u
	\\
	&+\sum_{1\leq |I'|\leq|I|\atop 1\leq |J'|\leq |J|}t^{|I'|-|I|-1}\ \uline{\big(t^{|I| - |I'|+1}\del_{\alpha}\big(t^{|I'|-|I|}\Lambda^{IJ}_{bI'J'})\big)}\del^{I'}L^{J'}u
	\\
	&+\sum_{1\leq |J'|\leq |J|}t^{-|I|-1}\ \uline{\big(t^{|I|+1}\del_{\alpha}\big(t^{-|I|}\Lambda^{IJ}_{bOJ'})\big)}\ t^{-1}L^{J'}u.
	\endaligned
	$$
	Remark that in right-hand-side of the above expression, the \uline{underlined coefficients} are homogeneous of degree zero. Furthermore, for the forth term, since $|I'|\geq 1$, we write
	$$
	\del^{I'}L^{J'}u = \del_{\beta}\del^{I''}L^{J'}u.
	$$
	For the last term, since $|J'|\geq 1$, we write 
	$$
	t^{-1}L^{J'}u = \delu_c L^{J''}u.
	$$
	So \eqref{eq2 lem1 commu} is established.

	\eqref{eq3 lem1 commu} is direct by \eqref{eq 2 comm}, we omit the detail.
\end{proof}
	
Now we are ready to prove lemma \ref{lem null-commu}

\begin{proof}[Proof of lemma \ref{lem null-commu}]
	Recall the decomposition of $\Hu(\del\del,\del)u$ in \eqref{eq 1 28-06-2019}.
	First,  we observe that $T_1[H,u]$ is a finite linear combination of $t^{-1}H^{\alpha\beta}\del_{\gamma}L_bu$ with homogeneous coefficients of degree zero (the elements of transition matrices are homogeneous of degree zero). Let $\Lambda$ be homogeneous of degree zero, then
	$$
	\aligned
	\,&[\del^IL^J, t^{-1}\Lambda H^{\alpha\beta}\del_{\gamma}L_b]u
	\\
	=&\sum_{|I_1|+|J_1|\geq 1,I_1+I_2+I_3=I\atop J_2+J_2+J_3=J}\!\!\!\!\!\del^{I_3}L^{J_3}(t^{-1}\Lambda)\ \del^{I_2}L^{J_2}H^{\alpha\beta}\ \del^{I_1}L^{J_1}\del_{\gamma} L_bu + t^{-1}\Lambda H^{\alpha\beta}[\del^IL^J,\del_{\gamma}L_b]u
	\endaligned
	$$
	For the first term we apply \eqref{eq1 lem 1 notation}:
	$$
	|\del^{I_1}L^{J_1}(t^{-1}\Lambda)\ \del^{I_2}L^{J_2}H^{\alpha\beta}\ \del^{I_3}L^{J_3}\del_{\gamma}L_bu|\leq Ct^{-1}|H|_{p_2,k_2}|\del u|_{p_1+1,k_1+1}
	$$
	where $p_1=|I_1|+|J_1|, k_1=|J_1|$, $p_2=|I_2|+|J_2|, k_2=|J_2|$.
	For the term $t^{-1}\Lambda H^{\alpha\beta} [\del^IL^J,\del_{\gamma}L_b]u$, we apply \eqref{eq2 lem1 commu} combined with \eqref{eq1 lem 1 notation}:
	$$
	|t^{-1}\Lambda H[\del^IL^J,\del_{\gamma}L_b]u|\leq Ct^{-1}|H||\del u|_{p,k}.
	$$
	where $p = |I|+|J|, k=|J|$. So we conclude that
	$$
	|[\del^IL^J, t^{-1}\Lambda H^{\alpha\beta}\del_{\gamma}L_b]u|\leq Ct^{-1}\sum_{p_1+p_2=p,p_1<p\atop k_1+k_2=k}|\del u|_{p_1+1,k_1+1}|H|_{p_2,k_2} + Ct^{-1}|H||\del u|_{p,k}
	$$
	and this leads to the bound of $T_1$. 
	
	The bound on $T_2$ can be established in the same manner (thanks to \eqref{eq3 lem1 commu}), we omit the detail. 
\end{proof}
\end{appendix}

\bibliographystyle{elsarticle-num}
\bibliography{WKGm-bibtex}

\begin{thebibliography}{10}
\expandafter\ifx\csname url\endcsname\relax
  \def\url#1{\texttt{#1}}\fi
\expandafter\ifx\csname urlprefix\endcsname\relax\def\urlprefix{URL }\fi
\expandafter\ifx\csname href\endcsname\relax
  \def\href#1#2{#2} \def\path#1{#1}\fi

\bibitem{M5}
Y.~Ma, Global solutions of nonlinear wave-{K}lein-gordon system in two spatial
  dimensions: strong coupling case, In preparation.

\bibitem{Hu2018}
C.~Huneau, Stability of minkowski space-time with a translation space-like
  killing field, C. Ann. PDE 4~(1) (2018) 12.
\newblock \href {https://doi.org/10.1007/s40818-018-0048-x}
  {\path{doi:10.1007/s40818-018-0048-x}}.

\bibitem{LM2}
P.~LeFloch, Y.~Ma, The nonlinear stability of {M}inkowski space for
  self-gravitating massive field. {T}he wave-{K}lein-{G}ordon model, Commun.
  Math. Phys. 346~(2)  603--665.
\newblock \href {https://doi.org/10.1007/s00220-015-2549-8}
  {\path{doi:10.1007/s00220-015-2549-8}}.

\bibitem{LM3}
P.~LeFloch, Y.~Ma, The global nonlinear stability of {M}inkowski space for
  self-gravitating massive fields, Worle Scientific, 2017.
\newblock \href {https://doi.org/10.1142/10730} {\path{doi:10.1142/10730}}.

\bibitem{Q.Wang2016}
Q.~Wang, An intrinsic hyperboloid approach for {E}instein {K}lein-{G}ordon
  equations, arXiv:math.AP/1607.01466.

\bibitem{Ionescu-Pausader}
A.~Ionescu, B.~Pausader, On the global regularity for a {W}ave-{K}lein-{G}ordon
  coupled system, arXiv:1703.02846v1.

\bibitem{M2}
Y.~Ma, Global solutions of quasilinear wave-{K}lein-{G}ordon system in two
  space dimension: completion of the proof, J. Hyperbol. Differ. Eq. 14~(4)
  627--670.
\newblock \href {https://doi.org/10.1142/S0219891617500217}
  {\path{doi:10.1142/S0219891617500217}}.

\bibitem{Stingo-2018}
A.~Stingo, Global existence of small amplitude solutions for a model quadratic
  quasi-linear coupled wave-{K}lein-{G}ordon system in two space dimension,
  with mildly decaying cauchy data, arXiv:1507.02035v1.

\bibitem{Kl1}
S.~Klainerman, Global existence for nonlinear wave equations, Commun. Pure
  Appl. Math. 33~(1) (1980) 43--101.
\newblock \href {https://doi.org/10.1002/cpa.3160330104}
  {\path{doi:10.1002/cpa.3160330104}}.

\bibitem{Christodoulou-1986}
D.~Christodoulou, Global solutions to non linear wave equations for small
  initial data, Commun. Pure Appl. Math. 39~(2) (1986) 267--282.
\newblock \href {https://doi.org/10.1002/cpa.3160390205}
  {\path{doi:10.1002/cpa.3160390205}}.

\bibitem{Kl2}
S.~Klainerman, Global existence of small amplitude solutions to nonlinear
  {K}lein-{G}ordon equations in four-spacetime dimensions, Commun. Pure Appl.
  Math. 38~(1) (1985) 631--641.
\newblock \href {https://doi.org/10.1002/cpa.3160380512}
  {\path{doi:10.1002/cpa.3160380512}}.

\bibitem{Shatah85}
J.~Shatah, Normal forms and quadratic nonlinear {K}lein-{G}ordon equations,
  Comm. Pure Appl. Math. 38 (1985) 685--696.
\newblock \href {https://doi.org/10.1002/cpa.3160380516}
  {\path{doi:10.1002/cpa.3160380516}}.

\bibitem{LM1}
P.~LeFloch, Y.~Ma, The hyperboloidal foliation method, World Scientific, 2015.

\bibitem{A1}
S.~Alinhac, The null condition for quasilinear wave equations in two-space
  dimension, {II}, Am. J. Math. 123~(6) (2001) 1071--1101.
\newblock \href {https://doi.org/10.1353/ajm.2001.0037}
  {\path{doi:10.1353/ajm.2001.0037}}.

\bibitem{A2}
S.~Alinhac, The null condition for quasilinear wave equations in two-space
  dimension {I}, Invent. math. 145~(3) (2001) 597--618.
\newblock \href {https://doi.org/10.1007/s002220100165}
  {\path{doi:10.1007/s002220100165}}.

\bibitem{Hoshiga-2006}
A.~Hoshiga, The existence of global solutions to systems of quasilinear wave
  equations with quadratic nonlinearities in 2-dimensional space, Funkcial.
  Ekvac. 49~(3) (2006) 357--384.
\newblock \href {https://doi.org/10.1619/fesi.49.357}
  {\path{doi:10.1619/fesi.49.357}}.

\bibitem{Godin-1993}
P.~Godin, Lifespan of solutions of semilinear wave equations in two space
  dimensions, Comm. Partial Differential Equations 18~(5-6) (1993) 895--916.
\newblock \href {https://doi.org/10.1080/03605309308820955}
  {\path{doi:10.1080/03605309308820955}}.

\bibitem{Dfx}
J.-M. Delort, {D. Fang and R. Xue}, Global existence of small solutions for
  quadratic quasilinear {K}lein-{G}ordon systems in two space dimensions, J.
  Funct. Anal. 211~(2) (2004) 288--323.
\newblock \href {https://doi.org/10.1016/j.jfa.2004.01.008}
  {\path{doi:10.1016/j.jfa.2004.01.008}}.

\bibitem{KS-2011}
Y.~Kawahara, H.~Sunagawa, Global small amplitude solutions for two-dimensional
  nonlinear klein-gordon systems in the presence of mass resonance, J. Differ.
  Equations 251~(9) (2011) 2549--2567.
\newblock \href {https://doi.org/10.1016/j.jde.2011.04.001}
  {\path{doi:10.1016/j.jde.2011.04.001}}.

\bibitem{MH-2017}
Y.~Ma, H.~Huang, A conformal-type energy inequality on hyperboloids and its
  application to quasi-linear wave equation in $\mathbb{R}^{3+1}$,
  arXiv:1711.00498v1 [math.AP].

\bibitem{Wong-2017}
W.Wong, Small data global existence and decay for two dimensional wave maps,
  arXiv:1712.07684v1 [math.AP].

\bibitem{Ho1}
L.~H\"ormander, Lectures on nonlinear hyperbolic differential equations,
  Springer Verlag, 1997.

\bibitem{M1}
Y.~Ma, Global solutions of quasilinear wave-{K}lein-{G}ordon system in two
  space dimension: technical tools, J. Hyperbol. Differ. Eq. 14~(4) (2017)
  591--625.
\newblock \href {https://doi.org/10.1142/S0219891617500205}
  {\path{doi:10.1142/S0219891617500205}}.

\end{thebibliography}
\end{document}